\documentclass[10pt, reqno]{amsart}
\usepackage{amssymb,amsmath,amsthm,color, cite,subcaption}
\theoremstyle{plain}
\newtheorem{theorem}{Theorem}[section]
\newtheorem{conjecture}{Conjecture}[section]

\newtheorem{proposition}[theorem]{Proposition}
\newtheorem{definition}[theorem]{Definition}
\newtheorem{lemma}[theorem]{Lemma}
\newtheorem{corollary}[theorem]{Corollary}

\theoremstyle{remark}
\newtheorem{remark}[theorem]{Remark}
\usepackage{graphicx}

\usepackage%{hyperref}
[bookmarks=true,
   bookmarksnumbered=false,
   bookmarkstype=toc]
{hyperref}
\numberwithin{equation}{section}

\newcommand{\C}{\mathbb{C}}
\newcommand{\R}{\mathbb{R}}

\newcommand{\N}{\mathcal{N}}

\newcommand{\F}{\mathcal{F}}
\newcommand{\G}{\mathcal{G}}
\newcommand{\V}{\mathcal{V}}
\renewcommand{\Im}{\operatorname{Im}}
\renewcommand{\Re}{\operatorname{Re}}
\renewcommand{\L}{\mathcal{L}}

\newcommand{\norm}[1]{\|#1\|}
\newcommand{\Jbr}[1]{\left\langle #1 \right\rangle}
\def\d{{\partial}}
\def\({\left[}
\def\){\right]}
\newcommand{\I}{\infty}
\newcommand{\abs}[1]{\left\lvert #1\right\rvert}

\newcommand{\eps}{\varepsilon}

\DeclareMathOperator{\sech}{sech}

\DeclareMathOperator{\arctanh}{arctanh}
\newcommand{\qtq}[1]{\quad\text{#1}\quad}

\renewcommand{\H}{\mathcal{H}}
\newcommand{\rre}{\mathbb{R}}
\newcommand{\pt}{\partial}

\begin{document}

\title[Stability of solitary waves]{Asymptotic stability of solitary waves for the $1d$ NLS with an attractive delta potential}

\author[S. Masaki]{Satoshi Masaki}
\address{Department of Systems Innovation, Graduate school of Engineering Science, Osaka University}
\email{masaki@sigmath.es.osaka-u.ac.jp}

\author[J. Murphy]{Jason Murphy}
\address{Department of Mathematics \& Statistics, Missouri University of Science \& Technology}
\email{jason.murphy@mst.edu}

\author[J. Segata]{Jun-ichi Segata}
\address{Department of Mathematical Sciences, Kyushu University} 
\email{segata@math.kyushu-u.ac.jp}

\begin{abstract}
We consider the one-dimensional nonlinear Schr\"odinger equation with an attractive delta potential and mass-supercritical nonlinearity. This equation admits a one-parameter family of solitary wave solutions in both the focusing and defocusing cases.  We establish asymptotic stability for all solitary waves satisfying a suitable spectral condition, namely, that the linearized operator around the solitary wave has a two-dimensional generalized kernel and no other eigenvalues or resonances. In particular, we extend our previous result \cite{MMS} beyond the regime of small solitary waves and extend the results of \cite{FOO, KO} from orbital to asymptotic stability for a suitable family of solitary waves. 
\end{abstract}

\maketitle

\section{Introduction}
We consider the one-dimensional nonlinear Schr\"odinger equation (NLS) with an attractive delta potential.  This equation has the form
\begin{equation}\label{nls}
\begin{cases}
i\partial_t u = Hu + f(u), \\
u|_{t=0}=u_0.
\end{cases}
\end{equation}

Here $H=-\tfrac12\partial_x^2 + q\delta_0$, where $q<0$ and $\delta_0$ is the Dirac delta distribution (see Section~\ref{S:delta}).  The nonlinearity in \eqref{nls} is assumed to be of the form $f(u)=\sigma |u|^p u$, where $\sigma\in\{\pm 1\}$ and $p>4$ (the mass-supercritical regime).  The choice $\sigma=+1$ yields the defocusing equation, while $\sigma=-1$ gives the focusing equation.  This equation is related to a simple model describing the resonant nonlinear propagation of light through optical wave guides with localized defects \cite{GHW}.

Equation \eqref{nls} admits an explicit one-parameter family of solitary wave solutions of the form $u(t,x)=e^{i\omega t}Q_{\omega}(x)$, where
\begin{equation}\label{nls-Q}
HQ_\omega+\omega Q_\omega + f(Q_\omega) = 0
\end{equation}
(see Section~\ref{S:solitons}).  Here the frequency $\omega$ is restricted to $(0,\tfrac12q^2)$ in the defocusing case and to $(\tfrac12q^2,\infty)$ in the focusing case, with $\tfrac12q^2$ corresponding to the eigenvalue for the underlying linear operator.  The stability properties of these solitary waves have been studied in previous works \cite{FOO, KO, TX, MMS, DP}.  In the mass-supercritical case $p>4$ with defocusing nonlinearity, all solitary waves are orbitally stable.  In the focusing case, there exists $\Omega\in(\tfrac12q^2,\infty)$ such that one has stability for $\omega\in(\tfrac12q^2,\Omega)$ and instability for $\omega\in[\Omega,\infty)$.  The frequency $\Omega$ arises as the unique frequency at which $\tfrac{d}{d\omega}\|Q_{\omega}\|_{L^2}^2=0$.  

In our previous work \cite{MMS}, we showed that one has \emph{asymptotic} stability for all solitary waves with frequency $\omega$ sufficiently close to $\tfrac12q^2$.   In this regime, the solitary waves arise as small multiples of the eigenfunction for the underlying linear equation, and the fact that the solitary waves are small plays a key role in the proof.%of asymptotic stability.  %For more details, see Section~\ref{S:solitons}. 

In this work, our main goal is to establish asymptotic stability beyond the regime of small solitary waves for the model \eqref{nls}.  We consider initial data of the form
\begin{equation}\label{initial-data}
u_0 = Q_{\omega_0} + \tilde v_0,
\end{equation}
where $Q_{\omega_0}$ is an orbitally stable solitary wave and $\tilde v_0$ is small in $H^1$ and in a weighted $L^2$ space.  Our main result (Theorem~\ref{T} below) is an asymptotic stability result for suitable choices of $\omega_0$.  In particular, we will show that if a certain spectral condition holds at the frequency $\omega_0$ (see Definition~\ref{D:spectral} below), then the solution to \eqref{nls} with initial data $u_0$ as in \eqref{initial-data} scatters to a nearby solitary wave as $t\to\infty$.  Before stating the theorem precisely, let us briefly introduce the general approach and explain the origin of the spectral condition that plays a key role in the argument.

Orbital stability (established in \cite{FOO, KO}) guarantees that for initial data of the form \eqref{initial-data}, the corresponding solution admits a decomposition of the form 
\begin{equation}\label{intro-decomposition1}
u(t) = e^{i\beta(t)}\bigl\{Q_{\omega_0} + \tilde v(t)\bigr\}
\end{equation}
for some phase $\beta(t)$, where $\tilde v(t)$ is small in $H^1$ for all $t\geq 0$.  One might first try to establish asymptotic stability by proving that in fact, $\tilde v(t)$ decays/scatters as $t\to\infty$.  To study the behavior of $\tilde v(t)$, one can first use \eqref{nls} and \eqref{nls-Q} to derive a nonlinear evolution equation for $\tilde v$.  The hope would then be to use this equation together with the smallness assumption on $\tilde v_0$ to run a bootstrap argument that yields decay estimates for $\tilde v$.  Such an argument would require estimates for the linear part of the evolution equation, which involves the matrix operator
\begin{equation}\label{Intro-L}
\L(\omega_0)=\left[\begin{array}{cc} 0 & H+\omega_0 \\ -H-\omega_0 & 0 \end{array}\right]+\sigma Q_{\omega_0}^p\left[\begin{array}{cc} 0 & 1 \\ -(p+1) & 0\end{array}\right]
\end{equation} 
arising from the linearization around the solitary wave. 

At this point, the spectral properties of the operator $\L(\omega_0)$ come into play.  Using \eqref{nls-Q}, we can immediately find two independent vectors in the generalized kernel of $\L(\omega_0)$ (namely, $[0\ Q_{\omega_0}]^t$ and $[\partial_\omega Q_{\omega}|_{\omega=\omega_0}\ 0]^t$), which implies the existence of a two-dimensional subspace on which the linear evolution $e^{t\L(\omega_0)}$ does \emph{not} decay.  Moreover, it is possible for $\L(\omega_0)$ to have additional purely imaginary eigenvalues, even when $Q_{\omega_0}$ is orbitally stable.  Thus we see that to obtain decay estimates, it would at least be necessary to restrict to the continuous spectral subspace of $\L(\omega_0)$. On the other hand, there is no reason \emph{a priori} that the perturbation $\tilde v(t)$ above should belong to this subspace; indeed, the decomposition \eqref{intro-decomposition1} incorporates only a single modulation parameter.  Thus it is not clear whether or not the component $\tilde v(t)$ actually \emph{decays} as $t\to\infty$.

In light of the above observations, we modify the approach as follows.  We firstly restrict attention to frequencies $\omega_0$ so that the dimension of the generalized kernel of $\L(\omega_0)$ is exactly two and such that $\L(\omega_0)$ has no other eigenvalues.  This is part of the spectral condition appearing below. Next, instead of \eqref{intro-decomposition1}, we look for a decomposition of the form
\[
u(t) = e^{i\Phi(t)}\{Q_{\omega(t)}+v(t)\},
\]
where now both the phase $\Phi(t)$ and the frequency $\omega(t)$ may vary in time.  In particular, we will choose these two modulation parameters to guarantee that $v(t)$ belongs to the continuous spectral subspace of $\L(\omega(t))$ for each $t\geq 0$.  This requires that we impose suitable orthogonality conditions (see \eqref{T-orthogonality} below), which we can accomplish by making use of the existing orbital stability results and the implicit function theorem (see Proposition~\ref{P:decomposition}).  We then essentially carry out the strategy outlined above, namely, we use dispersive estimates for the evolution group and a bootstrap argument to derive decay estimates for the perturbation $v(t)$.  We must also derive suitable estimates for the modulation parameters themselves.

Before stating the main result, we must discuss one additional technical issue that is also connected to the spectral condition stated below.  In the bootstrap estimates for both $v$ and the modulation parameters, we must exhibit enough decay to integrate in time.  The lowest order term in the modulation parameters will be quadratic in $v$, while the equation for $v$ will still ultimately contain some linear terms (due to a change of variables that serves to remove the time-dependence from the underlying linearized operator).  This means that we must do better than the typical $|t|^{-\frac12}$ decay rate associated to dispersive estimates for one-dimensional Schr\"odinger-type equations.  By utilizing weighted norms, we can obtain an estimate with an integrable decay rate.  However, this requires that the linear operator $\L$ obeys a further spectral condition, namely, the absence of a resonance at the edge of the continuous spectrum. 

We may now state the main spectral assumption and our main result.  We recall the definition of $\Omega$ given above (see also Section~\ref{S:solitons}).

\begin{definition}[Spectral condition]\label{D:spectral} Let $\omega\in(0,\tfrac12q^2)$ in the defocusing case and $\omega\in(\tfrac12q^2,\Omega)$ in the focusing case. We say that $\omega$ satisfies the spectral condition if $\L(\omega)$ has a generalized kernel of dimension $2$, no other gap or embedded eigenvalues, and no resonances.
\end{definition}

\begin{theorem}[Asymptotic stability and scattering]\label{T} Fix $q<0$, $p>4$, and $1<\alpha<\min\{\frac{p}{4},\frac32\}$.  

Let $\omega_0\in (0,\tfrac12 q^2)$ in the defocusing case and $\omega_0\in(\tfrac12q^2,\Omega)$ in the focusing case.  Suppose that the spectral condition in Definition~\ref{D:spectral} holds in a neighborhood of $\omega_0$.

For $\eps=\eps(q,\omega_0,p)>0$ and $\eta=\eta(\eps)>0$ sufficiently small, we have that if
\[
u_0 = Q_{\omega_0}+\tilde v_0\qtq{with} \|\tilde v_0\|_{H^1} + \|\langle x\rangle^\alpha \tilde v_0\|_{L^2} <\eta,
\]
then the corresponding global-in-time solution $u$ to \eqref{nls} admits a decomposition of the form
\[
u(t) = e^{i\Phi(t)}\{Q_{\omega(t)}+v(t)\}
\]
obeying the following properties:

The phase has the form
\[
\Phi(t)=\theta(t)+\int_0^t \omega(s)\,ds,
\]
and the modulation parameters obey
\[
|\dot\theta(t)|+|\dot\omega(t)| \lesssim \eps^2 \langle t\rangle^{-2\alpha} \qtq{for all}t\geq 0,
\]
where here and below the implicit constants depend on $p$ and $\omega_0$.  In particular, $\omega(t)$ converges to some $\omega_+$ as $t\to\infty$, with $|\omega_0-\omega_+|\lesssim\eps^2$.

The perturbation $v(t)$ obeys the orthogonality conditions
\begin{equation}\label{T-orthogonality}
\Re\langle v(t),Q_{\omega(t)}\rangle = \Re\langle v(t),i\partial_\omega Q_{\omega(t)}\rangle = 0\qtq{for all}t\geq 0,
\end{equation}
along with the estimates
\[
\|v(t)\|_{H^1}\lesssim \eps,\quad \|v(t)\|_{L^r} \lesssim \eps\langle t\rangle^{-[\frac12-\frac1r]},\qtq{and} \|\langle x\rangle^{-\alpha}v(t)\|_{L^2}\lesssim \eps \langle t\rangle^{-\alpha}
\]
for all $t\geq 0$, where $r=r(\alpha,p)$ is a sufficiently large (but finite) exponent.  Moreover, scattering holds in the sense that 
\[
\lim_{t\to\infty} \|v(t) - e^{t\L(\omega_+)}v_+\|_{L^2} = 0
\]
for some $v_+\in L^2$.
\end{theorem}

\begin{remark} The notation $\partial_\omega Q_{\omega(t)}$ appearing in \eqref{T-orthogonality} is short-hand for the expression
\[
\partial_\omega Q_{\omega}\bigr|_{\omega=\omega(t)}.
\] 
The shorter notation appearing in \eqref{T-orthogonality} will be used throughout the paper. 
\end{remark}

Theorem~\ref{T} is conditional on the assumption that $\omega_0$ satisfies the spectral condition in Definition~\ref{D:spectral}.  In this paper, we also investigate the range of frequencies for which the spectral condition holds.  Combining some analytic and numerical results in Section~\ref{S:spectral}, we arrive at the following conjecture.  Recall that $\Omega=\Omega(q,p)$ is the minimal unstable frequency introduced above (see also Section~\ref{S:solitons} below).

\begin{conjecture}\label{conjecture} Fix $q<0$ and $p>4$. 
\begin{itemize} 
\item In the focusing case ($\sigma=-1$), there exists $\omega_1=\omega_1(q,p)\in(\tfrac12q^2,\Omega)$ such that the spectral condition holds for $\omega\in(\tfrac12 q^2,\omega_1)$.
\item In the defocusing case ($\sigma=1$), the spectral condition holds for all $\omega\in(0,\tfrac12q^2)$.
\end{itemize} 
\end{conjecture} 
In the focusing case, the frequency $\omega_1$ is identified by the first appearance of a resonance at the edge of the continuous spectrum of $\mathcal{L}(\omega)$.  By using a shooting method in $\omega$, we can numerically estimate the value of $\omega_1$.  This frequency is always smaller than $\Omega$, and hence Theorem~\ref{T} does not treat the entire range of stable solitary waves in the focusing case.  On the other hand, by computing the mass of $Q_{\omega_1}$ (representing the largest solitary wave mass addressed by Theorem~\ref{T} in the focusing case), we can see that our result covers a range of solitary waves beyond the small data regime.  Fixing $q=-1$, we display some computed values of $\omega_1=\omega_1(p)$ and the mass $M$ of $Q_{\omega_1}$ in the following table (see also Figure~\ref{special-omegas1} on page~\pageref{special-omegas1}). 

{\footnotesize
\begin{table}[h]
\begin{tabular}{|c||c|c|c|c|c|c|c|c|c|c|c|c|c|c|c|}
\hline
$p$ & 4.2 & 4.4 & 4.6 & 4.8 & 5 & 5.2 & 5.4 & 5.6 & 5.8 & 6 & 6.2 \\
\hline 
$\omega_1$ & 2.278 & 1.996 & 1.785 & 1.621 & 1.482 & 1.387 & 1.301 & 1.229 & 1.168 & 1.116 & 1.072    \\
\hline 
$M$ & 1.286 & 1.218 & 1.165 & 1.123 & 1.089 & 1.061 & 1.038 & 1.019 & 1.003 & .989 & .976  \\
\hline
\end{tabular}
\end{table}\label{the-table}}

As $p\to\infty$, the minimal unstable frequency $\Omega$ itself converges to $\tfrac12q^2$ (see e.g. \eqref{Omegaeqn} below).  Consequently, the entire stable regime is restricted to the setting of small solitary waves, and hence Theorem~\ref{T} necessarily collapses to the small data regime as $p\to\infty$ in the focusing case.  In the defocusing case, Conjecture~\ref{conjecture} and Theorem~\ref{T} yield asymptotic stability for the entire family of solitary waves. 

In this paper, we do not give a rigorous proof of Conjecture~\ref{conjecture}.  Instead, we focus our attention on the problems of (i) establishing asymptotic stability under the assumed spectral condition and (ii) providing an estimate for the range of $\omega$ for which this condition holds.  We carry out (ii) in Section~\ref{S:spectral}.  In that section, we also state some further conjectures about the global spectral picture for the operators $\L(\omega)$ (see Sections~\ref{S:focusing-spectrum}~and~\ref{S:defocusing-spectrum}).  The rest of the paper is then focused on establishing (i).  A more complete analysis of the spectral properties of the operators $\L(\omega)$, as well the extension of Theorem~\ref{T} to handle different spectral conditions in the focusing case, will be addressed in future works.

Theorem~\ref{T} extends our previous work \cite{MMS} beyond the regime of small solitary waves.  In \cite{MMS} we worked with initial frequencies $\omega_0$ very close to $\tfrac12q^2$, in which case the solitary waves look like small multiples of the eigenfunction of the underlying linear problem. In that work, we were able to take the initial perturbation to be small merely in $H^1$ (i.e. without any weighted assumption), and we proved global space-time bounds and scattering for the corresponding perturbation $v$.  In this small-data regime, we could view the evolution equation for $v$ as a perturbation of the linear Schr\"odinger equation with delta potential.  In particular, we could include terms of the form $\mathcal{O}(vQ_\omega^p)$ together with the nonlinear terms when proving bootstrap estimates for $v$.  Consequently, we could rely purely on Strichartz-type estimates for the underlying linear equation, which ultimately allowed us to close the argument without introducing any weighted spaces.  

%In \cite{MMS}, the perturbation $v$ was constructed to obey the same orthogonality conditions as those appearing in \eqref{T-orthogonality}.  We relied crucially on the fact that this choice of orthogonality conditions makes the evolution equation for the modulation parameters at least quadratic in $v$.  This was necessary for proving $L_t^1$-type estimates, as $v$ could not be expected to obey any space-time estimates with time integrability exponents below $L_t^2$.  In addition, because we worked in the regime of small solitary waves, the fact that $v$ belonged to the continuous spectral subspace for the operator corresponding to the (moving) solitary wave also guaranteed that $v$ was near enough to the continuous spectral subspace of the operator $H=-\tfrac12 \partial_x^2+q\delta_0$ that we could utilize the Strichartz-type estimates adapted to this operator.

In the present setting, the solitary waves $Q_{\omega}$ are no longer assumed to be small, and we must include the $\mathcal{O}(vQ_{\omega}^p)$ terms in the linear part of the equation for $v$.  In particular, a central component of our argument is the proof of dispersive estimates for the matrix operators $e^{t\L(\omega)}$.  To close the boostrap argument for $v$ and the modulation parameters $(\theta,\omega)$, we then rely on the fact that the equations are at least quadratic in $v$.  For the modulation parameters, this is a consequence of the choice of orthogonality conditions (similar to \cite{MMS}).  As for the equation for $v$, one has terms that are quadratic or higher in $v$, as well as some terms involving the derivatives of the modulation parameters.  However, after a change of variables used to remove the time-dependence from the operator in the underlying linear equation, we end up with additional terms that are \emph{linear} in $v$ (see Proposition~\ref{P:COV}).  While these terms come with additional sources of smallness (via the modulation parameters), their time-decay is no better than that of $v$ itself.  As a result, we can close the bootstrap estimate only if $v$ itself obeys some integrable-in-time decay estimate.  In particular, we must prove estimates that exhibit stronger decay than the usual $|t|^{-\frac12}$ decay rate associated to Schr\"odinger-type dispersive estimates in one dimension.  Utilizing such dispersive estimates requires that the initial perturbation belongs to weighted spaces, rather than merely $H^1$.  

We are able to obtain weighted dispersive estimates in the spirit of \cite{KS,  GS2} that exhibit a decay rate of $|t|^{-\frac32}$. This type of dispersive estimate requires the non-resonance condition appearing in Definition~\ref{D:spectral}.  Indeed, such estimates fail for the free Schr\"odinger equation due to the resonance at zero energy.  In practice, we must use suitably interpolated versions of the weighted dispersive estimates in order to treat the entire mass-supercritical regime $p>4$.  Indeed, there is a direct relationship between the rate of decay and the number of weights used.  When estimating a weighted norm of the purely nonlinear term $|v|^p v$, one must balance the decay exhibited by $v$ itself with the growth exhibited by terms like $\|xv\|_{L^2}$.  As a result, the situation becomes increasingly subtle as the value of $p$ decreases. In fact, although we expect that analogues of Theorem~\ref{T} should hold in the mass-critical and mass-subcritical regime, our approach ultimately breaks down precisely at $p=4$ (see e.g. Lemma~\ref{L:controlF}). 

%As described above, we must prove that the perturbation $v$ obeys an integrable-in-time decay estimate. This requires that the weight $\alpha$ appearing in the statement of Theorem~\ref{T} is strictly greater than $1$.  On the other hand, to close the bootstrap argument requires that we estimate the contribution of the purely nonlinear term $|v|^pv$ in a weighted space.  Using the virial estimate, we can control the growth of $\langle x\rangle v$ in $L^2$, while we have dispersive decay estimates for $v$ in higher order Lebesgue spaces.  We found that one has enough decay to close the estimates provided $\alpha$ is not too large.  The resulting constraint is essentially due to scaling and works out to $\alpha<\frac{p}{4}$.  Combining the two constraints on $\alpha$, we find ourselves precisely constrained to $p>4$, the mass-supercritical regime.  Treating the mass-critical and mass-subcritical range is an interesting and challenging problem; for related results in higher dimensions (where stronger unweighted dispersive estimates are available), we refer the reader to \cite{KM, KZ}.

Theorem~\ref{T} also fits in the context of the many recent works on asymptotic stability for solitary wave solutions to dispersive equations, for which there is now a huge literature.  For a sample of related results, we refer the interested reader to \cite{BP, BS, CC1, CC2, CC3, CC4, DP, GS, GNT, KM, KZ, KS, MMS, Miz, SW, SW2, SW3}.  Concerning asymptotic stability for the specific model of the $1d$ NLS with a delta potential, our previous work \cite{MMS} addressed the case of small solitary waves, while Deift--Park \cite{DP} gave a fairly exhaustive treatment of the case of a cubic nonlinearity and even data (see also \cite{DH, GHW, HMZ, HZ, HZ2} for some related results).  In the setting of \cite{DP}, the equation becomes completely integrable and hence is amenable to special techniques such as inverse scattering and nonlinear steepest descent.  Another related work is that of \cite{CM}, which utilized a virial-type argument to prove global space-time estimates in negative exponentially weighted spaces for the perturbation around small solitary waves for \eqref{nls}.  Finally, \cite{MN} showed the failure of (unmodified) scattering to solitary waves for \eqref{nls} with long-range nonlinearities (i.e. $p\leq 2$).

In terms of the techniques employed here, the works that are most closely-related are those of Buslaev--Perelman \cite{BP}, Gang--Sigal \cite{GS}, and Krieger--Schlag \cite{KS}.  The works \cite{BP,GS} concerned asymptotic stability for $1d$ Schr\"odinger equations (with or without potential) with suitable nonlinearities, while \cite{KS} addressed the problem of stable manifolds for the $1d$ NLS with mass-supercritical power-type nonlinearities.  A common theme in all of these works, as well as the present work, is the development of suitable dispersive estimates for the (matrix) operators arising from the linearization of the nonlinear equation around the moving solitary wave, coupled with a bootstrap argument to derive decay estimates for the perturbation and the modulation parameters.  

The rest of this paper is organized as follows:
\begin{itemize}
\item In Section~\ref{S:notation}, we first collect some notation.  We formally introduce the Schr\"odinger operator $H=-\tfrac12\partial_x^2 + q\delta_0$ in Section~\ref{S:delta}, and in Section~\ref{S:solitons} we provide a more detailed introduction to the solitary wave solutions $Q_\omega$.  
\item Section~\ref{S:spectral} contains the spectral analysis for the operator $\L(\omega)$.  In particular, we provide some numerical evidence for the claims made in Conjecture~\ref{conjecture}.  We also give some further conjectures related to the global spectral picture for the operators $\L(\omega)$. 
\item In Section~\ref{S:dispersive}, we prove dispersive estimates for the evolution group $e^{t\L(\omega)}$ under the assumption that $\omega$ obeys the spectral condition Definition~\ref{D:spectral}.  The argument requires the construction of some particular solutions to an associated eigenvalue problem, which we carry out in Appendix~\ref{S:F1F4}.  
\item Finally, in Section~\ref{S:stability}, we prove the main result, Theorem~\ref{T}. 
\end{itemize}

\subsection*{Acknowledgements} S.M. was supported by JSPS KAKENHI Grant Numbers JP17K14219, JP17H02854, JP17H02851, and JP18KK0386. J.M. was supported in part by a Simons Collaboration Grant. J.S was supported by JSPS KAKENHI Grant Numbers JP17H02851, JP19H05597, and JP20H00118.  We are grateful to Rowan Killip and Maciej Zworski for helpful suggestions concerning numerical approaches for the spectral analysis.

%%%%%%%%%%%%%%%%%%%%%
%%%%%%%%%%%%%%%%%%%%%
%%%%%%%%%%%%%%%%%%%%%
%%%%%%%%%%%%%%%%%%%%%
%%%%%%%%%%%%%%%%%%%%%
%%%%%%%%%%%%%%%%%%%%%
\section{Notation and Preliminaries}\label{S:notation}

We write $A\lesssim B$ to denote the inequality $A\leq CB$ for some constant $C>0$.  We denote the dependence of the implicit constants on various parameters via subscripts, e.g. $A\lesssim_p B$ denotes $A\leq CB$ for some $C=C(p)>0$.  We will also make use of the big-oh notation $\mathcal{O}$, again indicating dependence on various parameters via subscripts.  We use the Japanese bracket notation $\langle x\rangle=\sqrt{1+x^2}$. We denote the transpose of a matrix or vector using the notation ${}^t$. 

We utilize the standard Lebesgue and Sobolev spaces.  In addition, we will make use of Lorentz spaces, defined via the quasi-norms
\[
\|f\|_{L^{a,b}} = \bigl\| \lambda\bigl| \{x:|f(x)>\lambda\}\bigr|^{\frac{1}{a}}\bigr\|_{L^b((0,\infty),\frac{d\lambda}{\lambda})}
\]
for $1\leq a<\infty$ and $1\leq b \leq\infty$.  In particular, we have $L^{a,a}=L^a$, while $L^{a,\infty}$ corresponds to weak $L^a$.  In general, one has the embedding $L^{a,b}\hookrightarrow L^{a,b'}$ for $b<b'$.  Lorentz spaces are convenient in our setting essentially due to the fact that pure powers belong to to Lorentz spaces, namely, $|x|^{-a}\in L^{\frac{d}{a},\infty}(\R^d)$ for $0<a\leq d$. 

Many standard inequalities have Lorentz-space analogues (see e.g. \cite{Hunt, ONeil}).  For example, we will make use of H\"older's inequality in the form
\[
\|fg\|_{L^{a,b}}\lesssim \|f\|_{L^{a_1,b_1}}\|g\|_{L^{a_2,b_2}}
\]
for $1\leq a,a_1,a_2<\infty$ and $1\leq q,q_1,q_2\leq\infty$ satisfying $\tfrac{1}{a}=\tfrac{1}{a_1}+\tfrac{1}{a_2}$ and $\tfrac{1}{b}=\tfrac{1}{b_1}+\tfrac{1}{b_2}$. 

We use the following notation for inner product:
\[
\langle f,g\rangle = \Re\int \bar f(x) g(x)\,dx.
\]
We will regularly identify a complex-valued function $f=f_1+if_2$ with the vector $[f_1\ f_2]^t$ taking values in $\R^2$.  Our definition of inner product corresponds to the standard inner product on $\R^2$-valued functions in the sense that
 \[
\bigg\langle \left[\begin{array}{r} f_1 \\ f_2 \end{array}\right],\left[\begin{array}{r} g_1 \\ g_2 \end{array}\right]\bigg\rangle =  \int f_1(x)g_1(x)\,dx + \int f_2(x)g_2(x)\,dx = \langle f,g\rangle.
\]

Finally, we recall the following standard Pauli matrices: 
\begin{equation}\label{pauli}
\sigma_1=\left[\begin{array}{rr} 0 & 1 \\ 1 & 0 \end{array}\right], \quad \sigma_3 = \left[\begin{array}{rr} 1 & 0 \\ 0 & - 1\end{array}\right].
\end{equation}

%%%%%%%%%%%%%%%%%%%%%
%%%%%%%%%%%%%%%%%%%%%
\subsection{The linear Schr\"odinger equation with a delta potential}\label{S:delta}

The linear Schr\"odinger equation with a delta potential is a classical model in quantum mechanics, which is discussed extensively in \cite{AGHH}.  We consider the case of an attractive potential, that is, we let
\[
H=-\tfrac12\partial_x^2 + q\delta(x) \qtq{with} q<0. 
\]
More precisely, the operator $H$ is defined by $H=-\tfrac12\partial_x^2$ on its domain, given by
\begin{equation}\label{domain}
D(H)=\{u\in H^1(\R)\cap H^2(\R\backslash\{0\}):\partial_x u(0+)-\partial_x u(0-)=2qu(0)\},
\end{equation}
where $0\pm$ denotes limit from the right/left, and $H$ extends to a self-adjoint operator on $L^2$ with purely absolutely continuous spectrum equal to $[0,\infty)$.  We frequently refer to the condition appearing in $D(H)$ as the jump condition at $x=0$.  If $q>0$ (the repulsive case), then $H$ has no eigenvalues, while if $q<0$ (the attractive case, under consideration in this paper) then $H$ has a single negative eigenvalue $-\tfrac12q^2$ with a one-dimensional eigenspace spanned by  $\phi_0(x):=|q|^{\frac12} e^{-|q|\,|x|}$.

%%%%%%%%%%%%%%%%%%%%%
%%%%%%%%%%%%%%%%%%%%%
\subsection{Solitary wave solutions}\label{S:solitons}  Equation \eqref{nls} admits a family of solitary wave solutions to \eqref{nls} of the form $u(t,x)=e^{i\omega t}Q_\omega(x)$, where $Q_\omega$ solves \eqref{nls-Q}. In the defocusing case (i.e. with $f(u)=+|u|^p u$), the frequency $\omega$ takes values in $(0,\tfrac12q^2)$ and the solitary waves take the explicit form
\[
Q_\omega(x) = \bigl[\tfrac{(p+2)\omega}{2}\bigr]^{\frac{1}{p}}\text{sinh}^{-\frac{2}{p}}\bigl[ p\sqrt{\tfrac{\omega}{2}}|x|+\text{arctanh}(\tfrac{\sqrt{2\omega}}{|q|})\bigr].
\]  
For $\omega=0$, one also obtains the solution
\begin{equation}\label{Q0def}
Q_0(x) = \bigl[\tfrac{(p+2)|q|^2}{(p|q|\,|x|+2)^2}\bigr]^{\frac{1}{p}},
\end{equation}
which belongs to $L^2$ provided $p<4$.  As we are restricting our attention to $p>4$, we will not directly consider this case in the present work. 

In the focusing case (i.e. with $f(u)=-|u|^p u$), the frequency $\omega$ takes values in $(\tfrac12q^2,\infty)$ and the solitary waves take the explicit form
\[
Q_\omega(x) = \bigl[\tfrac{(p+2)\omega}{2}\bigr]^{\frac{1}{p}}\text{cosh}^{-\frac{2}{p}}\bigl[ p\sqrt{\tfrac{\omega}{2}}|x|+\text{arctanh}(\tfrac{|q|}{\sqrt{2\omega}})\bigr].
\]
To derive these explicit formulas, one views \eqref{nls-Q} as a one-dimensional ODE on each side of $x=0$ coupled with the jump condition appearing in \eqref{domain}.  See e.g. \cite{KO, FOO}.  

The stability properties of these solitary waves were studied in \cite{KO, FOO, TX}. In the defocusing case, all solitary waves are orbitally stable \cite{KO}. In the focusing case, the authors of \cite{FOO} utilized the general theory of \cite{Shatah, ShatahStrauss} to prove the existence of $\Omega$ such that $Q_\omega$ is stable for $\omega\in(\tfrac12q^2,\Omega)$ and unstable for $\omega\in(\Omega,\infty)$.  Instability for $\omega=\Omega$ was subsequently established in \cite{TX}.  Here $\Omega$ is
defined as the solution to 
\begin{equation}\label{Omegaeqn}
\tfrac{p-4}{p}\int_{\arctanh(\frac{|q|}{\sqrt{2\omega}})} \sech^{\frac{4}{p}}(x)\,dx = \tfrac{|q|}{\sqrt{2\omega}}\bigl[1-\tfrac{q^2}{2\omega}\bigr]^{\frac{2}{p}-1}
\end{equation}
and represents the unique frequency at which $\tfrac{d}{d\omega}\|Q_\omega\|_{L^2}^2=0$. In fact, we have
\[
\tfrac{d}{d\omega}\tfrac12\|Q_{\omega}\|_{L^2}^2 = \langle Q_{\omega},\partial_\omega Q_{\omega}\rangle >0\qtq{for}\tfrac12q^2 <\omega<\Omega
\]
in the focusing case, while $\langle Q_{\omega},\partial_\omega Q_{\omega}\rangle<0$ for all $\omega\in(0,\tfrac12q^2)$ in the defocusing case.  In particular, this inner product will be nonzero for all $\omega$ under consideration in our main result (Theorem~\ref{T}).

We also note that by using the explicit formulas for $Q_{\omega}$, we have that $\langle\partial_\omega\rangle^2 Q_{\omega}\in H^{1,k}$ for any $k\geq 0$.  Moreover, the bounds are uniform over compact sets of $\omega$. %any set of $\omega$ that is a positive distance from $\tfrac12 q^2$ and from $0$ in the defocusing case. 

%%%%%%%%%%%%%%%%%%%%%
%%%%%%%%%%%%%%%%%%%%%
%%%%%%%%%%%%%%%%%%%%%
%%%%%%%%%%%%%%%%%%%%%
%%%%%%%%%%%%%%%%%%%%%
%%%%%%%%%%%%%%%%%%%%%
\section{Spectral analysis}\label{S:spectral}

In this section we discuss the spectral properties of the family of operators $\L(\omega)$ arising from the linearization of \eqref{nls} around the nonlinear bounds states $Q_\omega$.  We recall
\begin{equation}\label{D:L}
\L=\L(\omega) := \left[\begin{array}{cc} 0 & H+\omega \\ -H-\omega & 0  \end{array}\right]+\sigma Q_\omega^p\left[\begin{array}{cc} 0 & 1 \\-(p+1) & 0 \end{array}\right],
\end{equation}
where $H=-\tfrac12\partial_x^2 + q\delta_0$ with $q<0$ (see Section~\ref{S:delta}) and $Q_\omega$ are the nonlinear bound states (see Section~\ref{S:solitons}).  In particular, we recall that in the defocusing case $\sigma=+1$, we have $\omega\in(0,\tfrac12q^2)$, while in the focusing case $\sigma=-1$ we have $\omega\in(\tfrac12q^2,\infty)$.  We also recall that $\langle Q_\omega,\partial_\omega Q_\omega\rangle=\tfrac12\tfrac{d}{d\omega}\|Q_\omega\|_{L^2}^2$ is non-zero except for the case $\omega=\Omega$ in the focusing case, which marks the transition between stable/unstable solitary waves (see \cite{FOO} and Section~\ref{S:solitons}). 

We view $\L(\omega)$ as acting on $\R^2$-valued functions, where we identify a complex-valued function $f=f_1+if_2$ with the vector $[f_1\ f_2]^t$.  We introduce the notation
\begin{align*}
L_- = H+\omega+\sigma Q_\omega^p, \quad L_+ = H+\omega+\sigma(p+1)Q_\omega^p,
\end{align*}
so that
\[
\L=\L(\omega)=\left[\begin{array}{cc} 0 & L_- \\ -L_+ & 0 \end{array}\right]. 
\]

We will present a blend of analytic results, numerical results, and conjectures related to the spectral properties of $\L(\omega)$.  Our main goals in this section are are (i) to give a qualitative description of the global spectral properties of $\L(\omega)$ and (ii) for fixed values of $q$ and $p$, estimate the range of frequencies $\omega$ for which the spectral condition appearing in Definition~\ref{D:spectral} holds.  By addressing item (ii), we can verify that our main result (Theorem~\ref{T}) extends beyond the regime of \emph{small} solitary waves, which is the primary ambition of this paper.  Indeed, we need only compute the masses corresponding to the range of admissible frequencies.  

The results in this section provide some evidence for Conjecture~\ref{conjecture} stated in the introduction.  In addition, we will state some further conjectures concerning the spectral properties of $\mathcal{L}(\omega)$.  We plan to undertake a more complete analysis in future work. 

We begin by collecting some preliminary results concerning the spectrum of $\L(\omega)$.  From \eqref{D:L} we see that the continuous spectrum $\sigma_c(\mathcal{L}(\omega))$ of $\L(\omega)$ equals $i[\omega,\infty)\cup i[-\omega,-\infty)$ for all $\omega$.

Next, using some known facts about the operators $L_\pm$ (see e.g. \cite{FJ, LeCoz}), we can establish the following results concerning the generalized kernel of $\L$.  In particular, we have a two-dimensional generalized kernel for almost every value of $\omega$.  The only exception is $\omega=\Omega$ in the focusing case, at which point this subspace becomes four-dimensional.  We refer the reader to \cite{HZ2} for a discussion related to taking powers of $\L$.

\begin{lemma}\label{L:genker} Let $\omega\in(0,\tfrac12q^2)$ in the defocusing case, or $\omega\in(\tfrac12q^2,\infty)\backslash\{\Omega\}$ in the focusing case, so that
\begin{equation}\label{IPNZ}
\langle Q_\omega,\partial_\omega Q_\omega\rangle \neq 0.
\end{equation}
Then
\[
\ker \L^k = \text{span}\{\partial_\omega Q_\omega, i Q_\omega\}\qtq{for all}k\geq 2.
\]

In the focusing case, if $\omega=\Omega$, then there exist $h$ and $g$ such that 
\[
\ker\L^k = \text{span}\{\partial_\omega Q_\omega,iQ_\omega,h,ig\}\qtq{for}k\geq 4.
\]

\end{lemma}

\begin{proof} Throughout the proof, we use $c,c'$ to denote constants that may change from line to line.

We first suppose $\omega\neq\Omega$. We will rely on the following facts about the operators $L_\pm$ (see e.g. \cite{FJ, LeCoz}):
\begin{itemize}
\item[(i)] $L_-f = 0 \implies f =c Q_\omega$.
\item[(ii)] $L_+f = 0 \implies f = 0$.
\item[(iii)] $L_+f = cQ_\omega \implies f=c'\partial_\omega Q_\omega$.
\item[(iv)] $L_-f = c\partial_\omega Q_\omega \implies L_-f = 0 \implies f=c'Q_\omega$.
\end{itemize}
Indeed, (i) is just the fact that the kernel of $L_-$ is one-dimensional and spanned by $Q_\omega$, while (ii) is the fact that $L_+$ has no kernel.  Similarly, (iii) follows from $L_+\partial_\omega Q_\omega = -Q_\omega$ and (ii). Finally, the first implication in (iv) follows from the fact that $\partial_\omega Q_\omega \not\in[\ker(L_-)]^{\perp}$, which follows from \eqref{IPNZ} and (i). 

We proceed by induction.  First, for $k=2$, we write
\begin{equation}\label{L-squared}
\L^2=\left[\begin{array}{cc} - L_-L_+ & 0 \\ 0 & -L_+ L_-\end{array}\right].
\end{equation}
Thus if $\L^2 f = 0$, we have $L_-L_+ f_1 = 0$, yielding $L_+ f_1 = cQ_\omega$ and so $f_1 = c'\partial_\omega Q_\omega$.  Similarly, $L_+L_-f_2 = 0$, yielding $L_- f_2=0$ and hence $f_2=cQ_\omega$. 

Next, we suppose the result holds for some $k$ and suppose that
\[
\L^{k+1} f = \L^k\left[\begin{array}{rr} L_- f_2 \\ -L_+ f_1\end{array}\right] = 0.
\]
Then, by the inductive hypothesis, we must have $L_- f_2 = c\partial_\omega Q_\omega$, yielding $f_2 = c'Q_\omega$.  Similarly, $L_+ f_1 = cQ_\omega$, yielding $f_1 = c'\partial_\omega Q_\omega$.  This completes the proof in the case $\omega\neq\Omega$.

If $\omega=\Omega$, we have $\langle Q_\omega,\partial_\omega Q_\omega\rangle = 0$, and thus we may find $g$ such that $L_-g = \partial_\omega Q_\omega$.  We then define $h=L_+^{-1}g$, so that $L_+h = g$.  To prove the result we can then carry out an induction argument as above.\end{proof}

The eigenvalues of $\L(\omega)$ must be either real or purely imaginary.  Formally, this follows from the fact that $\L^2$ is self-adjoint (see \eqref{L-squared}), so that any eigenvalue $\lambda$ of $\L$ must obey $\lambda^2\in\R$.  To be more precise, we will use an argument from \cite{CGNT}. (See also \cite{BP, KS}.)

\begin{lemma} Eigenvalues of $\L$ must belong to $\R\cup i\R$.  Furthermore, nonzero eigenvalues must appear in $\pm$ or conjugate pairs. %and all eigenvalues are simple.
\end{lemma} 

\begin{proof} Suppose $\L[u\ v]^t = \lambda[u\ v]^t$ for some $\lambda\neq 0$ and nontrivial $[u\ v]^t$.  Then
\[
L_+ u =-\lambda v \qtq{and}L_-v = \lambda u,
\]
which implies
\[
\langle u,L_+ u\rangle = -\lambda\int \bar u v\,dx \qtq{and} \langle v,L_-v\rangle = \lambda\int \bar v u\,dx = \lambda\overline{\int \bar u v\,dx},
\]
so that
\[
\langle u,L_+u\rangle\langle v,L_-v\rangle = -\lambda^2\biggl|\int \bar u v\,dx\biggr|^2. 
\]

Now, if $v=0$, then since $L_+$ has no kernel we also find $u=0$, yielding a contradiction.  Similarly, since $\lambda\neq 0$, if $u=0$ then $v = \lambda^{-1}L_+ u=0$ and we obtain a contradiction. In particular, $L_- v = \lambda u\neq 0$ and hence $v$ is not a multiple of $Q_\omega$.  This implies $\langle v,L_-v\rangle>0$, which (using the identities above and recalling $\lambda\neq 0$) implies $\int \bar u v\,dx \neq 0$.  In particular, we deduce
\[
\lambda^2 = -\frac{\langle u,L_+ u\rangle\langle v,L_-v\rangle}{\bigl|\int \bar u v\,dx\bigr|^2} \in \R,
\]
yielding $\lambda\in\R\cup i \R$.  

The fact that eigenvalues occur in pairs follows from symmetry properties of $\L$ (see e.g. \cite{BP}). \end{proof}

\subsection{The focusing case.}\label{S:focusing-spectrum} In the focusing case, we have $\omega\in (\tfrac12q^2,\infty)$, with orbital stability for $\omega\in (\tfrac12q^2,\Omega)$ and instability for $\omega\in[\Omega,\infty)$.  We can try to piece together the global picture of the spectrum by first considering three regimes, which we can then try to connect in a consistent way.  In particular, we consider: (i) $\omega\to \tfrac12 q^2$, (ii) $\omega$ near $\Omega$, and (iii) $\omega\to\infty$. 
 
(i) First, from the explicit form of the solitary wave solutions (see Section~\ref{S:solitons}), the potentials in $L_\pm$ have the  form
\begin{equation}\label{the-potentials}
%\tfrac{\omega}{2}V_\pm\bigl(\sqrt{\tfrac{\omega}{2}}x\bigr)
 c_\pm\cdot\omega \cosh^{-2}\bigl[p\sqrt{\tfrac{\omega}{2}}|x|+\arctanh(\tfrac{|q|}{\sqrt{2\omega}})\bigr],
\end{equation}
where $c_-=-\tfrac{p+2}{2}$ and $c_+=-\tfrac{(p+1)(p+2)}{2}$. In particular, both potentials tend to zero as $\omega\to \tfrac12 q^2$, and so formally the operator $\L(\omega)$ converges to 
\[
\L(\tfrac12q^2):=\left[\begin{array}{cc} 0 & H+\tfrac12 q^2 \\ -[H+\tfrac 12 q^2] & 0 \end{array}\right].
\]
Now, the operator $H=-\tfrac12\partial_x^2+q\delta_0$ has a single negative eigenvalue $\tfrac12 q^2$ with a one-dimensional eigenspace spanned by $\phi_0(x) = |q|^{\frac12}e^{-|q|\,|x|}$; otherwise, $H$ has continuous spectrum in $[0,\infty)$.  Thus $\L(\tfrac12 q^2)$ has the two-dimensional kernel spanned by $\{\phi_0,i\phi_0\}$, and we can derive that any other eigenvalues must come in purely imaginary pairs.  However, the existence of a purely imaginary eigenvalue would imply the existence of a negative eigenvalue for the corresponding operator $\H(\tfrac12q^2)$ (see \eqref{defH}), which in this case is guaranteed to be nonnegative.  Therefore, we can see that our desired spectral condition holds formally as we send $\omega\to\tfrac12 q^2$, and so we expect the condition to hold at least in a neighborhood of $\tfrac12q^2$. (See also Section~\ref{S:embedded} below.)

(ii)  Analysis of the type carried out in Comech--Pelinovsky \cite[Section~3]{Comech-Pelinovsky}, together with the stability/instability results of \cite{KO, FOO}, can provide a picture of the spectrum in a small neighborhood of $\omega=\Omega$.  In particular, the `extra' two dimensions of the generalized kernel at $\omega=\Omega$ (see Lemma~\ref{L:genker}) will split off into a pair of eigenvalues, which will be real in the unstable case $(\omega>\Omega)$ and purely imaginary in the stable case $(\omega<\Omega)$. 

(iii) To study the regime $\omega\to\infty$, we introduce the scaling
\[
S_\omega f (x) := f(\sqrt{\omega}x).
\]

A direct computation using \eqref{the-potentials} shows that
\begin{align*}
L_\pm(\omega) & = \omega S_\omega[-\tfrac12\partial_x^2 + 1 + c_\pm \cosh^{-2}\bigl(\tfrac{p|x|}{\sqrt{2}}+\arctanh(\tfrac{|q|}{\sqrt{2\omega}})\bigr)+\tfrac{q}{\sqrt{\omega}}\delta_0]S_\omega^{-1} \\ 
& = \omega S_\omega[L_\pm^0 + E_{q,\omega,\pm}]S_\omega^{-1},
\end{align*}
where $L_\pm^0$ denote the operators arising in the linearization around the ground state for the \emph{free} NLS and $E_{q,\omega,\pm}$ will be `perturbative' terms.  More precisely, we recall that
\[
R(x) = \bigl(\tfrac{p+2}{2}\bigr)^{\frac{1}{p}} \cosh^{-\frac{2}{p}}(\tfrac{px}{\sqrt{2}}) 
\]
is the ground state soliton for the free NLS
\[
i\partial_t u + \tfrac12\partial_x^2 u = -|u|^p u,
\]
and that we have the corresponding linearized operators 
\[
L_+^0 = -\tfrac12\partial_x^2 + 1 - (p+1) R^p,\quad L_-^0 = - \tfrac12 \partial_x^2 + 1 - R^p. 
\]
The potentials $E_{q,\omega,\pm}$ introduced above are therefore given explicitly by 
\[
E_{q,\omega,\pm} = c_\pm[\cosh^{-2}\bigl(\tfrac{p|x|}{\sqrt{2}}+\arctanh(\tfrac{|q|}{\sqrt{2\omega}})\bigr) - \cosh^{-2}\bigl(\tfrac{p|x|}{\sqrt{2}})]+\tfrac{q}{\sqrt{\omega}}\delta_0. 
\]
The operator then takes the form
\begin{equation}\label{Lomegatoinfinity}
\L(\omega) = \omega S_\omega\bigl[ \L^0 + \mathcal{E}_{q,\omega}\bigr] S_\omega^{-1},
\end{equation}
where $\L^0$ is built out of $L_\pm^0$ in the usual way and $\mathcal{E}_{q,\omega}$ tends to zero as $\omega\to\infty$.  In particular, we expect that in the limit as $\omega\to\infty$, the spectrum of $\L(\omega)$ will be related to that of $\L^0$.  The spectral properties of this latter operator have been studied in great detail in previous works (see e.g. \cite[Proposition~9.2]{KS} or \cite{CGNT}). In the mass-supercritical setting, the discrete spectrum of $\L^0$ is six-dimensional, with a four-dimensional generalized kernel and a pair of real eigenvalues. %In addition, there are no embedded eigenvalues in the continuous spectrum, nor are there resonances at the edge of the continuous spectrum.

By connecting the three regimes described above, we can arrive at a conjecture for a qualitative description of the spectrum of $\L(\omega)$ for $\omega\in(\tfrac12q^2,\infty)$.

\begin{conjecture}[Spectrum of $\L(\omega)$ in the focusing, mass-supercritical case]\label{spectral-conjecture}\text{ }

There exists $\omega_1\in(\tfrac12q^2,\Omega)$ such that $\L(\omega_1)$ has a pair of resonances at $\pm i\omega_1$. For $\omega\in(\tfrac12q^2,\omega_1)$ the spectral condition of Definition~\ref{D:spectral} holds.  For $\omega>\omega_1$, there exists a first eigenvalue pair $\pm\lambda_1(\omega)$.  This pair corresponds to the pair of eigenvalues described in regime (ii) above.  In particular, we have 
\[
\pm \lambda_1(\omega) \in \begin{cases} i\R\backslash \sigma_c(\mathcal{L}(\omega))&\qtq{for}\omega_1<\omega<\Omega \\
\{0\} & \qtq{for}\omega=\Omega \\
 \R & \qtq{for}\omega>\Omega.\end{cases}
\]

There exists $\omega_2>\omega_1$ such that $\L(\omega_2)$ has a second pair of resonances at $\pm i\omega_2$.  For $\omega>\omega_2$, there is a second eigenvalue pair $\pm\lambda_2(\omega)$. This pair corresponds to two of the dimensions of the generalized kernel of $\L^0$ in regime (iii) above (see \eqref{Lomegatoinfinity}).  In particular, we have have 
\[
\pm\lambda_2(\omega) \in i\R\backslash\sigma_c(\L(\omega))\qtq{for}\omega>\omega_2,
\]
with $\pm\lambda_2(\omega)\to 0$ as $\omega\to\infty$.

For all $\omega\in(\tfrac12q^2,\infty)$, there are no eigenvalues embedded in the continuous spectrum of $\L(\omega)$. 
\end{conjecture}

The claims made in Conjecture~\ref{spectral-conjecture} are depicted in the following figures, which give qualitative representations of the spectrum of $\L(\omega)$ as $\omega$ varies in $(\tfrac12q^2,\infty)$.  We have presented the scenario in which the first pair of eigenvalues crosses the origin before the emergence of the second eigenvalue pair (which we found to be the case for $p>p_0\approx 4.54$).  

\centerline{
\includegraphics[width=.5\linewidth]{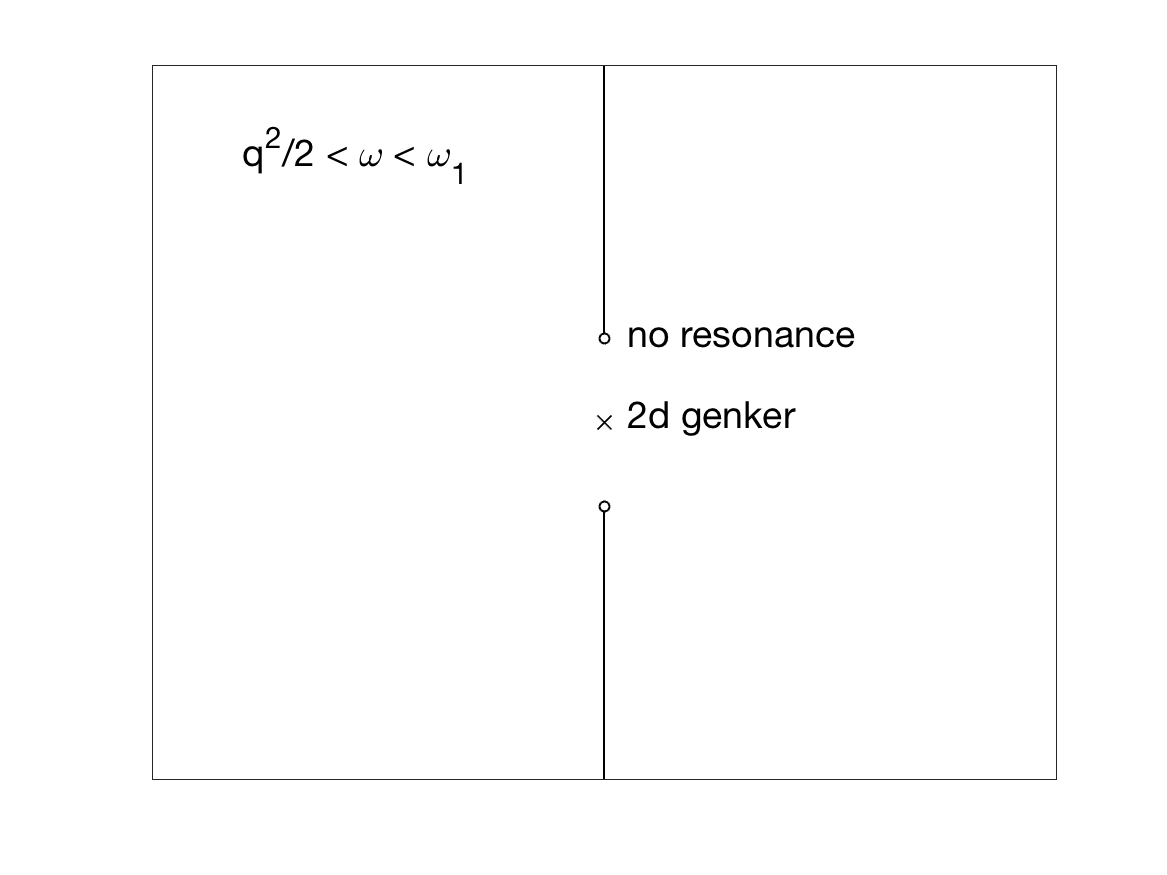}\includegraphics[width=.5\linewidth]{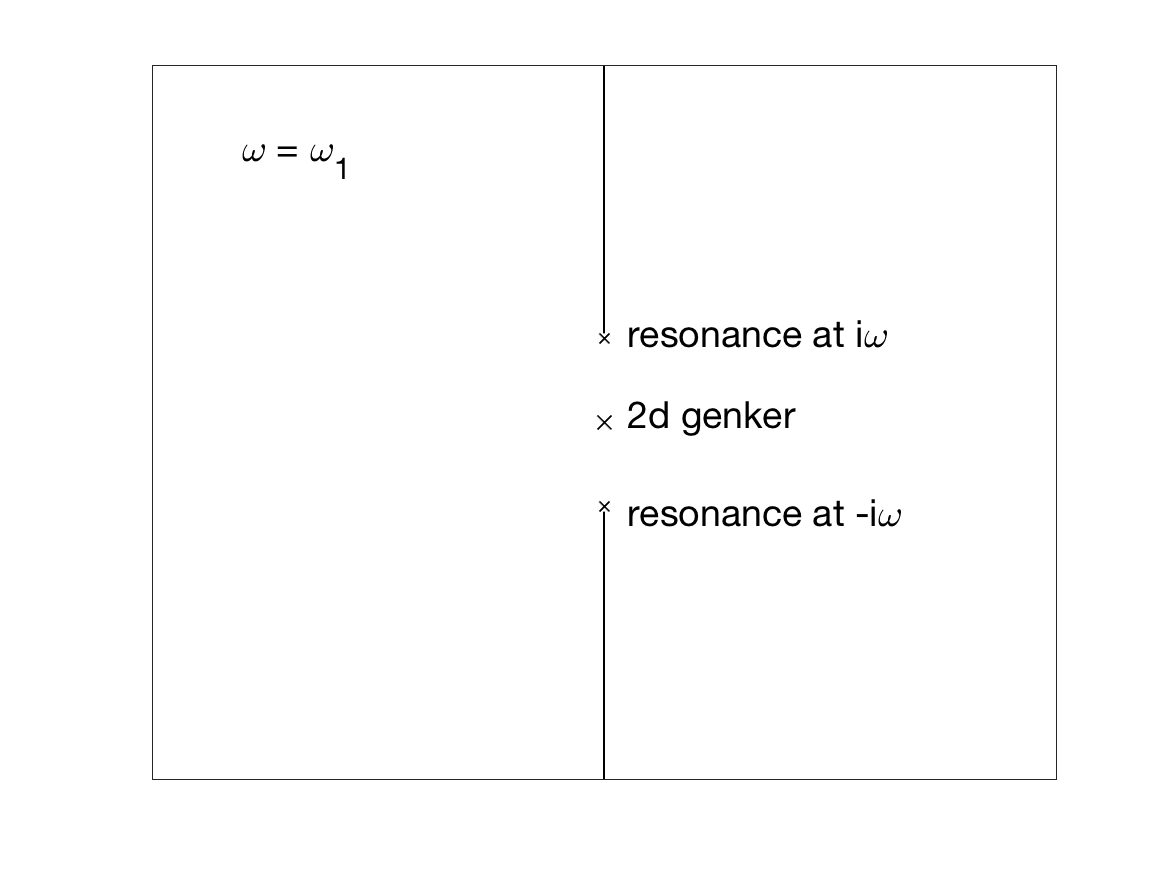}}

\centerline{
\includegraphics[width=.5\linewidth]{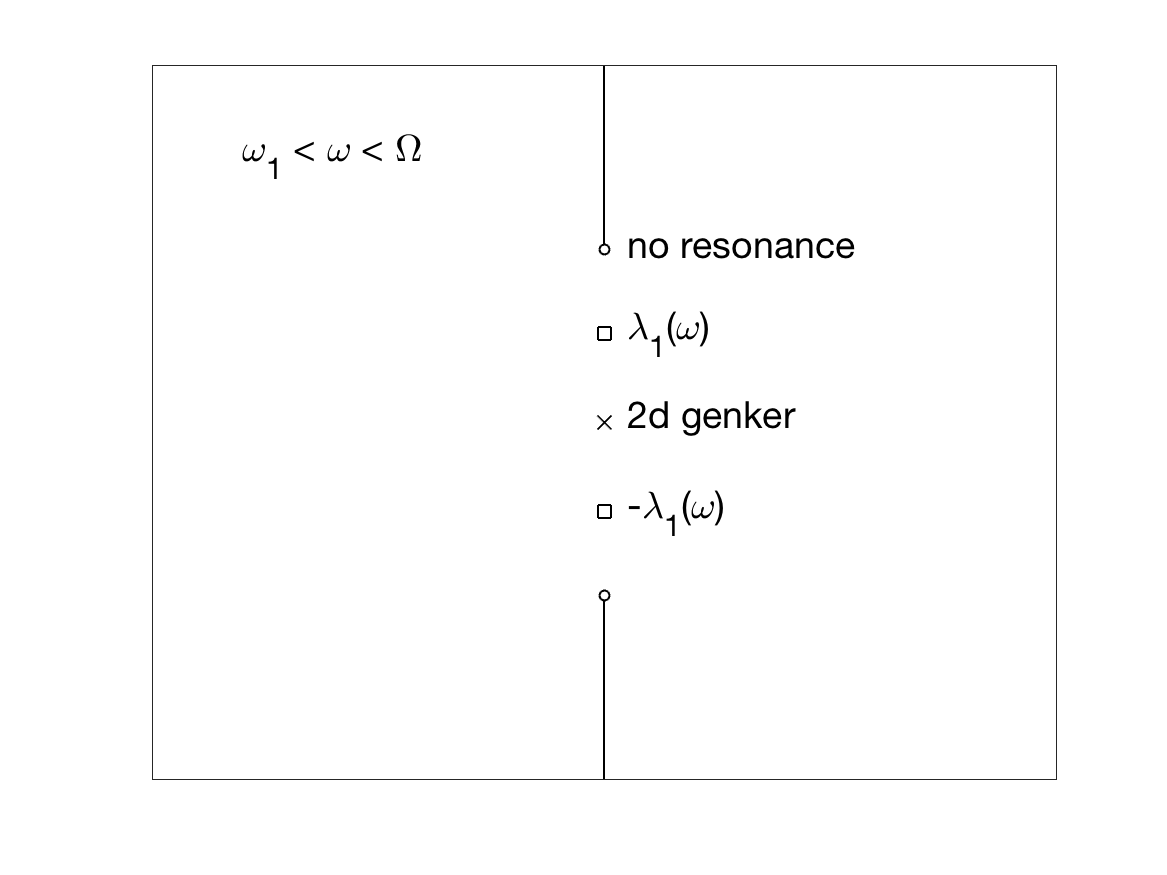}\includegraphics[width=.5\linewidth]{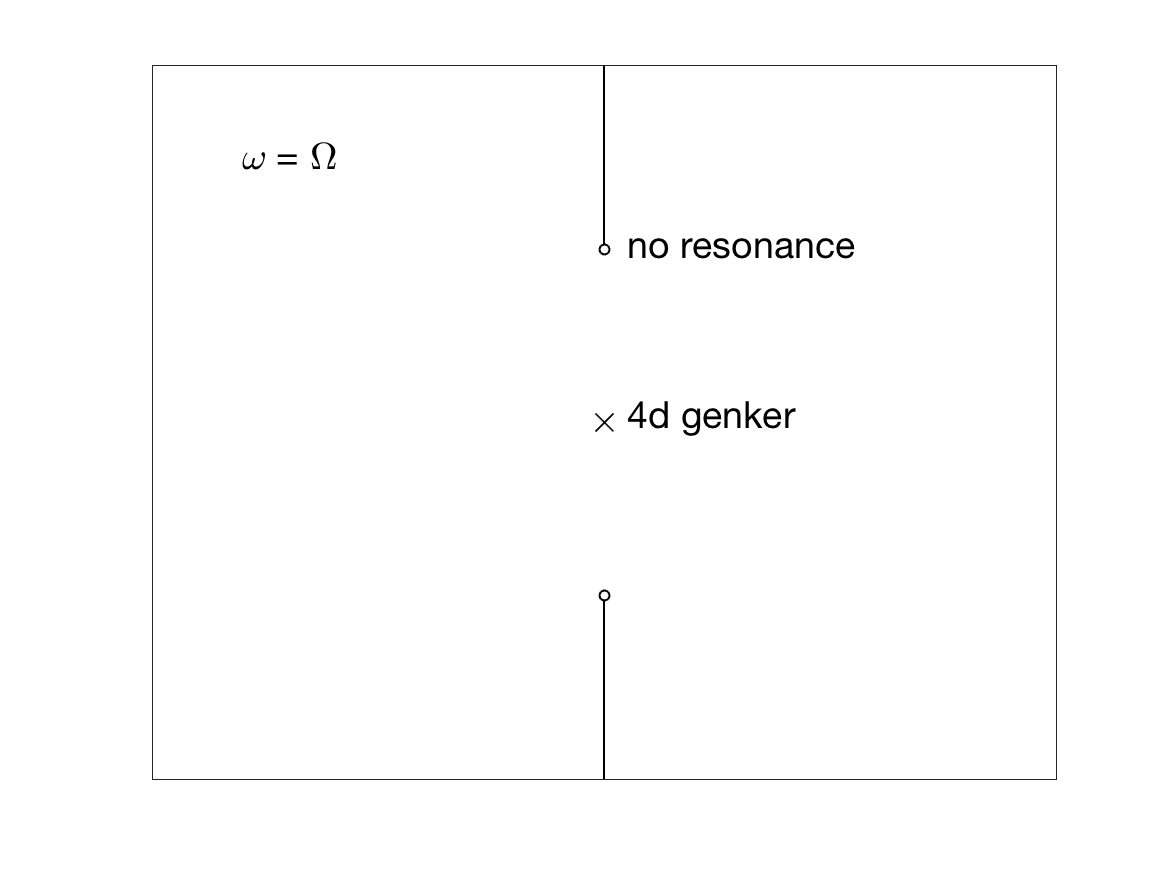}}

\centerline{
\includegraphics[width=.5\linewidth]{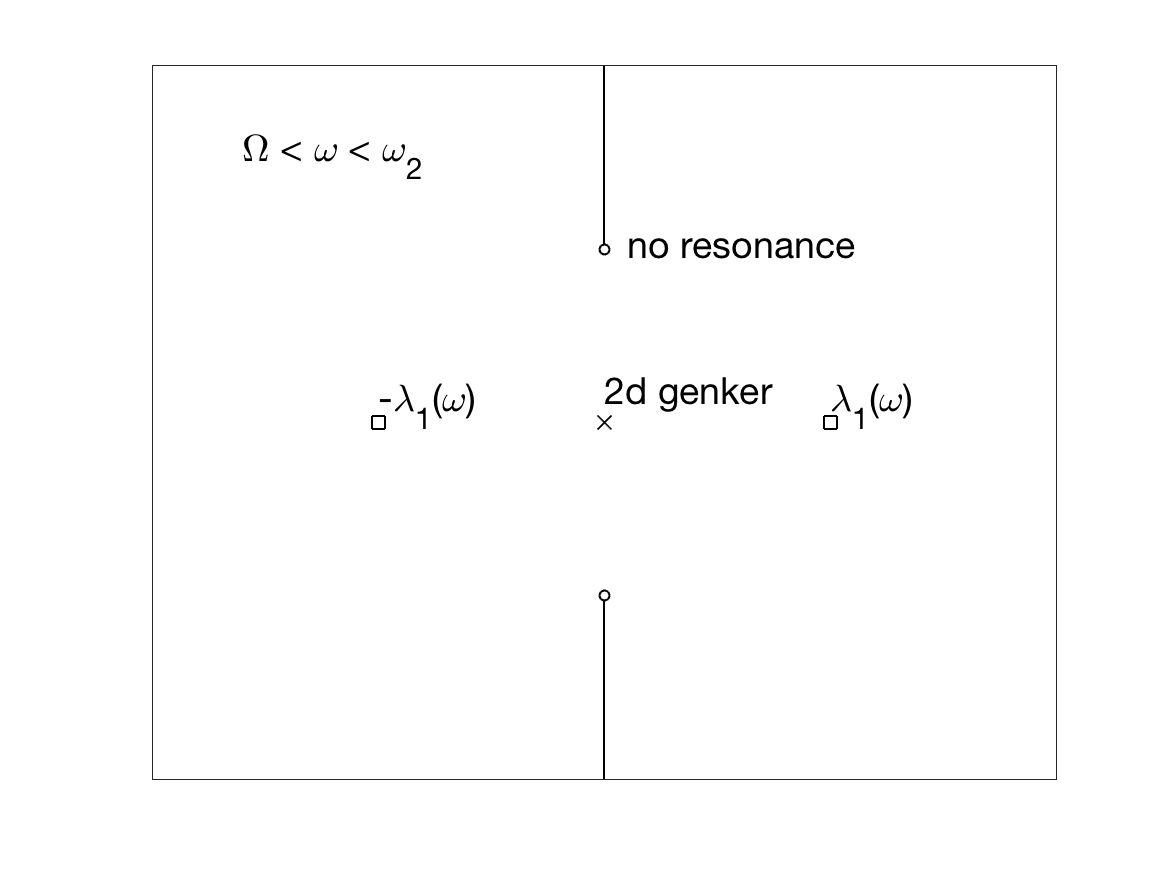}\includegraphics[width=.5\linewidth]{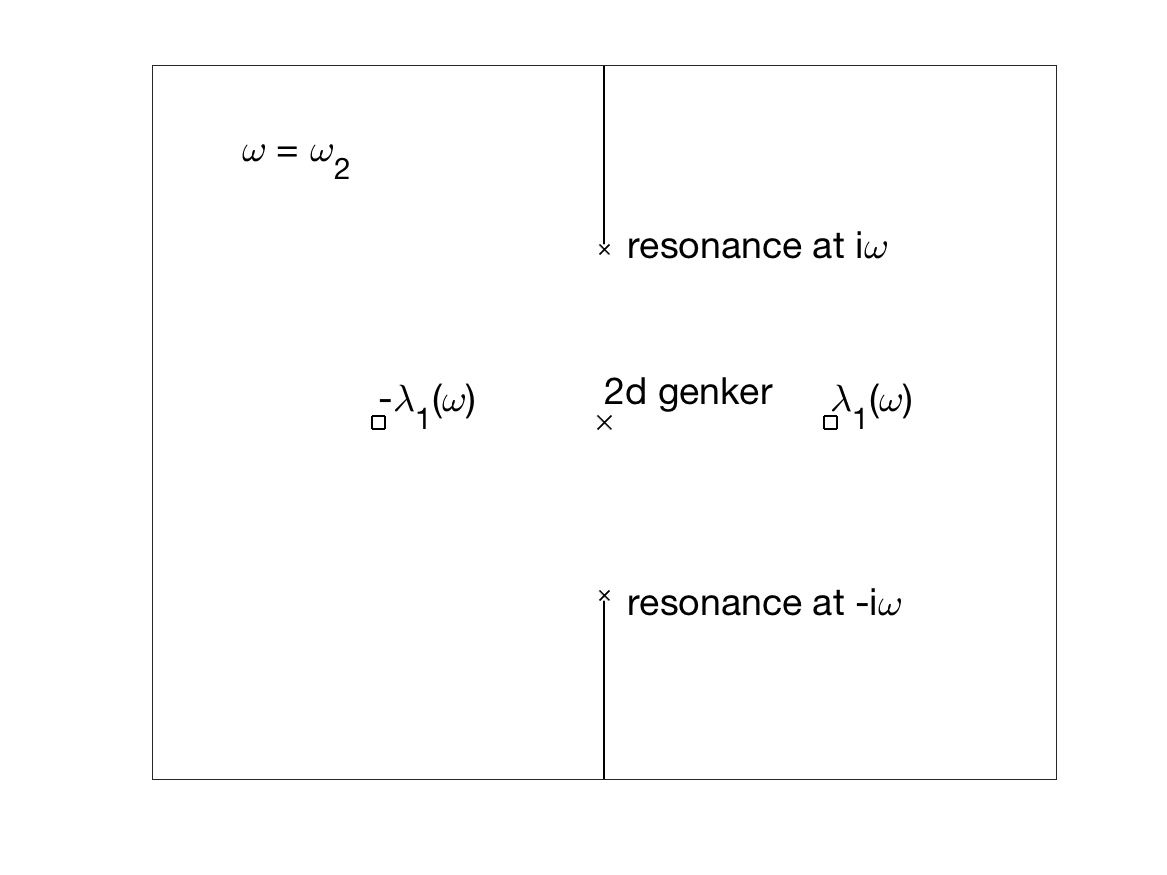}}

\centerline{
\includegraphics[width=.5\linewidth]{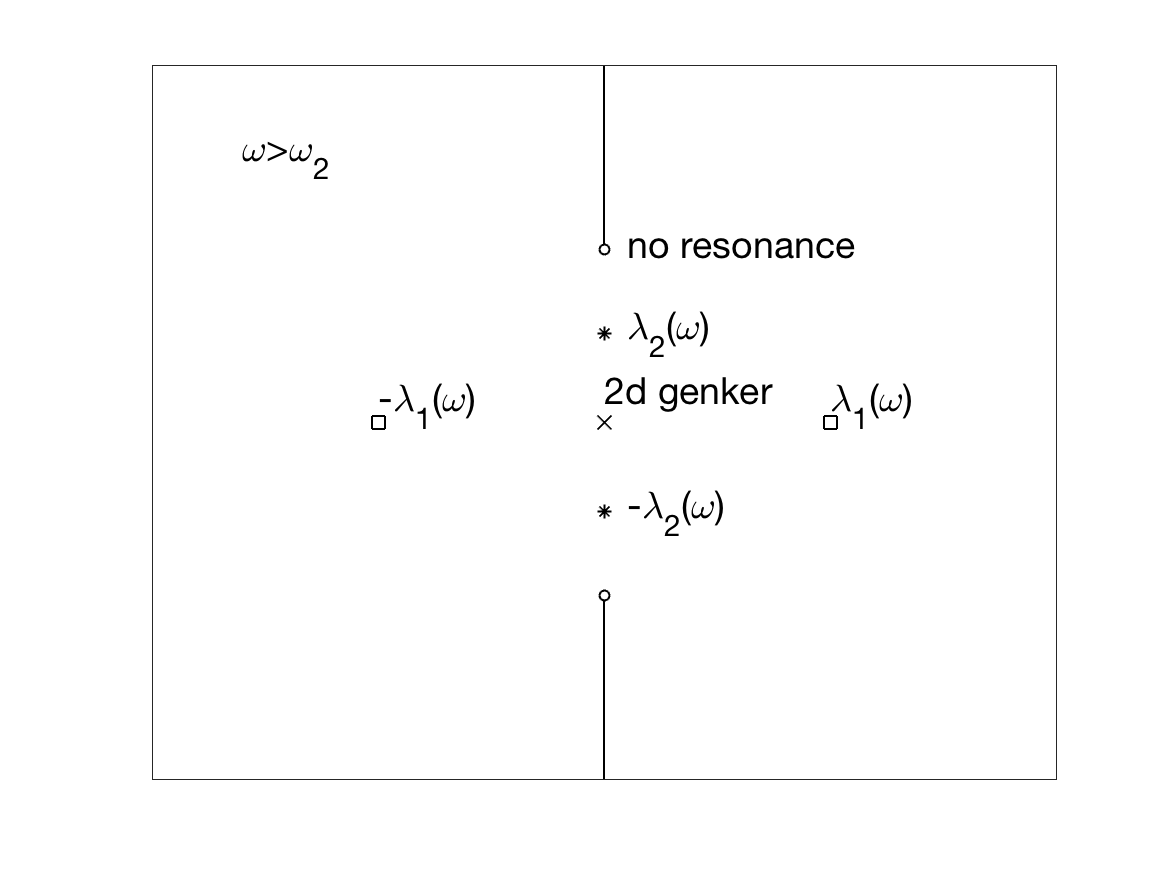}}

\subsubsection{Numerical search for resonances}\label{S:search} To provide evidence for the conjecture above, we undertook a numerical search for resonances of $\L(\omega)$ appearing at $\pm i\omega$.  For all numerical investigations, we fixed the parameter $q=-1$ and focused on finding the resonance at $i\omega$.  

To formulate the problem, we first rewrite the resonance equation
\[
\L(\omega)\left[\begin{array}{r} f \\ g \end{array}\right] = i\omega \left[\begin{array}{r} f \\ g \end{array}\right]
\]  
for complex-valued $f=f_1+if_2$ and $g = g_1 + ig_2$ as a system of four second-order ODEs, namely:
\begin{equation}\label{jODE-system}
\begin{aligned}
& L_-g_1 = -\omega f_2, \quad L_- g_2 = \omega f_1, \\
& L_+f_1 = \omega g_2, \quad L_+ f_2 = -\omega g_1. 
\end{aligned}
\end{equation}
As usual, to deal with the presence of the delta potential we view each equation as an ODE on $\R_-\cup\R_+$ together with the jump condition at $x=0$ (see \eqref{domain} at $x=0$). On each of $\R_\pm$, the operator $H$ acts as $-\tfrac12\partial_x^2$. 

As these equations are linear, it suffices to look for even or odd solutions separately.  In each case, it is enough to solve the ODE on $\R_+$, impose a suitable boundary condition at $x=0$, and extend the solution to $\R$ by taking either the even or odd reflection.  For even solutions, we obtain the boundary condition
\begin{equation}\label{jBC-even}
u'(0) = qu(0), 
\end{equation}
whereas for odd solutions we obtain 
\begin{equation}\label{jBC-odd}
u(0)=0.
\end{equation}

To solve \eqref{jODE-system} numerically, we utilize the {\tt bvp4c} Matlab package, viewing the equations as boundary-value problems on $[0,x_0]$ for sufficiently large $x_0$.  We chose $x_0=50$ and found that increasing $x_0$ further did not change the numerical results.  The ODE system \eqref{jODE-system} is straightforward to study numerically, as the potentials appearing in $L_\pm$ are given explicitly using the formulas for the nonlinear bound states given in Section~\ref{S:solitons}. 

We impose either \eqref{jBC-even} or \eqref{jBC-odd} for each of $f_1$, $f_2$, $g_1$, and $g_2$, depending on whether we seek even or odd solutions.  For the remaining four boundary conditions, one possibility is to impose $u'(x_0)=0$ for each of $f_1$, $f_2$, $g_1$, and $g_2$.  This is consistent with the existence of a resonance, which should correspond to a solution in $(L^\infty \cap \dot H^1)\backslash L^2$.  It would also be consistent with the existence of an eigenvalue. However, this choice of boundary condition turns out to be difficult to work with numerically, as it tends to yield solutions so small that they cannot be meaningfully distinguished from the zero solution.

Thus, for the remaining four boundary conditions we proceed as follows.  Recalling that we are looking for a resonance, we first seek to impose the conditions $f_1\to 1$ and $f_1'\to 0$ as $x\to\infty$. Thus, for the numerical computations we impose 
\[
f_1(x_0)=1\qtq{and} f'(x_0)=0.
\]
Next, by elliptic regularity it is natural to assume that the second derivatives vanish as $x\to\infty$ for solutions in $L^\infty\cap \dot H^1$.  In this case, taking the limit as $x\to\infty$ in \eqref{jODE-system} yields the two final numerical boundary conditions 
\[
f_1(x_0)=g_1(x_0) \qtq{and} f_2(x_0)=-g_2(x_0). 
\] 

We now employ a shooting method with the parameter $\omega$.  That is, we fix a discrete set $D$ and numerically solve the system \eqref{jODE-system} with the boundary conditions above for each $\omega\in D$.  Most solutions have asymptotically linear behavior as $x\to x_0$, corresponding to solutions to \eqref{jODE-system} that grow in magnitude as $x\to\infty$.  On the other hand, a resonance should be asymptotically flat.  To look for possible resonances, we therefore measure the norm of the derivatives $f_1'$, $f_2'$, $g_1'$, and $g_2'$ in a neighborhood of $x_0$.  If these norms become small around some $\omega'\in D$, it suggests the existence of a resonance near $\omega'$.  We may then take a refined set $D$ of frequencies near $\omega'$  and start the search again.  After several iterations, we can identify the resonance.  The idea of the method is illustrated briefly in Figure~\ref{shooting1}.

\begin{figure}[h]
\includegraphics[scale=.15]{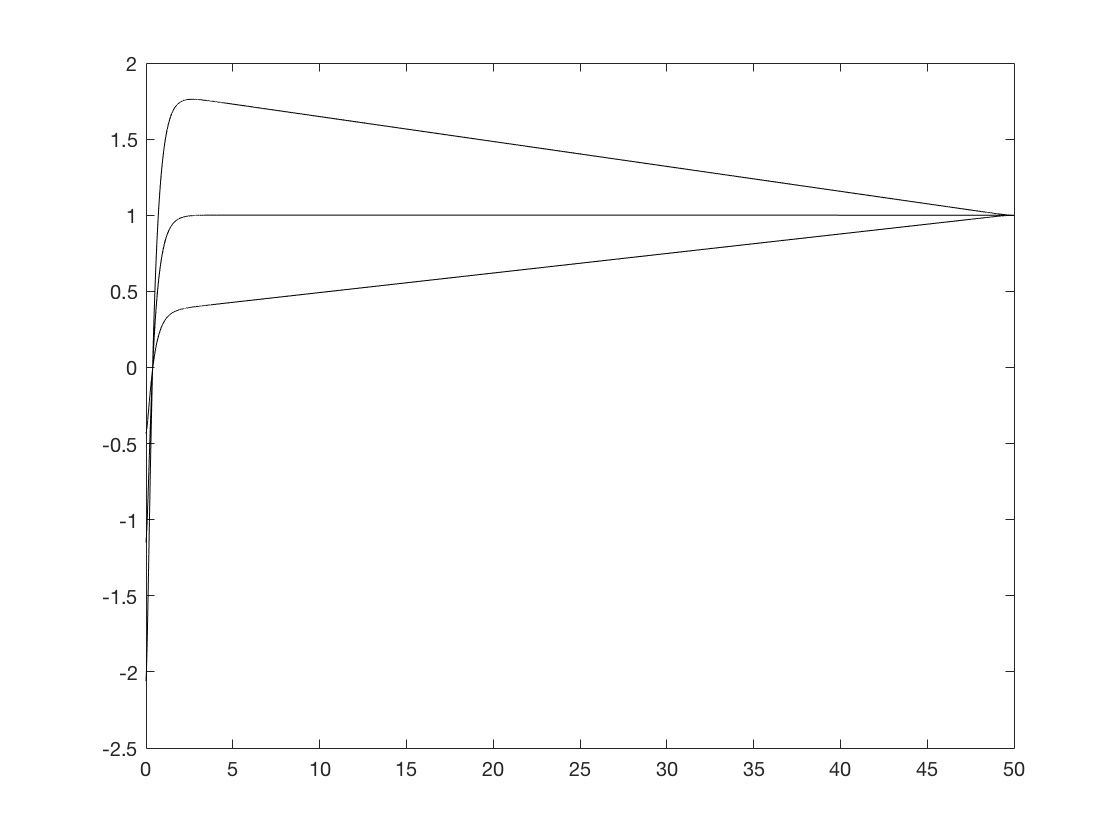}
\includegraphics[scale=.15]{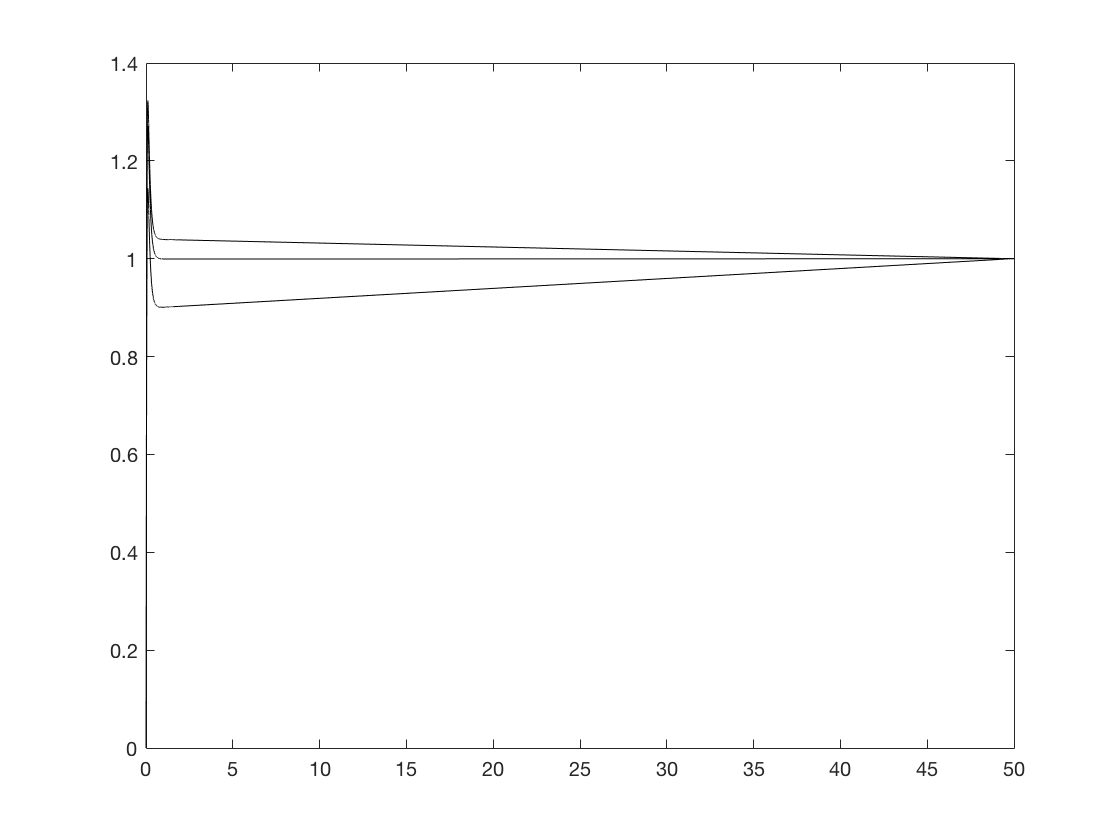}
\caption{\footnotesize (Left) The numerically computed solution $f_1$ in the system \eqref{jODE-system} with $p=5$ and three choices of $\omega$. The asymptotically flat solution corresponds to $\omega_1\approx1.49171$, at which we find an even resonance.  The other two solutions correspond to values of $\omega$ slightly above/below this value. (Right) The analogous situation when searching for an odd resonance, which we find at $\omega_2\approx 19.5722.$}\label{shooting1}
\end{figure} 

For a range of values of $p\geq 4$, we consistently found one even resonance and one odd resonance.  The even resonance always occurred at a value $\omega_1<\Omega$, which is what we expected in light of the discussion above.  As $p$ increases, the odd resonance initially occurs at a value $\omega_2<\Omega$ and later occurs at $\omega_2>\Omega$, with the transition occurring around the value $p_0\approx 4.54$.  In particular, for $p<p_0$ and $\omega_2<\omega<\Omega$, two pairs of imaginary eigenvalues can co-exist, whereas for $p>p_0$ this cannot occur.  The shapes of the resonances appear already in Figure~\ref{shooting1}.  Closer views of these solutions near the origin are displayed in Figure~\ref{show-resonances}.

\begin{figure}[h]
\includegraphics[scale=.15]{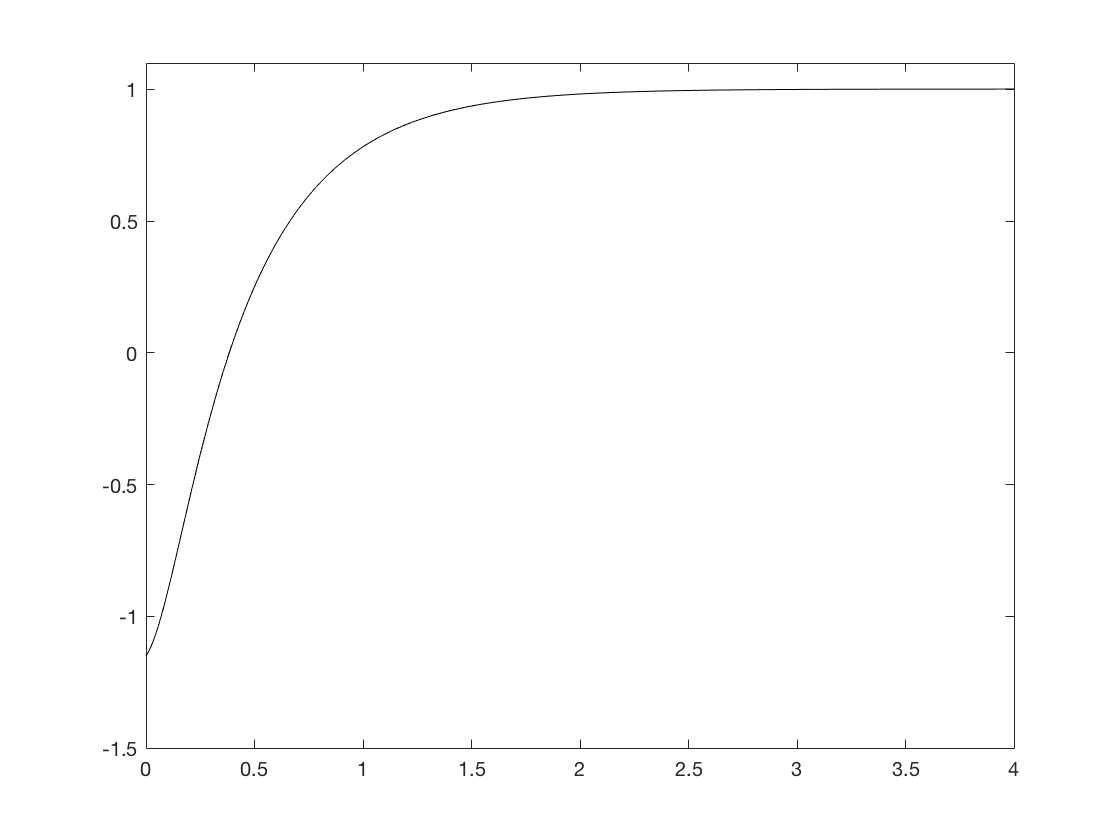}
\includegraphics[scale=.15]{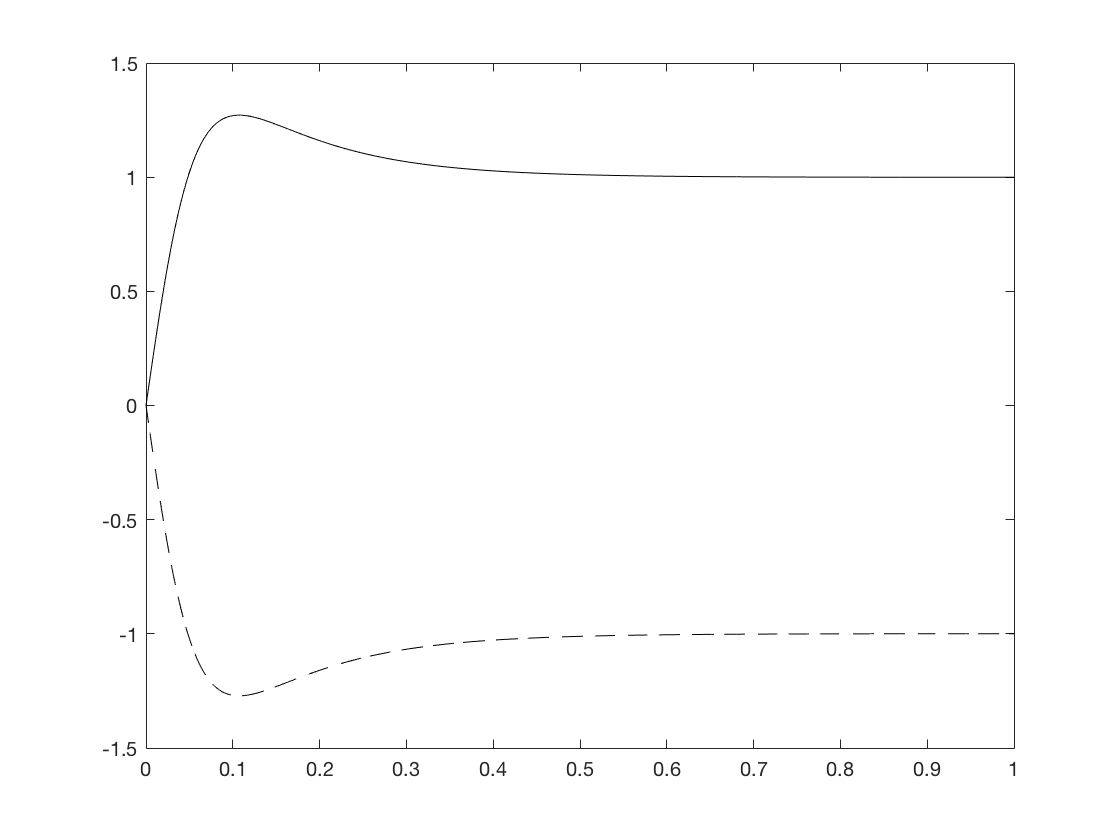}
\caption{\footnotesize (Left) A closer look at the shape of the $f_1$ component of the even resonance near $x=0$. We obtain the full solution by taking an even reflection across $x=0$. (Right) The shape of the $f_1$ and $f_2$ components of the odd resonance for $x>0$.  We obtain the full solution by taking an odd reflection across $x=0$.  These solutions both correspond to the case $p=5$, with the even resonance occurring at $\omega_1\approx 1.49171$ and the odd resonance occurring at $\omega_2 \approx 19.5722$.}\label{show-resonances}
\end{figure}

In Figures~\ref{special-omegas1}~and~\ref{odd-resonances}, we display our main numerical results on the existence of resonances for $\L(\omega)$.  Here $q=-1$ is fixed throughout.  In Figure~\ref{special-omegas1}, we plot the value of the even resonance $\omega_1$ and the threshold frequency $\Omega$ against the power $p$.  We also display the mass of $Q_{\omega_1}$ as a function of $\omega_1$, representing the upper bound on solitary wave mass for which our main result (Theorem~\ref{T}) applies in the focusing case.  In particular, this demonstrates that we can treat solitary waves that are $\mathcal{O}(1)$ in $L^2$, and hence our result does indeed extend beyond the small solitary wave regime.  We list some  computed values of $\omega_1$ and $M(Q_{\omega_1})$ in the table on page~\pageref{the-table}.  As $p\to\infty$, one finds that $\Omega$ (and hence $\omega_1$) converges to $\tfrac12q^2$, so that one has stability of any type only for small solitary waves.  On the other hand, as $p$ decreases to $4$ we obtain $\Omega\to\infty$, corresponding to the fact that all solitary waves are stable when $p=4$, while we are still able to find an even resonance appearing at $\omega_1\approx 2.6648$.

\begin{figure}[h]
\includegraphics[scale=.15]{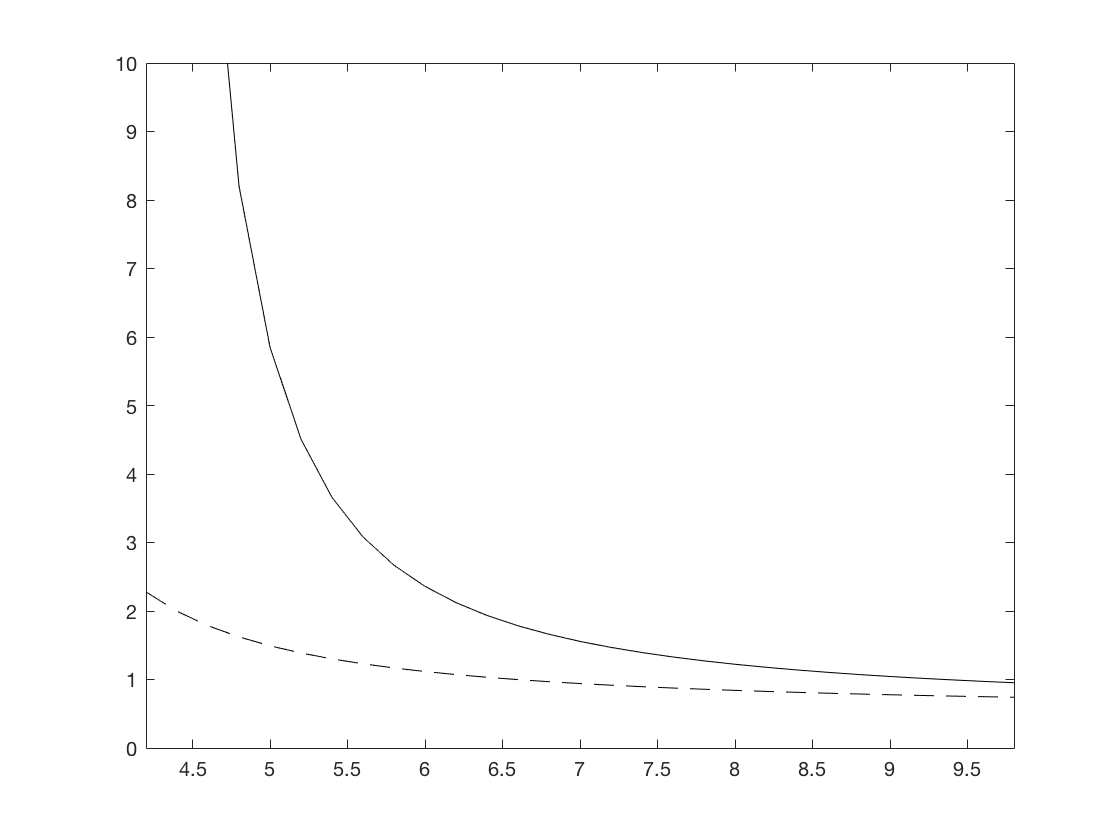}
\includegraphics[scale=.15]{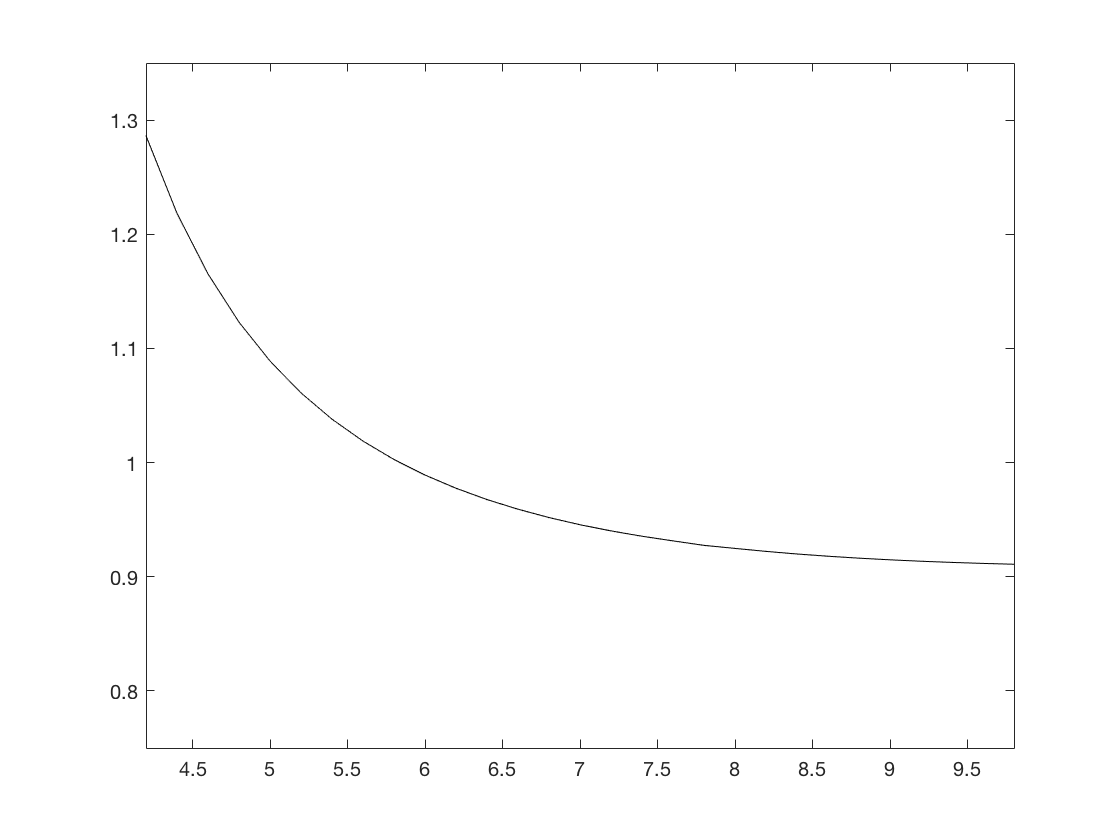}
\caption{\footnotesize (Left) A plot the curves $p\mapsto\Omega$ (solid line) and $p\mapsto \omega_1$ (dashed line).  (Right) A plot of the curve $p\mapsto M(Q_{\omega_1})$, which represents the upper bound on the mass of solitary waves addressed by Theorem~\ref{T} in the focusing case.}\label{special-omegas1}
\end{figure}

In Figure~\ref{odd-resonances}, we also plot  the odd resonance $\omega_2$, which appears to satisfy $\omega_2\to\infty$ as $p\to\infty$.  As these resonances always appear after the even resonance, they are not directly relevant for the analysis of the present paper.  

\begin{figure}[h]
\includegraphics[scale=.15]{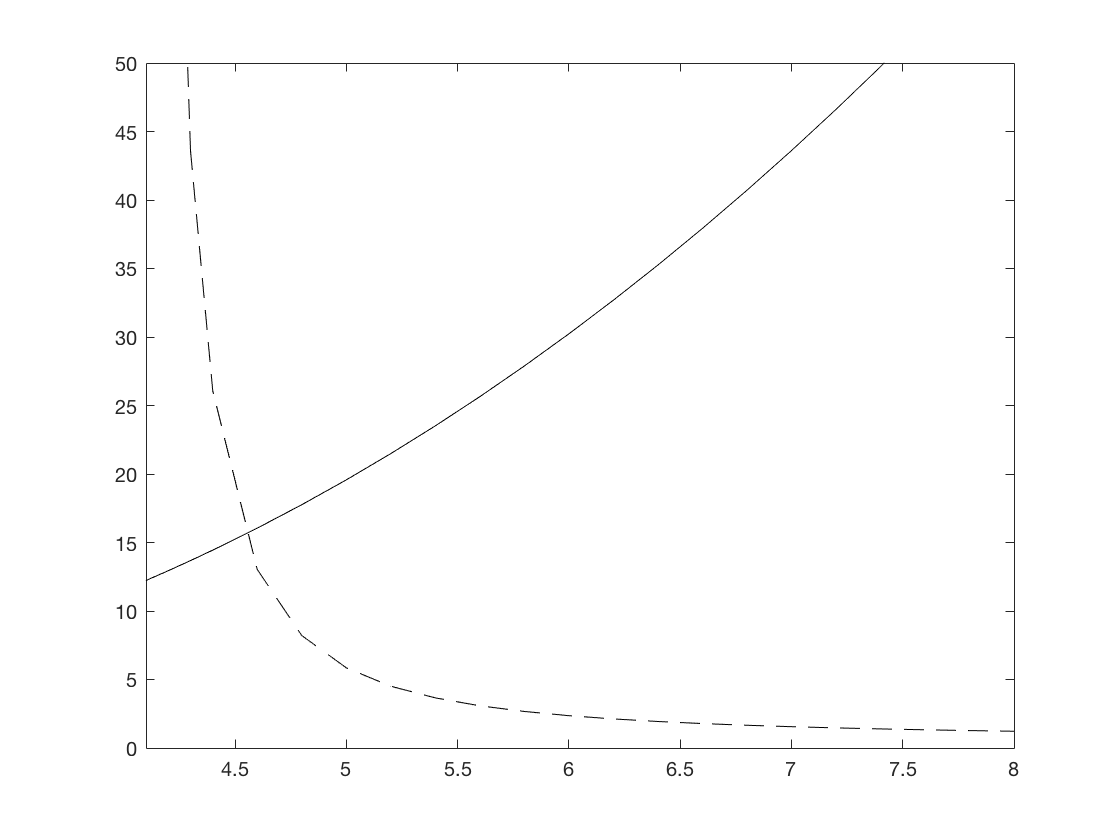}
\caption{\footnotesize A plot of $p\mapsto \Omega$ (dashed line) and $p\mapsto\omega_2$ (solid line).  The curves cross around $p\approx 4.54$.}\label{odd-resonances}
\end{figure}
 
\subsubsection{Embedded eigenvalues.}\label{S:embedded} In Conjecture~\ref{spectral-conjecture}, we assert that there are no eigenvalues embedded in the continuous spectrum of $\L(\omega)$.  This is a generic assumption in stability problems that has been established rigorously in some related settings (see e.g. \cite{BP, Perelman, KS, HZ2}); see also \cite{MS} for a numerical approach.  Some of the theory that will be developed in the next section already rules out the possibility of embedded eigenvalues of sufficiently large magnitude (in both the focusing and defocusing cases).  In that section, the analysis is done in terms of the operator
\begin{equation}\label{defH}
\H  =\left[\begin{array}{cc} H+\omega & 0 \\ 0 & -H-\omega\end{array}\right] + \sigma Q_\omega^p \left[\begin{array}{cc}\tfrac{p+2}{2} & \tfrac{p}{2} \\ -\tfrac{p}{2} & -\tfrac{p+2}{2}\end{array}\right],
\end{equation}
which is connected to $\mathcal{L}$ via 
\begin{equation*}
-iU^{\ast}{{\mathcal L}}U=\H,\quad U=\frac{1}{\sqrt{2}}
\left[\begin{array}{rr}1 & 1 \\ i & -i\end{array}\right].
\end{equation*}
Ruling out imaginary eigenvalues of $\L$ is then equivalent to ruling out real eigenvalues of $\H$.  In Proposition~\ref{P:spectral-meets-scattering}, we define a matrix-valued function $D(\xi)$ (given as the matrix Wronskian of two particular solutions to the generalized eigenvalue problem for $\H$) and show that $\omega+\tfrac12\xi^2$ is an eigenvalue of $\H$ if and only if $\det D(\xi)=0$.  On the other hand, in Lemma~\ref{L:AB} (see \eqref{determinant-asymptotic}) we derive that 
\[
\det D(\xi) = -4i\xi\sqrt{\xi^2+4\omega} + \mathcal{O}(1) \qtq{as}|\xi|\to\infty. 
\]
This precludes the existence of large eigenvalues for $H$, and so rules out the possibility of large imaginary eigenvalues for $\L$.

We can consider the limiting case $\omega\to\tfrac12q^2$.  In this case, one can compute explicitly to derive
\[
\det D(\xi) \to -4[i\xi-q][\sqrt{\xi^2+4\omega}+q],
\]
which is always nonzero.  Using continuity with respect to $\omega$, this implies the absence of embedded eigenvalues at least for $\omega$ in a neighborhood of $\tfrac12q^2$.  In the present paper, we take the absence of any embedded eigenvalues for $\L(\omega)$ as an assumption (see the spectral condition in Definition~\ref{D:spectral}).  We will address this problem in more detail in future work.

%%%%%
\subsection{The defocusing case.}\label{S:defocusing-spectrum} In the defocusing case, orbital stability holds for all $\omega\in(0,\tfrac12q^2)$.  Our conjecture for the defocusing case is that the spectral condition in Definition~\ref{D:spectral} holds for all $\omega\in(0,\tfrac12q^2)$, so that Theorem~\ref{T} applies to the entire family of solitary wave solutions.  In what follows, we give some evidence for this conjecture.  A more rigorous analysis will be the topic of future work. 

The limit $\omega\to\tfrac12 q^2$ is essentially the same as in the focusing case, as the solitary wave solutions converge to zero.  Thus we once again expect the spectral condition to hold at least in a neighborhood of $\tfrac12q^2$.  As in the focusing case, we carried out a numerical search for resonances at the edge of the continuous spectrum signaling the emergence of imaginary eigenvalues.  In contrast to the focusing case, we found none.  In addition, orbital stability suggests that real eigenvalues should not occur, as these would correspond to exponentially growing modes for the underlying linearized equation. This leads us to the assertion that for all $\omega\in(0,\tfrac12q^2)$, the discrete spectrum of $\L(\omega)$ consists only of the $2d$ generalized kernel. Similar to the focusing case, we also assert that embedded eigenvalues do not occur, with the special cases of large magnitude eigenvalues already precluded as in Section~\ref{S:embedded}.  

Proceeding analogously to the focusing case, we may also attempt to connect the regime $\omega\to\tfrac12q^2$ to the other limiting regime $\omega\to 0$.  In this latter regime, we obtain the operator
\begin{equation}\label{mathcalL0}
\L(0) := \left[\begin{array}{cc} 0 & H \\ -H & 0 \end{array}\right] + V(x)\left[\begin{array}{cc} 0 & 1 \\ -(p+1) & 0 \end{array}\right],
\end{equation}
where
\[
V(x) = \frac{(p+2)|q|^2}{(p|q||x|+2)^2}
\]
(see \eqref{Q0def}).  If we additionally send $q\to 0$ in \eqref{mathcalL0}, then we have $V\to 0$ and $H\to -\tfrac12\partial_x^2$, and the discrete spectrum consists only of a resonance of multiplicity 2 at the origin (corresponding to the constant functions $[1\ 0]^t$ and $[0\ 1]^t$).  For $q<0$, we retain this resonance at zero, which corresponds to the limit of the $2d$ generalized generalized kernel of $\mathcal{L}(\omega)$.  We obtain a resonance (rather than an eigenvalue) due to the fact that in the mass-supercritical regime, the limits of $Q_\omega$ and $\partial_\omega Q_\omega$ as $\omega\to 0$ fail to belong to $L^2$ (see \eqref{Q0def}).  In light of the above, we expect that there is no additional discrete spectrum for the limiting operator $\L(0)$.  Our assertion that the discrete spectrum of $\L(\omega)$ is given only by its $2d$ generalized kernel for all $\omega\in(0,\tfrac12q^2)$ then connects these two limiting regimes in a consistent way.

\subsection{Projection formula.} To conclude this section, we establish a formula for the projection away from the continuous spectrum for $\L(\omega)$, under the assumption that $\omega$ obeys the spectral condition appearing in Definition~\ref{D:spectral}.  The argument is the same as in \cite[Proposition~C.2]{GS}.  

\begin{proposition}\label{P:Projection} Assume that $\omega$ is such that \eqref{IPNZ} holds and $\L(\omega)$ has no resonances and no eigenvalues other than zero.  Let $P_c$ denote the projection onto the continuous spectrum of $\L(\omega)$. Then
\[
(I-P_c) f = \frac{\Re \langle f,Q_\omega\rangle}{\langle Q_\omega,\partial_\omega Q_\omega\rangle} \partial_\omega Q_\omega + \frac{\Re \langle f,i\partial_\omega Q_\omega\rangle}{\langle Q_\omega,\partial_\omega Q_\omega\rangle} iQ_\omega. 
\]
\end{proposition}

\begin{proof} Using Lemma~\ref{L:genker}, we find that the range of $I-P_c$ equals the span of $\{\partial_\omega Q_\omega,iQ_\omega\}$. Thus we have
\[
(I-P_c)f = a_1 \partial_\omega Q_\omega + ia_2Q_\omega
\]
for some $a_1,a_2\in\R$.   In particular,
\begin{align*}
\Re \langle (I-P_c)f, Q_\omega\rangle & = a_1\langle Q_\omega,\partial_\omega Q_\omega\rangle, \\
\Re\langle (I-P_c)f,i\partial_\omega Q_\omega\rangle & = a_2 \langle Q_\omega,\partial_\omega Q_\omega\rangle. 
\end{align*}
As the transpose $P_c^t$ is the projection onto the continuous spectrum of $\L^t$, we derive that $(I-P_c^t)Q_\omega=Q_\omega$ and $(I-P_c^t)i\partial_\omega Q_\omega = i\partial_\omega Q_\omega$.  Solving for $a_1$ and $a_2$ now yields the result. 
\end{proof}

%%%%%%%%%%%%%%%%%%%%%
%%%%%%%%%%%%%%%%%%%%%
%%%%%%%%%%%%%%%%%%%%%
%%%%%%%%%%%%%%%%%%%%%
%%%%%%%%%%%%%%%%%%%%%
%%%%%%%%%%%%%%%%%%%%%
\section{Dispersive estimates}\label{S:dispersive}

In this section we derive dispersive estimates 
for the evolution group $e^{t\mathcal{L}}$, where $\mathcal{L}=\mathcal{L}(\omega)$ is as in \eqref{D:L}.  We first observe that 
\begin{equation*}
-iU^{\ast}{{\mathcal L}}U=\H,
\end{equation*}
where 
\[
U=\frac{1}{\sqrt{2}}
\left[\begin{array}{rr}1 & 1 \\ i & -i\end{array}\right]
\]
and
\begin{equation}\label{defH}
\H = \H_0+\V:= \left[\begin{array}{cc} H+\omega & 0 \\ 0 & -H-\omega\end{array}\right] + \sigma Q_\omega^p \left[\begin{array}{cc}\tfrac{p+2}{2} & \tfrac{p}{2} \\ -\tfrac{p}{2} & -\tfrac{p+2}{2}\end{array}\right]. 
\end{equation}
In particular, it is equivalent to prove dispersive estimates for the unitary group $e^{it\H}$, as we have 
$e^{t{{\mathcal L}}}=Ue^{it\H}U^{\ast}$.

\begin{proposition}[Dispersive Estimates]\label{P:D} Suppose $\langle Q_{\omega},\partial_\omega Q_{\omega}\rangle\neq 0$ and the spectral condition in Definition~\ref{D:spectral} holds. Then, writing $P_{c}$ for the projection to the absolutely continuous spectrum of $\H$, we have 
\begin{align}
\|e^{it\H}P_c \varphi\|_{L^\infty}&\lesssim |t|^{-\frac12}\|\varphi\|_{L^1}, \label{dispersive1} \\
\|\langle x\rangle^{-1}e^{it\H}P_c\varphi\|_{L^\infty}&\lesssim |t|^{-\frac32}\|\langle x\rangle \varphi\|_{L^1}. \label{dispersive2}
\end{align}
In addition, we have the $L^2$ bound
\begin{equation}\label{dispersive3}
\|e^{it\H}P_c\varphi\|_{L^2} \lesssim \|\varphi\|_{L^2}.
\end{equation}
\end{proposition}

\begin{remark} Throughout this section the parameter $\omega$ is fixed and will often be suppressed in formulas.  In Section~\ref{S:stability}, however, we will need to apply the dispersive estimates for varying choices of $\omega$.  The proof below will demonstrate that the implicit constants appearing in \eqref{dispersive1} and \eqref{dispersive2} depend certain norms of the potential $\mathcal{V}$, and in particular they depend continuously on the parameter $\omega$.\end{remark}

To prove Proposition~\ref{P:D}, we will adapt the arguments of Goldberg--Schlag \cite{GS2} and Krieger--Schlag \cite[Sections~7--8]{KS}. The starting point for the argument is the following spectral resolution of the propagator (see \cite{GS,KS}): 
\begin{eqnarray}
\lefteqn{2\pi i\langle e^{it\H}{ P_c}\varphi,\psi\rangle} \label{rep-form}\\
& =&\biggl[\int_{-\infty}^{-\omega}+\int_{\omega}^\infty\biggr]e^{it\lambda}\big\langle \big\{(\H-\lambda-i0)^{-1}-(\H-\lambda+i0)^{-1}\big\}\varphi,\psi\big\rangle\,d\lambda 
\nonumber\\
& =& e^{it\omega}\int_0^\infty e^{\frac{it\xi^2}{2}}\xi\big\langle\big\{(\H-\tfrac12\xi^2-\omega+i0)^{-1}-(\H-\tfrac12\xi^2-\omega-i0)^{-1}\big\}\varphi,\psi\big\rangle\,d\xi 
\nonumber\\
\label{resolvent1}\\
& &+ e^{-it\omega}\int_0^\infty e^{-\frac{it\xi^2}{2}}\xi\big\langle\big\{(\H+\tfrac12\xi^2+\omega+i0)^{-1}-(\H+\tfrac12\xi^2+\omega-i0)^{-1}\big\}\varphi,\psi\big\rangle\,d\xi,
\nonumber\\
\label{resolvent2}
\end{eqnarray}
where we have changed variables via $\lambda=\tfrac12\xi^2+\omega$ on $(\omega,\infty)$ and $\lambda=-\tfrac12\xi^2-\omega$ on $(-\infty,-\omega$). To estimate these integrals, we split into low-energy and high-energy regions, separated by a threshold energy $\lambda=C\|\mathcal{V}\|_{L^1}$ for some sufficiently large $C$.  %In the low-energy regime, we will work with an expression for the integral kernel of the resolvent built out of solutions to the associated generalized eigenvalue problem.  In the high-energy regime, we will estimate using a Born series expansion of the resolvent, relying essentially on integration by parts. 
%The rest of this section will be organized as follows:
We will first compute the integral kernel of the resolvent and the spectral measure by using solutions to the associated generalized eigenvalue problem (i.e. by means of the distorted Fourier transform).  Using this representation, we will then estimate the contribution of the low-energy regime to the dispersive estimates \eqref{dispersive1} and \eqref{dispersive2}.  Finally, we will estimate the contribution of the high-energy regime to these estimates by utilizing instead the Born series expansion for the resolvent.

%%%%%%%%%%%%%%%%%%%%%%%%%%%%%%%%%%%%%
\subsection{Representation formula}
In this section, we compute the integral kernel of the resolvent in terms of solutions to the generalized eigenvalue problem.  We follow the general strategy set out in \cite{KS}.  We consider the generalized eigenvalue problem
\begin{equation}\label{gep}
\H f = (\tfrac12\xi^2 + \omega)f
\end{equation}
and begin with some basic ODE considerations (see Appendix~\ref{S:F1F4} for more details).  First, for $\C^2$-valued functions $f$ and $g$, we define the Wronskian
\[
W[f,g](x) := {}^tf'(x)g(x) - {}^tf(x) g'(x)\in\C,
\]
and similarly for matrix-valued solutions $F,G$ to \eqref{gep} we define the Wronskian
\[
\mathcal{W}[F,G](x) = {}^tF'(x)G(x) - {}^tF(x)G'(x)\in \C^{2\times 2}.
\]
The Wronskian of two solutions to \eqref{gep} is independent of $x$. For the vector-valued case, one shows using \eqref{gep} that $\tfrac{d}{dx}W[f,g]=0$ for $x\neq 0$, so that the Wronskian is constant on $(-\infty,0)$ and on $(0,\infty)$.  The jump condition guarantees the two constants are the same.  For the matrix-valued case, one can then follow the argument of \cite[Lemma~5.10]{KS}. 

We next introduce some particular solutions to \eqref{gep}.  

\begin{lemma}\label{L:f1-f4} The problem \eqref{gep} admits four solutions $f_1,f_2,f_3,\tilde f_4$ obeying the following. Defining
\[
\mu=\sqrt{\xi^2 + 4\omega},
\]
there exists $\gamma>0$ such that for $x\geq 0$, $\ell\in\{0,1,2,\dots\}$, and $k\in\{0,1\}$, we have
\begin{align}
\biggl| \partial_\xi^\ell \partial_x^k\biggl\{ e^{-ix\xi}f_1(x,\xi)-\left[\begin{array}{r}1 \\ 0 \end{array}\right] \biggr\}\biggr| &\lesssim 
{\mu^{-1}}\langle x\rangle^{\ell+1}e^{-\min\{\gamma,\mu\}x}, 
\label{f1-bd} \\
\biggl| \partial_\xi^\ell \partial_x^k\biggl\{ e^{ix\xi}f_2(x,\xi)-\left[\begin{array}{r}1 \\ 0 \end{array}\right] \biggr\}\biggr| &\lesssim {
\mu^{-1}}\langle x\rangle^{\ell+1}e^{-\min\{\gamma,\mu\}x}, \label{f2-bd} \\
\biggl| \partial_\xi^\ell \partial_x^k\biggl\{ e^{\mu x}f_3(x,\xi)-\left[\begin{array}{r}0 \\ 1 \end{array}\right] \biggr\}\biggr| &\lesssim 
\begin{cases}
\mu^{-1}e^{-\gamma x}\qquad(\ell=k=0), \\
\mu^{-\ell-1}e^{-2\gamma x}\qquad(\ell+k\geq1),
\end{cases}
\label{f3-bd} \\
\biggl| \partial_\xi^\ell \partial_x^k\biggl\{ e^{-\mu x}\tilde f_4(x,\xi)-\left[\begin{array}{r}0 \\ 1 \end{array}\right] \biggr\}\biggr| &\lesssim 
{
\begin{cases}
\mu^{-1}\langle x\rangle^{2}e^{-\min\{\gamma,\mu\}x}
\qquad(\ell=0,k\geq 0), \\
\mu^{-2} \Jbr{x}^{2+\ell} e^{-\min\{\gamma,\mu\}}\qquad(\ell\geq 1, k\geq 0).
\end{cases}}
\nonumber
\end{align}
The implicit constants above depend continuously on $\omega$. Furthermore,
\begin{align*}
&f_1(x,-\xi)=\overline{f_1(x,\xi)}=f_2(x,\xi),\quad f_2(x,-\xi)=\overline{f_2(x,\xi)}=f_1(x,\xi), \\
&f_3(x,-\xi)=\overline{f_3(x,\xi)}=f_3(x,\xi),\quad \tilde f_4(x,-\xi)=\overline{\tilde f_4(x,\xi)}=\tilde f_4(x,\xi),
\end{align*}
and
\[
W[f_1,f_2]=2i\xi,\quad W[f_1,f_3]=W[f_2,f_3]=0,\qtq{and} W[f_3,\tilde f_4]=-2\mu. 
\]
\end{lemma}

The construction of these solutions runs parallel to the constructions appearing in \cite[Lemmas~5.2, 5.3, 5.5]{KS}, with the main differences arising from the presence of the jump condition at $x=0$.  For the sake of completeness, we have included the details in Appendix~\ref{S:F1F4}.  With Lemma~\ref{L:f1-f4} in hand, we may proceed towards our main goal.

\begin{definition} We define
\[
f_4(x,\xi)=\tilde f_4(x,\xi)-c_1(\xi)f_1(x,\xi)-c_2 f_2(x,\xi)
\]
for $\xi\neq 0$, where 
\[
c_1(\xi)=-\tfrac{1}{2i\xi}W[f_2,\tilde f_4]\qtq{and}c_2(\xi)=\tfrac{1}{2i\xi}W[f_1,\tilde f_4].
\]
\end{definition}

\begin{remark} The solutions $f_1,\dots,f_4$ obey the following:
\[
W[f_1,f_2]=2i\xi,\quad W[f_3,f_4]=-2\mu,
\]
and 
\[
W[f_1,f_3]=W[f_1,f_4]=W[f_2,f_3]=W[f_2,f_4]=0. 
\]
\end{remark}

\begin{definition} As the potential appearing in the operator $\mathcal{H}$ is even, we may introduce
\[
g_j(x,\xi):=f_j(-x,\xi),\quad j\in\{1,2,3,4\},
\]
which are then solutions to \eqref{gep} with the same asymptotic behavior as $x\to-\infty$ as $f_j$ when $x\to+\infty$. 
\end{definition}

We next introduce some matrix-valued solutions to \eqref{gep}. 

\begin{definition} For $j\in\{1,2\}$, we define $F_j=F_j(x,\xi)$ and $G_j=G_j(x,\xi)$ via
\begin{align*}
F_1 = [f_1\ f_3],\quad F_2=[f_2 \ f_4],\quad G_1=[g_2\ g_4],\quad G_2=[g_1\ g_3].
\end{align*}
\end{definition}

The solutions $F_j$ and $G_j$ are connected, as the following lemma shows.  The proof is the same as the argument as in \cite[Lemma~5.14]{KS}.  

\begin{lemma}\label{L:AB-exist} For any $\xi\in\R$, there exist $2\times 2$ matrices $A(\xi)$ and $B(\xi)$ such that
\[
F_1(x,\xi)=G_1(x,\xi)A(\xi)+G_2(x,\xi)B(\xi).
\]
Furthermore, we have $A(-\xi)=\overline{A(\xi)}$, $B(-\xi)=\overline{B(\xi)}$, and
\[
G_2(x,\xi)=F_2(x,\xi)A(\xi)+F_1(x,\xi)B(\xi). 
\]
\end{lemma}
%
%\begin{lemma}\label{L:AB}
%In addition, for all $k\in\{1,2,\dots\}$, we have the following asymptotic behavior as $|\xi|\to\infty$: 
%\begin{align}
%&A(\xi)=I+\mathcal{O}(|\xi|^{-1})\qtq{and} \partial_\xi^k A(\xi)=\mathcal{O}(|\xi|^{-k-1}),\label{AB-behave1}\\
%&B(\xi) = \mathcal{O}(|\xi|^{-1})\qtq{and} \partial_\xi^k B(\xi)=\mathcal{O}(|\xi|^{-k-1}).\label{AB-behave2}
%\end{align}
%\end{lemma}

\begin{remark} The matrices $A$ and $B$ play a role similar to 
the transmission and reflection coefficients of $H$ (see \eqref{ABR} below, for example).
\end{remark}

We would like to see that $\{f_1,f_3,g_1,g_3\}$ is a basis for the set of solutions to \eqref{gep}.  With this in mind, we define the following.

\begin{definition} We set
\[
D(\xi):=\mathcal{W}[F_1,G_2](\xi). 
\]
\end{definition}

The next result provides the essential connection between the spectral analysis and the scattering theory of $\mathcal{H}$. 

\begin{proposition}\label{P:spectral-meets-scattering} For $\xi\neq 0$, the following are equivalent:
\begin{itemize}
\item[(i)] $\det D(\xi)=0$.
\item[(ii)] $\tfrac12\xi^2+\omega$ is an eigenvalue of $\H$.
\item[(iii)] $\det A(\xi)=0$.
\end{itemize}
For $\xi=0$, the following are equivalent:
\begin{itemize}
\item[(i')] $\det D(0)=0$.
\item[(ii')] $\omega$ is a resonance or eigenvalue of $\H$. 
\end{itemize}
\end{proposition}

\begin{proof} As the proofs are the same as in \cite[Lemma~5.17]{KS} and \cite[Lemma~5.20]{KS}, we will only recall a few of the main components of the proof. 

The equivalence of (i) and (iii) is a consequence of the 
fact that
\begin{equation}\label{DA-ID}
D(\xi)=\mathcal{W}[F_1,G_2]={}^t\mathcal{W}[F_1,G_2]=
\left[\begin{array}{rr} 
2i\xi & 0 \\ 0 & -2\mu \end{array}\right]A(\xi), 
\end{equation}
which can be shown exactly as in \cite[Lemma~5.14]{KS}. 

Let us prove (iii)$\implies$(ii). Assume that $\det A(\xi)=0$. 
Then, there exists $0\neq v(\xi)\in\C^2$ 
such that $A(\xi)v(\xi)=0$. 
By the same argument as that in  \cite[Lemma 5.16]{KS}, we have 
for $\xi\neq0$, 
\begin{eqnarray}
-2i\xi P&=&-2i\xi A(\xi)^{\ast}PA(\xi)
-2\mu A(\xi)^{\ast}QB(\xi)\nonumber\\
& &+2\mu B(\xi)^{\ast}QA(\xi)
+2i\xi B(\xi)^{\ast}PB(\xi),\label{ABR}
\end{eqnarray}
where 
\[
P=\left[\begin{array}{rr} 
1& 0 \\ 0 & 0\end{array}\right]\qtq{and} 
Q=\left[\begin{array}{rr} 
0& 0 \\ 0 & 1\end{array}\right].
\] 
From this identity, we see $-2i\xi\langle Pv(\xi),v(\xi)\rangle
=2i\xi|PB(\xi)v(\xi)|^2$. Since $\xi\neq0$, we find that 
$Pv(\xi)=PB(\xi)v(\xi)=0$. 
On the other hand, by Lemma \ref{L:AB-exist}, we see 
\begin{equation}\label{f1vg2bv}
F_1(x,\xi)v(\xi)=G_1(x,\xi)A(\xi)v(\xi)+G_2(x,\xi)B(\xi)v=G_2(x,\xi)B(\xi)v(\xi).
\end{equation}
Since $Pv(\xi)=0$ and $PB(\xi)v(\xi)=0$, the left-hand side of \eqref{f1vg2bv} belongs to $L^2(0,\infty)$ and the right-hand side of \eqref{f1vg2bv} belongs to $L^2(-\infty,0)$.  Thus $F_1(x,\xi)v(\xi)$ is a nontrivial $L^2$ solution of \eqref{gep}, which yields (ii). 

For the remaining implications, we again refer the reader to \cite{KS}. \end{proof}

In particular, in light of \eqref{DA-ID} and the assumption that $\omega$ satisfies the spectral condition in Definition~\ref{D:spectral}, we have the following.

\begin{corollary}\label{C:Dneq0} For all $\xi\in\R$, $\det D(\xi)\neq 0$.  Moreover, for any $\xi\in\R$, 
\[
A(\xi)^{-1} = D(\xi)^{-1}\left[\begin{array}{rr} 2i\xi & 0 \\ 0 & -2\mu \end{array}\right].
\]
\end{corollary} 

We now introduce the following matrix-valued functions.

\begin{definition}\label{def:FG} We define
\begin{align*}
\F(x,\xi) &= F_1(x,\xi)A(\xi)^{-1}\left[\begin{array}{r} 1 \\ 0 \end{array}\right] = 2i\xi F_1(x,\xi)D(\xi)^{-1}\left[\begin{array}{r} 1 \\ 0 \end{array}\right], \\
\G(x,\xi)&=G_2(x,\xi)A(\xi)^{-1}\left[\begin{array}{r} 1 \\ 0 \end{array}\right] = 2i\xi G_2(x,\xi)D(\xi)^{-1}\left[\begin{array}{r} 1 \\ 0 \end{array}\right]. 
\end{align*}
\end{definition}

These functions will be essential in establishing the representation formula in Proposition~\ref{P:Rep}.  We first show that bounded solutions to \eqref{gep} may be expressed in terms of $\F$ and $\G$.

\begin{lemma}\label{L:FGrep} If $f$ is a bounded solution to \eqref{gep}, then there exist $\alpha(\xi)$ and $\beta(\xi)$ such that
\[
f(x,\xi)=\alpha(\xi)\F(x,\xi)+\beta(\xi)\G(x,\xi)
\]
for $x\neq 0$.
\end{lemma}

\begin{proof} By Corollary~\ref{C:Dneq0} (which yields $\det D(\xi)\neq 0$), we have that $\{f_1,f_3,g_1,g_3\}$ is a basis for the set of solutions to \eqref{gep}.  Thus we may write $f$ in the form
\[
f(x,\xi)=F_1(x,\xi)\left[\begin{array}{r} d_1(\xi) \\ d_2(\xi) \end{array}\right] + G_2(x,\xi)\left[\begin{array}{r} d_3(\xi) \\ d_4(\xi)\end{array}\right]. 
\]
Combining this with Lemma~\ref{L:AB-exist} (i.e. expanding $F_1$ in terms of $G_1$ and $G_2$) and sending $x\to-\infty$, we can deduce (using the assumption that $f$ is bounded) that
\[
A(\xi)\left[\begin{array}{r} d_1(\xi) \\ d_2(\xi) \end{array}\right]=\alpha(\xi)\left[\begin{array}{r} 1 \\ 0 \end{array}\right]
\]
for some $\alpha(\xi)$.  Similarly (expanding $G_2$ in terms of $F_1$ and $F_2$ and sending $x\to\infty$) we can deduce
\[
A(\xi)\left[\begin{array}{r} d_3(\xi) \\ d_4(\xi)\end{array}\right] = \beta(\xi)\left[\begin{array}{r} 1 \\ 0 \end{array}\right]
\]
for some $\beta(\xi)$.  This implies the result. \end{proof}

We will also need to understand the asymptotic behavior of $\F$ and $\G$. The starting point is to understand the asymptotic behavior of the matrices $A(\xi)$ and $B(\xi)$.

\begin{lemma}\label{L:AB}
For all $k\in\{1,2,\dots\}$, we have the following as $|\xi|\to\infty$: 
\begin{align}
&A(\xi)=I+\mathcal{O}(|\xi|^{-1})\qtq{and} \partial_\xi^k A(\xi)=
%\mathcal{O}(|\xi|^{-k-1})
{\mathcal{O}(|\xi|^{-1})},\label{AB-behave1}\\
&B(\xi) = \mathcal{O}(|\xi|^{-1})\qtq{and} \partial_\xi^k B(\xi)=
%=\mathcal{O}(|\xi|^{-k-1})
{\mathcal{O}(|\xi|^{-1})},\label{AB-behave2}
\end{align}
where the implicit constants depend continuously on $\omega$.\end{lemma}

\begin{proof} The key is to analyze $D(\xi)=\mathcal{W}[F_1,G_2](\xi)$, which is connected to the matrix $A(\xi)$ via \eqref{DA-ID} above.  Similarly, analysis of the Wronskian $\mathcal{W}[F_1,G_1](\xi)$ provides the desired information about $B(\xi)$.

We begin by writing
\[
\mathcal{W}[F_1,G_2] = \left[\begin{array}{rr} W[f_1,g_1] & W[f_1,g_3] \\ W[f_3,g_1] & W[f_3,g_3]\end{array}\right]. 
\] 
We now utilize the estimates appearing in Lemma~\ref{L:f1-f4} (at $x=0$) to derive that
\[
W[f_1,g_1]=2i\xi+\mathcal{O}(1),\quad 
\partial_\xi W[f_1,g_1]=%2i+\mathcal{O}(\mu^{-1})
{\mathcal{O}(1)},\quad 
\partial_\xi^j W[f_1,g_1]= %\mathcal{O}(\mu^{-j})
{\mathcal{O}(1)}
\]
for $j\geq 2$.  Similarly we obtain the estimates
\begin{align*}
&W[f_1,g_3](\xi)=\mathcal{O}(1),\quad \partial_\xi^j 
W[f_1,g_3](\xi)=%\mathcal{O}(\mu^{-j}) 
{\mathcal{O}(1)}\qtq{for}j\geq 1,\\
&W[f_3,g_1](\xi)=\mathcal{O}(1),\quad 
\partial_\xi^j W[f_3,g_1](\xi)=%\mathcal{O}(\mu^{-j})
{\mathcal{O}(1)}\qtq{for}j\geq 1, \\
&W[f_3,g_3]=-2\mu+\mathcal{O}(1),\quad 
\partial_\xi^jW[f_3,g_3] = \partial_\xi^j(-2\mu)+\mathcal{O}(\mu^{-j})
\qtq{for}j\geq 1,
\end{align*}
where we recall $\mu=\sqrt{\xi^2+4\omega}$.

Putting all of the pieces together, we have
\begin{equation}\label{determinant-asymptotic}
\mathcal{W}[F_1,G_2]=\left[\begin{array}{rr} 2i\xi & 0 \\ 0 & -2\mu\end{array}\right]  + \mathcal{O}(1),
\end{equation}
and additionally
\[
\partial_\xi^j \mathcal{W}[F_1,G_2] =% \partial_\xi^j \left[\begin{array}{rr} 2i\xi & 0 \\ 0 & -2\mu\end{array}\right]  + \mathcal{O}(\mu^{-j})
{\mathcal{O}(1)}
\]
for $j\geq 1$. Combining this with \eqref{DA-ID}, we deduce that $A(\xi)=I+\mathcal{O}(\mu^{-1})$.  The fact that we can differentiate the approximation above then allows us to deduce $\partial_\xi^k A(\xi)=%\mathcal{O}(|\xi|^{-k-1})
{\mathcal{O}(|\xi|^{-1})}$, as desired. 

A similar argument based on the fact that
\[
\mathcal{W}[F_1,G_1](\xi)=\left[\begin{array}{rr} W[f_1,g_2] & W[f_1,g_4] \\ W[f_3,g_2]&W[f_3,g_4]\end{array}\right]=-{}^t B(\xi)\left[\begin{array}{rr} 2i\xi & 0 \\ 0 & -2\mu\end{array}\right]
\]
yields to the desired estimates for $B$.  For the second equality above, we refer the reader to \cite[Lemma~5.14]{KS}. \end{proof}

We turn to the behavior of $\F$ and $\G$. We first introduce two scalar functions $\tilde T(\xi)$ and $\tilde R(\xi)$ (the transmission and reflection coefficients).

\begin{definition} We define $\tilde T(\xi)$ and $\tilde R(\xi)$ by imposing
\begin{equation}\label{Def:TRtilde}
\begin{aligned}
2i\xi\left[\begin{array}{rr} 1 & 0 \\ 0 & 0 \end{array}\right]D(\xi)^{-1}\left[\begin{array}{r} 1 \\ 0 \end{array}\right]& =:\left[\begin{array}{c} \tilde T(\xi) \\ 0\end{array}\right], \\
 2i\xi\left[\begin{array}{rr} 1 & 0 \\ 0 & 0 \end{array}\right]B(\xi)D(\xi)^{-1}\left[\begin{array}{r} 1 \\ 0 \end{array}\right] & =: \left[\begin{array}{c} \tilde R(\xi) \\ 0 \end{array}\right]. 
\end{aligned}
\end{equation}
\end{definition}

\begin{remark}[Properties of $\tilde T$ and $\tilde R$]\label{R:TR}  Using \eqref{DA-ID} and Lemma~\ref{L:AB}, we may deduce that $\tilde T$ and $\tilde R$, along with their derivatives, are bounded functions.\end{remark}

\begin{lemma}\label{L:FG-behave} For $\ell\in\{0,1,2,\dots\}$, we have
\begin{align*}
\partial_\xi^\ell \biggl(\F(x,\xi) -\tilde T(\xi)e^{ix\xi}\left[\begin{array}{r} 1 \\ 0 \end{array}\right]\biggr)&\lesssim\mu^{-1} x^{\ell+1}e^{-\min\{\gamma,\mu\}x})\qtq{as}x\to\infty,\\
\partial_\xi^\ell \biggl(\F(x,\xi)-\biggl(e^{ix\xi}+\tilde R(\xi)e^{-ix\xi}\bigg)\left[\begin{array}{r} 1 \\ 0 \end{array}\right]\biggr)&\lesssim \mu^{-1}|x|^{\ell+1}e^{\min\{\gamma,\mu\}x}\qtq{as}x\to-\infty, \\
\partial_\xi^\ell \biggl(\G(x,\xi)-\tilde T(\xi)e^{-ix\xi}\left[\begin{array}{r} 1 \\ 0 \end{array}\right]\biggr) &\lesssim\mu^{-1}|x|^{\ell+1}e^{\min\{\gamma,\mu\}x}\qtq{as}x\to-\infty,\\
\partial_\xi^\ell \biggl(\G(x,\xi)-\biggl(e^{-ix\xi}+\tilde R(\xi)e^{ix\xi}\biggr)\left[\begin{array}{r} 1 \\ 0 \end{array}\right]\biggr)&\lesssim\mu^{-1}x^{\ell+1}e^{-\min\{\gamma,\mu\}x}\qtq{as}x\to\infty,
\end{align*}
with implicit constants depending continuously on $\omega$.
\end{lemma}

\begin{proof} We focus on describing $\F$ in the region $x\geq 0$.  Similar arguments treat the remaining cases.

By definition, we may write
\begin{align*}
\F & = 2i\xi[f_1\ 0]D(\xi)^{-1}\left[\begin{array}{r} 1 \\ 0 \end{array}\right] + 2i\xi[0\ f_3]D(\xi)^{-1}\left[\begin{array}{r} 1 \\ 0 \end{array}\right]. 
\end{align*}
We now insert the approximations to $f_1$ and $f_3$ from Lemma~\ref{L:f1-f4}.  The contribution of $f_1$ is then
\begin{equation}\label{FFF1}
\tilde T(\xi)e^{ix\xi}\left[\begin{array}{r} 1 \\ 0 \end{array}\right] + 2i\xi e^{ix\xi}\left[ J_1 \begin{array}{r} 0 \\ 0 \end{array}\right]D(\xi)^{-1}\left[\begin{array}{r} 1 \\ 0 \end{array}\right],
\end{equation}
where 
\[
J_1(x,\xi)=e^{-ix\xi}f_1(x,\xi)-\left[\begin{array}{r} 1 \\ 0 \end{array}\right].
\]
Similarly, the contribution of $f_3$ is 
\begin{equation}\label{FFF2}
2i\xi e^{-\mu x}\left[\begin{array}{rr} 0 & 0 \\ 0 & 1 \end{array}\right]D(\xi)^{-1}\left[\begin{array}{r} 1 \\ 0 \end{array}\right] + 2i\xi e^{-\mu x}\left[\begin{array}{r} 0 \\ 0 \end{array} J_3\right] D(\xi)^{-1} \left[\begin{array}{r} 1 \\ 0 \end{array}\right],
\end{equation}
where 
\[
J_3(x,\xi)=e^{\mu x}f_3(x,\xi)-\left[\begin{array}{r} 0 \\ 1 \end{array}\right].
\]  

Our task is to estimate the second term in \eqref{FFF1} and both terms in \eqref{FFF2}.  The key ingredients are the estimates in Lemma~\ref{L:f1-f4}, as well as Lemma~\ref{L:AB} (where we recall the connection between $A(\xi)$ and $D(\xi)$ given by Corollary~\ref{C:Dneq0}).  As the estimates are all similar, let us only present just the estimate of the second term in \eqref{FFF1}. 

Given $\ell\geq 0$, we estimate 
\begin{align*}
\sum_{i+j+k=\ell}&  \bigl| \partial_\xi^i e^{ix\xi}\bigr|\,\biggl| \left[\partial_\xi^j J_1 \begin{array}{r} 0 \\ 0 \end{array}\right]\biggr| \,\biggl| \partial_\xi^k[2i\xi D(\xi)^{-1}]\left[\begin{array}{r} 1 \\ 0 \end{array}\right]\biggr| \\
%& \lesssim \sum_{i+j+k=\ell} |x|^i \mu^{-j-1}\langle x\rangle^{j+1}e^{-\min\{\gamma,\mu\}x} \langle \xi\rangle^{-k} \\
& {
\lesssim \sum_{i+j+k=\ell} |x|^i \mu^{-1}
\langle x\rangle^{j+1}e^{-\min\{\gamma,\mu\}x}}%\\
%& 
\lesssim \mu^{-1}\langle x\rangle^{\ell+1} e^{-\min\{\gamma,\mu\}x},
\end{align*}
as desired. \end{proof}

With the particular solutions constructed above, we can construct the integral kernel of the resolvent:
\begin{lemma}\label{L:ResKer} For $\xi\neq 0$, 
\begin{align*}
[\H-\tfrac12\xi^2-\omega-i0]^{-1}(x,y) & = \begin{cases} -2F_1(x,\xi)D^{-1}(\xi){}^tG_2(y,\xi)\sigma_3 & x\geq y, \\ -2G_2(x,\xi)D^{-1}(\xi){}^tF_1(y,\xi)\sigma_3 & x\leq y, \end{cases}\\
[\H-\tfrac12\xi^2-\omega+i0]^{-1}(x,y) & = \begin{cases} -2F_1(x,-\xi)D^{-1}(-\xi){}^tG_2(y,-\xi)\sigma_3 & x\geq y, \\ -2G_2(x,-\xi)D^{-1}(-\xi){}^t F_1(y,-\xi)\sigma_3 & x \leq y.\end{cases}
\end{align*}
\end{lemma}

\begin{proof} It suffices to prove the first identity.  We introduce $h=[\H-\tfrac12\xi^2-\omega-i0]^{-1}\varphi$, i.e. the solution to $[\H-\tfrac12\xi^2-\omega-i0]h=\varphi$.  By variation of parameters, we may write
\[
h(x,\xi)=F_1(x,\xi)\left[\begin{array}{r} d_1(x,\xi) \\ d_2(x,\xi)\end{array}\right] + G_2(x,\xi)\left[\begin{array}{r} d_3(x,\xi) \\ d_4(x,\xi)\end{array}\right],
\]
where (denoting $\tfrac{d}{dx}$ by $'$)
\[
F_1(x,\xi)\left[\begin{array}{r} d_1'(x,\xi) \\ d_2'(x,\xi)\end{array}\right] + G_2(x,\xi)\left[\begin{array}{r} d_3'(x,\xi) \\ d_4'(x,\xi)\end{array}\right] = \left[\begin{array}{r} 0 \\ 0 \end{array}\right]
\]
and
\[
\left[\begin{array}{rr} F_1(x,\xi) & G_2(x,\xi) \\ F_1'(x,\xi) & G_2'(x,\xi)\end{array}\right]\left[\begin{array}{r} d_1'(x,\xi) \\ d_2'(x,\xi)\\ \hline  d_3'(x,\xi) \\ d_4'(x,\xi)\end{array}\right] = \left[\begin{array}{c} 0 \\ 0 \\ \hline -2\varphi_1 \\ 2\varphi_2\end{array}\right]=\left[\begin{array}{c} 0 \\ 0 \\ \hline -2\sigma_3\varphi \end{array}\right]
\]
As $\det D(\xi)\neq 0$, we can invert the matrix appearing on the left (cf. \cite[(5.22)]{KS}), which leads to
\begin{align*}
\left[\begin{array}{r} d_1'(x,\xi) \\ d_2'(x,\xi)\\ \hline  d_3'(x,\xi) \\ d_4'(x,\xi)\end{array}\right] & = \left[\begin{array}{rr} -{}^t D^{-1}(\xi)G_2'(x,\xi) & {}^t D^{-1}(\xi){}^t G_2(x,\xi) \\ {}^t D^{-1}(\xi) {}^t F_1'(x,\xi) & -{}^t D^{-1}(\xi){}^t F_1(x,\xi) \end{array}\right]\left[\begin{array}{c} 0 \\ 0 \\ \hline -2\sigma_3\varphi \end{array}\right] \\
& = \left[\begin{array}{r} -2{}^t D^{-1}(\xi) {}^tG_2(x,\xi)\sigma_3 \varphi \\ 2{}^tD^{-1}(\xi){}^t F_1(x,\xi)\sigma_3 \varphi\end{array}\right]. 
\end{align*}
In particular, integrating in $x$ yields
\begin{align*}
\left[\begin{array}{r} d_1(x,\xi)\\d_2(x,\xi)\end{array}\right]&=-2{}^tD^{-1}(\xi)\int_{-\infty}^x {}^tG_2(y,\xi)\sigma_3\varphi(y)\,dy, \\
\left[\begin{array}{r} d_3(x,\xi)\\d_4(x,\xi)\end{array}\right]&=-2{}^tD^{-1}(\xi)\int_{x}^\infty {}^tF_1(y,\xi)\sigma_3\varphi(y)\,dy.
\end{align*} 
This implies the desired identity. \end{proof}

With the integral kernel for the resolvent in place, we can now compute the spectral measure.

\begin{definition} We define $E=E(x,\xi)$ by
\[
E = [\F\ \G],
\]
where $\F$ and $\G$ are as in Definition~\ref{def:FG}.
\end{definition}

\begin{lemma}\label{L:SM} For $\xi\in\R$,
\begin{equation}\label{SM1}
[\H-\tfrac12\xi^2-\omega-i0]^{-1}(x,y)-[\H-\tfrac12\xi^2-\omega+i0]^{-1}(x,y)=-\tfrac{1}{i\xi}E(x,\xi)E(y,\xi)^*\sigma_3,
\end{equation}
where $\sigma_3$ is as in \eqref{pauli}. 
\end{lemma}

\begin{proof} We proceed as in \cite[Lemma~6.7]{KS}. We set
\[
K(x,y,\xi)=\text{LHS}\eqref{SM1}.
\]
Now, for fixed $y$, $K(x,y,\xi)$ is a bounded solution to \eqref{gep}, so that by Lemma~\ref{L:FGrep} we may write
\[
K(x,y,\xi)=E(x,\xi)M(y,\xi)\qtq{for some }M\in\C^{2\times 2}. 
\]
Using Lemma~\ref{L:ResKer}, we also see that
\[
\sigma_3K(x,y,\xi)^*\sigma_3 = -K(y,x,\xi),
\]
which (after some rearranging) implies
\[
C(\xi):=E(y,\xi)^{-1}\sigma_3 M(y,\xi)^*=-M(x,\xi)\sigma_3[E(x,\xi)^*]^{-1}. 
\]
Observing that $C(\xi)^*=-C(\xi)^{-1}$, we see that $C(\xi)$ takes the form
\[
C(\xi)=\left[\begin{array}{rr} i\alpha & \beta \\ -\bar\beta & i\eps\end{array}\right]
\]
for some $\alpha,\eps\in\R$ and $\beta\in\C$. We therefore obtain
\begin{equation}\label{KECE}
K(x,y,\xi)=-E(x,\xi)C(\xi)E(y,\xi)^*\sigma_3. 
\end{equation}

To complete the proof we must show $C(\xi)=-\tfrac{1}{i\xi}I$, i.e. $\alpha=\eps=-\tfrac{1}{\xi}$ and $\beta=0$.  To this end, we will use two different expressions for $K(x,y,\xi)$ and send $x\to\infty$ and $y\to-\infty$.  First, using Lemma~\ref{L:ResKer} and \eqref{Def:TRtilde} in the original definition of $K$, we deduce that 
\begin{align*}
K(x,y,\xi) & = \left[\tfrac{1}{i\xi}\tilde T(\xi)e^{i(x-y)\xi}+\tfrac{1}{i\xi}\tilde T(-\xi)e^{-i(x-y)\xi}\right]\left[\begin{array}{rr} 1  & 0 \\ 0 & 0 \end{array}\right] +o(1)
\end{align*}
as $x\to\infty$ and $y\to-\infty$.  On the other hand, using \eqref{KECE} and the asymptotic behavior given in Lemma~\ref{L:FG-behave}, we deduce
\begin{align*}
K(x,y,\xi) & = -\biggl[\left\{\beta |\tilde T(\xi)|^2-\bar\beta |\tilde R(\xi)|^2+i\alpha \tilde T(\xi)\tilde R(-\xi)+i\eps \tilde T(-\xi)\tilde T(\xi)\right\} e^{i(x+y)\xi} \\
& \quad\quad -\bar\beta e^{-i(x+y)\xi} + \left\{i\alpha \tilde T(\xi)-\bar\beta\tilde R(\xi)\right\}e^{i(x-y)\xi} \\
& \quad\quad +\left\{i\eps\tilde T(-\xi)-\bar\beta \tilde R(-\xi)\right\} e^{-i(x-y)\xi} \biggr]\left[\begin{array}{rr} 1 & 0 \\ 0 & 0 \end{array}\right] + o(1)
\end{align*}
as $x\to\infty$ and $y\to-\infty$.  It follows that $\beta=0$, while $\alpha=\eps=-\tfrac{1}{\xi}$, as desired.  
\end{proof}

Finally, we can obtain the following representation formula.

\begin{definition} We define $\F_\pm(x,\xi)$ and $\G_\pm(x,\xi)$ via
\begin{align*}
\F_+=\F,&\quad F_-=\sigma_1\F, \\
\G_+=\G,&\quad \G_-=\sigma_1\G,
\end{align*}
where $\sigma_1$ is as in \eqref{pauli}, and we set
\[
e_\pm(x,\xi) = \begin{cases} \F_{\pm}(x,\xi) & \xi\geq 0, \\ \G_{\pm}(x,-\xi) & \xi\leq 0. \end{cases}
\]
\end{definition}

\begin{proposition}[Representation formula]\label{P:Rep} With $e_\pm=e_\pm(x,\xi)$ as above, we have 
\begin{equation}\label{E:rep-form}
\begin{aligned}
\langle e^{it\H}P_c \varphi,\psi\rangle & =\tfrac{1}{2\pi}e^{it\omega}\int e^{\frac{it\xi^2}{2}}\langle \varphi,\sigma_3 e_+(\cdot,\xi)\rangle\overline{\langle \psi,e_+(\cdot,\xi)\rangle}\,d\xi \\
& \quad + \tfrac{1}{2\pi} e^{-it\omega}\int e^{-\frac{it\xi^2}{2}}\langle \varphi,\sigma_3 e_-(\cdot,\xi)\rangle\overline{\langle \psi,e_-(\cdot,\xi)\rangle}\,d\xi.
\end{aligned}
\end{equation}
%with
%\begin{align}
%\|e_\pm(x,\cdot)\|_{W_{\xi}^{j,\infty}} &\lesssim \langle x\rangle^j,\quad j\in\{0,1,2,\dots\}, \label{epm-1} \\
%\|\langle \varphi,\sigma_3 e_\pm(\cdot,\xi)\rangle\|_{W_\xi^{2,\infty}} &\lesssim \|\langle x\rangle^2 \varphi\|_{L^1}, \label{epm-2}\\
%e_\pm(x,0)&=0 \qtq{for}x\neq 0. \label{epm-3}
%\end{align}
\end{proposition}

\begin{proof}[Proof of Proposition~\ref{P:Rep}] Using Lemma~\ref{L:SM}, we write
\begin{align*}
[\H-&\tfrac12\xi^2-\omega-i0]^{-1}(x,y)-[\H-\tfrac12\xi^2-\omega+i0]^{-1}(x,y) \\
& = -\tfrac{1}{i\xi}[e_+(x,\xi)\ e_+(x,-\xi)][e_+(y,\xi)\ e_+(y,-\xi)]^*\sigma_3.
\end{align*}
We also claim that
\begin{equation}\label{with-plusses}
\begin{aligned}
&[\H+\tfrac12\xi^2+\omega-i0]^{-1}(x,y)-[\H+\tfrac12\xi^2+\omega+i0]^{-1}(x,y) \\
&\quad = -\tfrac{1}{i\xi}[e_-(x,\xi)\ e_-(x,-\xi)][e_-(y,\xi)\ e_-(y,-\xi)]^*\sigma_3.
\end{aligned}
\end{equation}
Inserting these relations into the expression above and simplifying then leads to the desired formula \eqref{E:rep-form}.  To derive \eqref{with-plusses}, we again use Lemma~\ref{L:SM}, along with the identity
\[
\H+\tfrac12\xi^2 + \omega = -\sigma_1[H-\tfrac12\xi^2 - \omega]\sigma_1.
\]
Inserting these identities for the integral kernel into \eqref{resolvent1} and \eqref{resolvent2} now leads to the representation formula. \end{proof}

\begin{remark} The $L^2\to L^2$ estimate \eqref{dispersive3} now follows from Proposition~\ref{P:Rep} and Lemma~\ref{L:FG-behave}. In the next sections, we will prove \eqref{dispersive1} and \eqref{dispersive2} by considering separately the low-energy and high-energy contributions.
\end{remark}

\subsection{Estimates for low energies}

In this section, we estimate the contribution of the low energy regime to the dispersive estimates \eqref{dispersive1} and \eqref{dispersive2}.   To implement the cutoff to low and high frequencies, we let $\chi$ be a smooth, even bump function such that $\chi(\xi)=1$ for $|\xi|\leq\lambda$, and  $\chi(\xi)=0$ for $|\xi|\ge2\lambda$.  We will choose $\lambda=C\|\mathcal{V}\|_{L^1}$ for some sufficiently large $C$. 

Recalling the representation formula \eqref{rep-form} and Lemma~\ref{L:ResKer}, we have the following: Let $P_c^+$ and $P_c^-$ be projections to positive part and negative part of the absolutely continuous 
spectrum of ${{\mathcal H}}$. Then 
\begin{eqnarray}
\lefteqn{2\pi i\langle e^{it{{\mathcal H}}}\chi({{\mathcal H}}-\omega)
P_c^+\varphi,\psi\rangle}
\nonumber\\
&=&
e^{it\omega}\int_{0}^{\infty}
e^{\frac{i}{2}t\xi^{2}}\xi\chi(\xi^2)
\langle\{({{\mathcal H}}-\tfrac{\xi^{2}}{2}-\omega-i0)^{-1}
-({{\mathcal H}}-\tfrac{\xi^{2}}{2}-\omega+i0)^{-1}\}\varphi,\psi\rangle 
d\xi\nonumber\\
&=&
-2e^{it\omega}\int\!\int_{x\geq y}
\int_{0}^{\infty}
e^{\frac{i}{2}t\xi^{2}}\xi\chi(\xi^2)
F_1(x,\xi)D^{-1}(\xi){}^tG_2(y,\xi)\sigma_3\varphi(y)\overline{\psi}(x)
\,d\xi\, dx\,dy\nonumber\\
& &
+2e^{it\omega}\int\!\int_{x\leq y}
\int_{0}^{\infty}
e^{\frac{i}{2}t\xi^{2}}\xi\chi(\xi^2)
G_2(x,\xi)D^{-1}(\xi){}^tF_1(y,\xi)\sigma_3\varphi(y)\overline{\psi}(x)
\,d\xi\,dx\,dy,\nonumber\\
\label{low}
\end{eqnarray}
with a similar formula for $e^{it{{\mathcal H}}}\chi({{\mathcal H}}-\omega)P_c^-$. 

Because the functions $F_j$, $G_j$, and $D^{-1}$ have the same decay/regularity properties as the analogous functions appearing in \cite{KS}, we can follow much of the presentation in \cite{KS}.  Accordingly, we shall keep our presentation rather brief. 

It is sufficient to consider only first term on the right-hand side of \eqref{low}.  We split the region of the integral into the three subregions: $x\geq y\geq 0$, $x\geq 0\geq y$ and $0\geq x\geq y$.  By symmetry, it suffices to consider the first two regions.

For the case $x\geq 0\geq y$, we split the $\xi$ integral as follows: 
\begin{eqnarray}
\lefteqn{\int_{0}^{\infty}
e^{\frac{i}{2}t\xi^{2}}\xi\chi(\xi^2)
F_1(x,\xi)D^{-1}(\xi){}^tG_2(y,\xi)\sigma_3d\xi}\nonumber\\
&=&
\int_{0}^{\infty}
e^{\frac{i}{2}t\xi^{2}}\xi\chi(\xi^2)
\left[f_1(x,\xi)  \begin{array}{cc}0\\0\end{array}\right]
D^{-1}(\xi)
\left[\begin{array}{cc}{}^tf_1(-y,\xi)  \\0\ \ \ 0\end{array}\right]
\sigma_3d\xi\nonumber\\
& &
+\int_{0}^{\infty}
e^{\frac{i}{2}t\xi^{2}}\xi\chi(\xi^2)
\left[f_1(x,\xi)  \begin{array}{cc}0\\0\end{array}\right]
D^{-1}(\xi)
\left[\begin{array}{cc}0\ \ \ 0  \\ {}^tf_3(-y,\xi)\end{array}\right]
\sigma_3d\xi\nonumber\\
& &+\int_{0}^{\infty}
e^{\frac{i}{2}t\xi^{2}}\xi\chi(\xi^2)
\left[ \begin{array}{cc}0\\0\end{array}f_3(x,\xi)\right]
D^{-1}(\xi)
\left[\begin{array}{cc}{}^tf_1(-y,\xi)  \\0\ \ \ 0\end{array}\right]
\sigma_3d\xi\nonumber\\
& &+\int_{0}^{\infty}
e^{\frac{i}{2}t\xi^{2}}\xi\chi(\xi^2)
\left[\begin{array}{cc}0\\0\end{array}f_3(x,\xi)\right]
D^{-1}(\xi)
\left[\begin{array}{cc}0\ \ \ 0  \\ {}^tf_3(-y,\xi)\end{array}\right]
\sigma_3d\xi\nonumber\\
&=:&I_1(x,y)+I_2(x,y)+I_3(x,y)+I_4(x,y).
\nonumber
\end{eqnarray}
Let us show how to estimate $I_1$; the remaining terms are similar.  Writing $\tilde{\chi}$ for a cut-off function such that $\tilde{\chi}\chi=\chi$, we have
\begin{eqnarray*}
\lefteqn{|I_1(x,y)|}\\
&\lesssim&t^{-\frac12}
\int_{-\infty}^{\infty}
|[\xi\chi(\xi^2)D^{-1}(\xi)]\check{{}\ }\ast
[\tilde{\chi}(\xi^2)e^{-ix\xi}f_1(x,\xi)]\check{{}\ }
\ast[\tilde{\chi}(\xi^2)e^{iy\xi}f_1(-y,\xi)]\check{{}\ }
(\eta)|d\eta\\
&\lesssim&t^{-\frac12}
\|[\xi\chi(\xi^2)D^{-1}(\xi)]\check{{}\ }\|_{L^1}
\|[\tilde{\chi}(\xi^2)e^{-ix\xi}f_1(x,\xi)]\check{{}\ }\|_{L^1}
\|[\tilde{\chi}(\xi^2)e^{iy\xi}f_1(-y,\xi)]\check{{}\ }\|_{L^1}.
\end{eqnarray*}

Using Lemmas~\ref{L:f1-f4}~and~\ref{L:AB}, we deduce that
\begin{eqnarray*}
\|[\xi\chi(\xi^2)D^{-1}(\xi)]\check{{}\ }\|_{L^1}\lesssim 1
\end{eqnarray*}
and
\begin{eqnarray*}
\qquad\sup_{x\geq 0}\|[\tilde{\chi}(\xi^2)e^{-ix\xi}f_1(x,\xi)]\check{{}\ }\|_{L^1}
\lesssim
\sup_{x\geq 0}\|(1-\pt_{\xi}^2)
\{\tilde{\chi}(\xi^2)e^{-ix\xi}f_1(x,\xi)\}\|_{L^2}
\lesssim 1.
\end{eqnarray*}
Similarly, 
\begin{eqnarray*}
\sup_{y\le0}\|[\tilde{\chi}(\xi^2)e^{iy\xi}f_1(-y,\xi)]\check{{}\ }\|_{L^1}\lesssim 1.
\end{eqnarray*}
Hence we have 
\begin{eqnarray*}
\sup_{x\geq 0, y\le0}|I_1(x,y)|\lesssim t^{-\frac12}.
\end{eqnarray*}

For the case $0\geq x\geq y$, using Lemma~\ref{L:AB-exist}, we find
\begin{eqnarray*}
\lefteqn{\int_{0}^{\infty}
e^{\frac{i}{2}t\xi^{2}}\xi\chi(\xi^2)
F_1(x,\xi)D^{-1}(\xi){}^tG_2(y,\xi)\sigma_3d\xi}\nonumber\\
&=&\int_{0}^{\infty}
e^{\frac{i}{2}t\xi^{2}}\xi\chi(\xi^2)
G_1(x,\xi)A(\xi)D^{-1}(\xi){}^tG_2(y,\xi)\sigma_3d\xi\\
& &+\int_{0}^{\infty}
e^{\frac{i}{2}t\xi^{2}}\xi\chi(\xi^2)
G_2(x,\xi)B(\xi)D^{-1}(\xi){}^tG_2(y,\xi)\sigma_3d\xi\\
&=:&J(x,y)+K(x,y).
\end{eqnarray*}
Again, we split $J$ into the following four pieces.
\begin{eqnarray}
J(x,y)&=&
\int_{0}^{\infty}
e^{\frac{i}{2}t\xi^{2}}\xi\chi(\xi^2)
\left[f_2(-x,\xi)  \begin{array}{cc}0\\0\end{array}\right]
A(\xi)D^{-1}(\xi)
\left[\begin{array}{cc}f_1(-y,\xi)  \\0\ \ \ 0\end{array}\right]
\sigma_3d\xi\nonumber\\
& &
+\int_{0}^{\infty}
e^{\frac{i}{2}t\xi^{2}}\xi\chi(\xi^2)
\left[f_2(-x,\xi)  \begin{array}{cc}0\\0\end{array}\right]
A(\xi)D^{-1}(\xi)
\left[\begin{array}{cc}0\ \ \ 0  \\ f_3(-y,\xi)\end{array}\right]
\sigma_3d\xi\nonumber\\
& &+\int_{0}^{\infty}
e^{\frac{i}{2}t\xi^{2}}\xi\chi(\xi^2)
\left[ \begin{array}{cc}0\\0\end{array}f_4(-x,\xi)\right]
A(\xi)D^{-1}(\xi)
\left[\begin{array}{cc}{}f_1(-y,\xi)  \\0\ \ \ 0\end{array}\right]
\sigma_3d\xi\nonumber\\
& &+\int_{0}^{\infty}
e^{\frac{i}{2}t\xi^{2}}\xi\chi(\xi^2)
\left[\begin{array}{cc}0\\0\end{array}f_4(-x,\xi)\right]
A(\xi)D^{-1}(\xi)
\left[\begin{array}{cc}0\ \ \ 0  \\ f_3(-y,\xi)\end{array}\right]
\sigma_3d\xi\nonumber\\
&=:&J_{1}(x,y)+J_{2}(x,y)+J_{3}(x,y)+J_{4}(x,y).
\nonumber
\end{eqnarray}
Using Corollary~\ref{C:Dneq0}, we may deduce that $J_{2}\equiv J_{3}\equiv 0$.   Arguing as in the case $x\geq 0\geq y$, we can then obtain
\begin{eqnarray*}
\sup_{0\geq x\geq y}|J(x,y)|\lesssim
\sup_{0\geq x\geq y}(|J_{1}(x,y)|+|J_{4}(x,y)|)\lesssim t^{-\frac12},
\end{eqnarray*}
with a similar estimate for $K$. 

Collecting the above estimates, we conclude 
\begin{eqnarray}
\left|2\pi i\langle e^{it{{\mathcal H}}}\chi({{\mathcal H}}-\omega)
P_c\varphi,\psi\rangle\right|&\lesssim&
t^{-\frac12}\|\varphi\|_{L^1}\|\psi\|_{L^1},\label{LE}
\end{eqnarray}
with implicit constant depending on $\mathcal{V}$.  This completes the estimation of the contribution of the low energies to the unweighted dispersive estimate \eqref{dispersive1}. 

For the low-energy contribution to the weighted estimate \eqref{dispersive2}, we will once again proceed essentially as in \cite{KS} (see \cite[Proposition~8.1]{KS}) due to the fact that the distorted plane waves $e_{\pm}$ 
constructed above share the same decay/regularity properties as those appearing in \cite{KS}.

By Proposition \ref{P:Rep}, we have
\begin{eqnarray*}
\lefteqn{2\pi i
\langle e^{it{{\mathcal H}}}\chi({{\mathcal H}}-\omega)
P_{\text{c}}^+\varphi,\psi\rangle}\\
&=&
-4ie^{it\omega}
\int_0^{\infty}
e^{\frac{i}{2}t\xi^2}\xi^2\chi(\xi)
\langle f,\sigma_3F_1(\cdot,\xi)D(\xi)^{-1}e\rangle 
%\langle g,\sigma_3F_1(\cdot,\xi)D(\xi)^{-1}e\rangle  d\xi\\
\\
& &\qquad\qquad\qquad\qquad\qquad\qquad\times
\langle g,\sigma_3F_1(\cdot,\xi)D(\xi)^{-1}e\rangle  d\xi\\
& &
-4ie^{it\omega}
\int^0_{-\infty}
e^{\frac{i}{2}t\xi^2}\xi^2\chi(\xi)
\langle f,\sigma_3G_2(\cdot,-\xi)D(-\xi)^{-1}e\rangle  
%\langle g,\sigma_3G_2(\cdot,-\xi)D(\xi)^{-1}e\rangle  d\xi.
\\
& &\qquad\qquad\qquad\qquad\qquad\qquad
\times\langle g,\sigma_3G_2(\cdot,-\xi)D(\xi)^{-1}e\rangle  d\xi.
\end{eqnarray*}
Integrating by parts via the identity
\begin{eqnarray*}
e^{\frac{i}{2}t\xi^2}=\frac{\pt_{\xi}(e^{\frac{i}{2}t\xi^2})}{it\xi},
\end{eqnarray*}
we have
\begin{eqnarray*}
\lefteqn{2\pi i
\langle e^{it{{\mathcal H}}}(1-\chi({{\mathcal H}}-\omega))
P_{\text{c}}^+\varphi,\psi\rangle}\\
&=&
\frac{4}{t}e^{it\omega}
\int_0^{\infty}
e^{\frac{i}{2}t\xi^2}\pt_{\xi}\{\xi\chi(\xi)
\langle f,\sigma_3F_1(\cdot,\xi)D(\xi)^{-1}e\rangle 
\\
& &\qquad\qquad\qquad\qquad\qquad\qquad
\times\langle g,\sigma_3F_1(\cdot,\xi)D(\xi)^{-1}e\rangle  \}d\xi\\
& &
+\frac{4}{t}e^{it\omega}
\int^0_{-\infty}
e^{\frac{i}{2}t\xi^2}\pt_{\xi}\{\xi\chi(\xi)
\langle f,\sigma_3G_2(\cdot,-\xi)D(-\xi)^{-1}e\rangle 
\\
& &\qquad\qquad\qquad\qquad\qquad\qquad 
\times\langle g,\sigma_3G_2(\cdot,-\xi)D(-\xi)^{-1}e\rangle \} d\xi\\
&=:&L_1+L_2.
\end{eqnarray*}

We focus our attention on $L_1$, as $L_2$ can be treated in a similar way.  We split $L_1$ into the following two pieces: 
\begin{eqnarray*}
L_1
&=&\frac{4}{t}e^{it\omega}
\int_0^{\infty}
e^{\frac{i}{2}t\xi^2}\pt_{\xi}(\xi\chi(\xi))
\langle f,\sigma_3F_1(\cdot,\xi)D(\xi)^{-1}e\rangle 
\\
& &\qquad\qquad\qquad\qquad\qquad\qquad\times
\langle g,\sigma_3F_1(\cdot,\xi)D(\xi)^{-1}e\rangle d\xi\\
& &+\frac{4}{t}e^{it\omega}
\int_0^{\infty}
e^{\frac{i}{2}t\xi^2}\xi\chi(\xi)
\pt_{\xi}\{\langle f,\sigma_3F_1(\cdot,\xi)D(\xi)^{-1}e\rangle 
\\
& &\qquad\qquad\qquad\qquad\qquad\qquad\times
\langle g,\sigma_3F_1(\cdot,\xi)D(\xi)^{-1}e\rangle 
\}d\xi\\
&=:&L_{1,1}+L_{1,2}.
\end{eqnarray*}

Let us now focus now on the term $L_{1,1}$; we refer the reader to \cite[Proposition~8.1]{KS} for details concerning the remaining terms.  We let $\omega(\xi)=\pt_{\xi}(\xi\chi(\xi))$ and 
let $\tilde{\omega}(\xi)$ be an another smooth 
cut-off function satisfying $\omega=
\omega\tilde{\omega}$. 
By the $L^1-L^\infty$ dispersive estimate for the free Schr\"odinger group, we see
\begin{eqnarray*}
|L_{1,1}|
%&\lesssim&t^{-\frac32}
%\|\tilde{{{\mathcal F}}}_{\xi\mapsto u}^{-1}[\pt_{\xi}(\xi^2\chi(\xi))]\|_{L_u^1}
%\|\tilde{{{\mathcal F}}}_{\xi\mapsto u}^{-1}
%[\tilde{\chi}(\xi)\langle f,\sigma_3F_1(\cdot,\xi)D(\xi)^{-1}e\rangle ]\|_{L_u^1}\\
%& &\qquad\qquad\times
%\|\tilde{{{\mathcal F}}}_{\xi\mapsto u}^{-1}
%[\tilde{\chi}(\xi)\langle g,\sigma_3F_1(\cdot,\xi)D(\xi)^{-1}e\rangle ]\|_{L_u^1}\\
&\lesssim&t^{-\frac32}
\|\tilde{{{\mathcal F}}}_{\xi\mapsto u}^{-1}
[\omega(\xi)\langle f,\sigma_3F_1(\cdot,\xi)D(\xi)^{-1}e\rangle ]\|_{L_u^1}\\
& &\qquad\qquad\times
\|\tilde{{{\mathcal F}}}_{\xi\mapsto u}^{-1}
[\tilde{\omega}(\xi)\langle g,\sigma_3F_1(\cdot,\xi)D(\xi)^{-1}e\rangle ]\|_{L_u^1},
\end{eqnarray*}
where $\tilde{{{\mathcal F}}}$ is the Fourier 
cosine transform. If we can prove
\begin{eqnarray}
\|\tilde{{{\mathcal F}}}_{\xi\mapsto u}^{-1}
[\omega(\xi)F_1(\cdot,\xi)D(\xi)^{-1}e]\|_{L_u^1}
\lesssim\langle x\rangle,\label{u22}
\end{eqnarray}
then we can obtain 
\begin{eqnarray*}
\|\tilde{{{\mathcal F}}}_{\xi\mapsto u}^{-1}
[\omega(\xi)\langle f,\sigma_3F_1(\cdot,\xi)D(\xi)^{-1}e\rangle ]\|_{L_u^1}
\lesssim\|\langle x\rangle f\|_{L_x^1}.
\end{eqnarray*}
This bound, along with with a similar estimate for 
\[
\tilde{{{\mathcal F}}}_{\xi\mapsto u}^{-1}
[\tilde{\omega}(\xi)\langle g,\sigma_3F_1(\cdot,\xi)D(\xi)^{-1}e\rangle ],
\]
will yield the desired decay estimate \eqref{L2}. 

We turn to \eqref{u22}.  For the case $x\ge0$, an integration by 
parts implies 
\begin{eqnarray*}
\lefteqn{\tilde{{{\mathcal F}}}_{\xi\mapsto u}^{-1}
[\omega(\xi)F_1(\cdot,\xi)D(\xi)^{-1}e](u)}\\
&=&
\sum_{\pm}
\biggl\{
-\frac{1}{2i(x\pm u)}\omega(0)f_1(x,0)D(0)^{-1}_{11}\\
& &-\frac{1}{2i(x\pm u)}\int_0^{\infty}e^{i(x\pm u)\xi}
\pt_{\xi}\left\{\omega(\xi)e^{-ix\xi}f_1(x,\xi)D(\xi)^{-1}_{11}\right\}d\xi\biggl\}\\
& &-\frac1u
\int_0^{\infty}\sin(\xi u)
\pt_{\xi}\left\{\omega(\xi)f_3(x,\xi)D(\xi)^{-1}_{21}\right\}d\xi,
\end{eqnarray*}
where $D(\xi)^{-1}_{ij}$ is the $(i,j)$ component of $D(\xi)^{-1}$. Integrating by parts in the above identity and using Lemma \ref{L:f1-f4}, we have
\begin{eqnarray*}
\left|\tilde{{{\mathcal F}}}_{\xi\mapsto u}^{-1}
[\omega(\xi)F_1(\cdot,\xi)D(\xi)^{-1}e](u)\right|
\lesssim\frac{|x|}{|x^2-u^2|}+\frac{1}{(u+x)^2}+\frac{1}{(u-x)^2}+\frac{1}{u^2}.
\end{eqnarray*}
Combining this with the estimate
\begin{eqnarray*}
\left|\tilde{{{\mathcal F}}}_{\xi\mapsto u}^{-1}
[\omega(\xi)F_1(\cdot,\xi)D(\xi)^{-1}e](u)\right|
\lesssim1,
\end{eqnarray*}
we obtain \eqref{u22}. 

For the case $x\le0$ the key is to rewrite $F_1(x,\xi)D^{-1}(\xi)e$ 
in terms of $f_1(-x,\xi)$ and $f_3(-x,\xi)$: 
\begin{eqnarray}
%\lefteqn{
F_1(x,\xi)D^{-1}(\xi)e%}\\
&=&
\frac{f_1(-x,-\xi)-f_1(-x,\xi)}{2i\xi}
+\frac{\tilde{R}(\xi)+1}{2i\xi}f_1(-x,\xi)\nonumber\\
& &+f_3(-x,\xi)(B(\xi)D^{-1}(\xi))_{21}.\label{kk}
\end{eqnarray}
Now observe that the second term in the right hand side of \eqref{kk} is smooth for any $\xi\in\rre$, due to the fact that $\tilde{R}(0)=-1$ (see Lemma~\ref{L:FG-behave}). Substituting \eqref{kk} into $\tilde{{{\mathcal F}}}_{\xi\mapsto u}^{-1} [\omega(\xi)F_1(\cdot,\xi)D(\xi)^{-1}e](u)$ and applying an  integration by parts and Lemma \ref{L:f1-f4} to the resulting expression, we obtain \eqref{u22}.  We again refer the reader to \cite[Proof of Proposition 8.1]{KS} for more complete details. 

Collecting the above estimates, we arrive at the estimate
\begin{eqnarray}
\langle e^{it{{\mathcal H}}}(1-\chi({{\mathcal H}}-\omega))
P_{c}\varphi,\psi\rangle
\lesssim t^{-\frac32}\|\langle x\rangle f\|_{L^1}\|\langle x\rangle g\|_{L^1},
\label{L2}
\end{eqnarray}
with some implicit constant depending on $\mathcal{V}$, which completes the estimation of the contribution of the low energies  to the weighted dispersive estimate \eqref{dispersive2}. 

%%%%%%%%%%%%%%%%%%%%%
\subsection{Estimates for high energies}

In this section, we obtain estimates for the high-energy contributions by using the Born series expansion 
\begin{eqnarray*}
\lefteqn{({{\mathcal H}}-\tfrac{\xi^{2}}{2}-\omega\mp i0)^{-1}}\\
&=&\sum_{n=0}^{\infty}(-1)^n
({{\mathcal H}}_0-\tfrac{\xi^{2}}{2}-\omega\mp i0)^{-1}
\{{{\mathcal V}}({{\mathcal H}}_0-\tfrac{\xi^{2}}{2}-\omega\mp i0)^{-1}\}^n.
\end{eqnarray*}
In particular, we write 
\begin{eqnarray}
\lefteqn{2\pi i\langle e^{it{{\mathcal H}}}(1-\chi({{\mathcal H}}-\omega))
P_c\varphi,\psi\rangle}
\nonumber\\
&=&
%\frac{1}{2\pi i}
e^{it\omega}\sum_{n=0}^{\infty}(-1)^n\int_{0}^{\infty}
e^{\frac{i}{2}t\xi^{2}}\xi[1-\chi(\xi^2)]\nonumber\\
& &\qquad\times\left[
\langle\{
({{\mathcal H}}_0-\tfrac{\xi^{2}}{2}-\omega- i0)^{-1}
\{{{\mathcal V}}({{\mathcal H}}_0-\tfrac{\xi^{2}}{2}-\omega- i0)^{-1}\}^n
\}\varphi,\psi\rangle\right.\nonumber\\
& &\qquad\quad-\left.
\langle\{
({{\mathcal H}}_0-\tfrac{\xi^{2}}{2}-\omega+ i0)^{-1}
\{{{\mathcal V}}({{\mathcal H}}_0-\tfrac{\xi^{2}}{2}-\omega+ i0)^{-1}\}^n
\}\varphi,\psi\rangle\right] d\xi\nonumber\\
&=:&e^{it\omega}\sum_{n=0}^{\infty}(-1)^nI_n.\label{Sum}
\end{eqnarray}

To obtain decay estimates for $I_n$, we use that 
the integral kernel of 
$({{\mathcal H}}_{0}-\frac{\xi^{2}}{2}-\omega\mp i0)^{-1}$ is given by 
\begin{eqnarray*}
\left[
\begin{array}{cc} 
(H-\frac{\xi^{2}}{2}\mp i0)^{-1}(x,y) & 0 \\
0 & -(H+\frac{\mu^{2}}{2}\pm i0)^{-1}(x,y)
\end{array}
\right],
\end{eqnarray*}
where $\mu=\sqrt{\xi^2+4\omega}$ and 
\begin{eqnarray*}
(H-\tfrac{\xi^{2}}{2}\mp i0)^{-1}(x,y)
&=&\mp\tfrac{i}{\xi}
\left\{e^{\mp i\xi|x-y|}
+R(\mp\xi)e^{\mp i\xi(|x|+|y|)}\right\}\\
&=:&K^1_{\mp}(x,y,\xi),\\
(H+\tfrac{\mu^{2}}{2}\pm i0)^{-1}(x,y)
&=&-\tfrac{1}{\mu}
\left\{e^{-\mu|x-y|}
+R(i\mu)%\frac{q}{\mu+q}
e^{-\mu(|x|+|y|)}\right\}\\
&=:&K^2(x,y,\xi).
\end{eqnarray*}
Here $R(\xi)=\tfrac{q}{i\xi-q}$ is the reflection coefficient associated to the operator $H=-\tfrac12\partial_x^2+q\delta_0$. 

We first evaluate $I_0$. A direct computation shows  
\begin{eqnarray*}
I_0&=&\int_{0}^{\infty}
e^{\frac{i}{2}t\xi^{2}}\xi[1-\chi(\xi^2)]\\
& &\quad\times\left[
\int_{-\infty}^{+\infty}\left\{\int_{-\infty}^{+\infty}
\left[K^1_-(x,y,\xi)-K^1_+(x,y,\xi)\right]\varphi_1(y)dy\right\}\overline{\psi_1(x)}dx
\right]d\xi\\
&=&\int_{-\infty}^{+\infty}\int_{-\infty}^{+\infty}
\left\{\int_{0}^{\infty}
e^{\frac{i}{2}t\xi^{2}}\xi[1-\chi(\xi^2)]
\left[K^1_-(x,y,\xi)-K^1_+(x,y,\xi)\right]d\xi\right\}\\
& &\qquad\qquad\qquad\qquad\qquad\qquad\qquad
\qquad\qquad\qquad
\times\overline{\psi_1(x)}\varphi_1(y)\,dx\,dy,
\end{eqnarray*}
where $\varphi=\left[\begin{array}{cc}\varphi_1\\
\bar\varphi_1\end{array}\right]$ and $\psi=\left[\begin{array}{cc}\psi_1\\
\bar\psi_1\end{array}\right]$. Hence if we can show that
\begin{eqnarray}
\sup_{x,y\in\rre}\left|\int_{0}^{\infty}e^{\frac{i}{2}t\xi^{2}}\xi[1-\chi(\xi^2)]
\left[K^1_-(x,y,\xi)-K^1_+(x,y,\xi)\right]d\xi\right|\lesssim t^{-\frac12},\label{a1}
\end{eqnarray}
then we will have the desired estimate 
\begin{eqnarray}
|I_0|\lesssim t^{-\frac12}\|\varphi\|_{L^1}\|\psi\|_{L^1}. \label{I0} 
\end{eqnarray}
To obtain \eqref{a1}, it suffices to show 
\begin{eqnarray}
\sup_{a\in\rre}\left|\int_{0}^{\infty}e^{ia\xi+\frac{i}{2}t\xi^{2}}[1-\chi(\xi^2)]g(\xi)
d\xi\right|\lesssim t^{-\frac12},\label{a2}
\end{eqnarray}
where $g=1, R$ or $\bar{R}$. Using the identity 
\begin{equation}\label{IBPID}
e^{ia\xi+\frac{i}{2}t\xi^{2}}
=\frac{\pt_{\xi}\left\{(\xi+\frac{a}{t})e^{ia\xi+\frac{i}{2}t\xi^{2}}\right\}
}{1+it(\xi+\frac{a}{t})^2}
\end{equation}
and integrating by parts, we obtain
\begin{eqnarray*}
\lefteqn{\int_{0}^{\infty}e^{ia\xi+\frac{i}{2}t\xi^{2}}[1-\chi(\xi^2)]g(\xi)
d\xi}\\
&=&
-\int_{0}^{\infty}(\xi+\frac{a}{t})e^{ia\xi+\frac{i}{2}t\xi^{2}}
\pt_{\xi}\left\{\frac{[1-\chi(\xi^2)]g(\xi)}{1+it(\xi+\frac{a}{t})^2}\right\}
d\xi\\
&=&
-\int_{0}^{\infty}(\xi+\frac{a}{t})e^{ia\xi+\frac{i}{2}t\xi^{2}}
\frac{\pt_{\xi}\left\{[1-\chi(\xi^2)]g(\xi)\right\}}{1+it(\xi+\frac{a}{t})^2}d\xi\\
& &
+2it\int_{0}^{\infty}(\xi+\frac{a}{t})^2e^{ia\xi+\frac{i}{2}t\xi^{2}}
\frac{[1-\chi(\xi^2)]g(\xi)}{\left\{1+it(\xi+\frac{a}{t})^2\right\}^2}
d\xi.
\end{eqnarray*}
Thus we can estimate
\begin{eqnarray*}
\lefteqn{
\left|\int_{0}^{\infty}e^{ia\xi+\frac{i}{2}t\xi^{2}}[1-\chi(\xi^2)]g(\xi)
d\xi\right|}\\
&\lesssim&
t^{-\frac12}\int_{0}^{\infty}\left|\pt_{\xi}\left\{[1-\chi(\xi^2)]g(\xi)\right\}\right|d\xi
+\int_{0}^{\infty}\frac{d\xi}{1+t(\xi+\frac{a}{t})^2}\\
&\lesssim&t^{-\frac12},
\end{eqnarray*}
with implicit constants independent of $a$.  This yields \eqref{a2}, as desired.

We turn to the estimation of $I_n$ for $n\geq1$. For this, we rely on the following identity:
\begin{eqnarray*}
\lefteqn{({{\mathcal H}}_q-\tfrac{\xi^{2}}{2}-\omega\mp i0)^{-1}
\{{{\mathcal V}}({{\mathcal H}}_q-\tfrac{\xi^{2}}{2}-\omega\mp i0)^{-1}\}^n
\}\varphi(x)}\\
&=&
\int_{\rre^{n+1}}
{{\mathcal K}}_{\mp}^n(x,x_1,\cdots, x_{n+1})\varphi(x_{n+1})
dx_1\cdots dx_{n+1},
\end{eqnarray*}
where 
\begin{eqnarray*}
\lefteqn{{{\mathcal K}}_{\mp}^n(x,x_1,\cdots, x_{n+1},\xi)}\\
&=&
\left[
\begin{array}{cc}
K_{\mp}^1(x,x_1,\xi) & 0 \\
0 & K^2(x,x_1,\xi)
\end{array}
\right]\\
& &\times\prod_{j=1}^{n}
\left[
\begin{array}{cc}
V_1(x_j) & V_2(x_j) \\
-V_2(x_j) & -V_1(x_j)
\end{array}
\right]
\left[
\begin{array}{cc}
K_{\mp}^1(x_{j},x_{j+1},\xi) & 0 \\
0 & K^2(x_{j},x_{j+1},\xi)
\end{array}
\right],
\end{eqnarray*}
where $K^1_{\pm}$ and $K^2$ are as above and we have set $V_1(x)=\sigma(\frac{p}{2}+1)Q_\omega^{p}(x)$ and 
$V_2(x)=\sigma\frac{p}{2}Q_\omega^{p}(x)$. 

It follows that
\begin{align*}
|I_n| & \lesssim \sum_{\pm} \biggl| \int_0^\infty\int_{\R^{n+2}} e^{\frac{i}{2}t\xi^2}\xi[1-\chi(\xi^2)] \\
&\quad \quad \times K_{\mp}^n(x,x_1,\dots,x_{n+1},\xi)\varphi(x_{n+1})\bar\psi(x)\,dx\,dx_1\cdots dx_{n+1} \,d\xi\biggr|. 
\end{align*}
We now observe that each component of the matrices 
${{\mathcal K}}_{\mp}^n(x,x_1,\cdots, x_{n+1})$ 
is sum of $2^n$ terms of the form
\begin{eqnarray*}
C(p)\prod_{j=1}^nV_1(x_j)\xi^{-n-1}
g(\xi)e^{ih(x,x_1,\cdots,x_{n+1})\xi}e^{-k(x,x_1,\cdots,x_{n+1})\mu},
\end{eqnarray*}
where $g$ is some power of $R$ or $\bar{R}$, and $h$ and  $k$ are real-valued functions with $k(x,x_1,\cdots,x_{n+1})\geq 0$.  Thus the right-hand side of the above inequality  
is bounded by sum of $2^n$ terms of the form
\begin{eqnarray*}
\int_{\rre^{n+2}}
|\psi(x)|\prod_{j=1}^nV_1(x_j)|\varphi(x_{n+1})|\,\bigl|
I(x,x_1,\cdots,x_{n+1},\xi)\bigr|dx\,dx_1\cdots dx_{n+1},
\end{eqnarray*}
where
\begin{eqnarray*}
\lefteqn{I(x,x_1,\cdots,x_{n+1},\xi)}\\
&=&
\int_{0}^{\infty}e^{\frac{i}{2}t\xi^{2}}\xi^{-n}[1-\chi(\xi^2)]
g(\xi)e^{ih(x,x_1,\cdots,x_{n+1})\xi}e^{-k(x,x_1,\cdots,x_{n+1})\mu}d\xi.
\end{eqnarray*}
Thus, if we can prove that
\begin{eqnarray}
\sup_{a\in\rre,b\geq 0}
\left|\int_{0}^{\infty}e^{\frac{i}{2}t\xi^{2}}e^{ia\xi}e^{-b\mu}\xi^{-n}[1-\chi(\xi^2)]
g(\xi)d\xi\right|\leq C\lambda^{-n+1}t^{-\frac12},
\label{q1}
\end{eqnarray}
where the constant $C$ is independent of $n$ and $g$, 
then we can obtain the estimate 
\begin{eqnarray}
|I_n|\lesssim t^{-\frac12}2^n \lambda^{-n+1}\|{{\mathcal V}}\|_{L^1}^n
\|\varphi\|_{L^1}\|\psi\|_{L^1},\label{In}
\end{eqnarray}
with implicit constant independent of $n$. 

We turn to \eqref{q1}. Using \eqref{IBPID} once again and integrating by parts, we obtain
\begin{eqnarray*}
\lefteqn{\int_{0}^{\infty}e^{ia\xi+\frac{i}{2}t\xi^{2}}
e^{-b\mu}\xi^{-n}[1-\chi(\xi^2)]
g(\xi)
d\xi}\\
&=&
-\int_{0}^{\infty}(\xi+\frac{a}{t})e^{ia\xi+\frac{i}{2}t\xi^{2}}
\pt_{\xi}\left\{\frac{e^{-b\mu}\xi^{-n}[1-\chi(\xi^2)]g(\xi)}{
1+it(\xi+\frac{a}{t})^2}\right\}d\xi\\
&=&
-\int_{0}^{\infty}(\xi+\frac{a}{t})e^{ia\xi+\frac{i}{2}t\xi^{2}}
\frac{\pt_{\xi}\left\{e^{-b\mu}\xi^{-n}[1-\chi(\xi^2)]
g(\xi)\right\}}{1+it(\xi+\frac{a}{t})^2}d\xi\\
& &
+2it\int_{0}^{\infty}(\xi+\frac{a}{t})^2e^{ia\xi+\frac{i}{2}t\xi^{2}}
\frac{e^{-b\mu}\xi^{-n}[1-\chi(\xi^2)]g(\xi)}{\left\{
1+it(\xi+\frac{a}{t})^2\right\}^2}d\xi.
\end{eqnarray*}
Thus we firstly obtain 
\begin{eqnarray*}
\lefteqn{
\left|\int_{0}^{\infty}e^{ia\xi+\frac{i}{2}t\xi^{2}}
e^{-b\mu}\xi^{-n}[1-\chi(\xi^2)]g(\xi)
d\xi\right|}\\
&\lesssim&
t^{-\frac12}\int_{0}^{\infty}\left|\pt_{\xi}
\left\{e^{-b\mu}\xi^{-n}[1-\chi(\xi^2)]g(\xi)\right\}\right|d\xi
+\lambda^{-n}\int_{0}^{\infty}\frac{d\xi}{1+t(\xi+\frac{a}{t})^2}.
\end{eqnarray*}

To continue, we estimate as follows:
\begin{eqnarray*}
\lefteqn{\int_{0}^{\infty}\left|\pt_{\xi}
\left\{e^{-b\mu}\xi^{-n}[1-\chi(\xi^2)]g(\xi)\right\}\right|d\xi}\\
&\leq&
b\int_{0}^{\infty}e^{-b\mu}\xi^{-n}[1-\chi(\xi^2)]|g(\xi)|d\xi
+
\int_{0}^{\infty}e^{-b\mu}\left|\pt_{\xi}
\left\{\xi^{-n}[1-\chi(\xi^2)]g(\xi)\right\}\right|d\xi\\
&\leq&\lambda^{-n}b\int_{0}^{\infty}e^{-b\mu}d\xi
+\int_{0}^{\infty}\left|\pt_{\xi}
\left\{\xi^{-n}[1-\chi(\xi^2)]g(\xi)\right\}\right|d\xi\\
&\lesssim&\lambda^{-n+1},
\end{eqnarray*}
where we note that the implicit constants are independent of $b$.  In particular, continuing from above we obtain \eqref{q1}. 

Substituting \eqref{I0} and \eqref{In} into \eqref{Sum}, we see that with  $\lambda=C\|{{\mathcal V}}\|_{L^1}$, we obtain 
\begin{eqnarray}
\lefteqn{\left|2\pi i\langle e^{it{{\mathcal H}}}(1-\chi({{\mathcal H}}-\omega))
P_c\varphi,\psi\rangle\right|}
\nonumber\\
&\lesssim&
t^{-\frac12}\|\varphi\|_{L^1}\|\psi\|_{L^1}
\sum_{n=0}^{\infty}2^n \lambda^{-n+1}\|{{\mathcal V}}\|_{L^1}^n
\nonumber\\
&\lesssim&
t^{-\frac12}\|\varphi\|_{L^1}\|\psi\|_{L^1},\label{HE}
\end{eqnarray}
yielding the desired unweighted estimate \eqref{dispersive1} in the high-energy regime.  

For the high-energy contribution to the weighted estimate \eqref{dispersive2}, we begin by performing an integration by parts in \eqref{Sum}, obtaining
\begin{eqnarray}
\lefteqn{2\pi i \langle e^{it{{\mathcal H}}}(1-\chi({{\mathcal H}}-\omega))
P_c\varphi,\psi\rangle}
\nonumber\\
&=&
%\frac{1}{2\pi i}
t^{-1}e^{it\omega}\sum_{n=0}^{\infty}(-1)^n\int_{0}^{\infty}
e^{\frac{i}{2}t\xi^{2}}\pt_{\xi}\biggl[\{[1-\chi(\xi^2)]\nonumber\\
& &\qquad\times\left[
\langle\{
({{\mathcal H}}_q-\tfrac{\xi^{2}}{2}-\omega- i0)^{-1}
\{{{\mathcal V}}({{\mathcal H}}_q-\tfrac{\xi^{2}}{2}-\omega- i0)^{-1}\}^n
\}\varphi,\psi\rangle\right.\nonumber\\
& &\qquad\quad-\left.
\langle\{
({{\mathcal H}}_q-\tfrac{\xi^{2}}{2}-\omega+ i0)^{-1}
\{{{\mathcal V}}({{\mathcal H}}_q-\tfrac{\xi^{2}}{2}-\omega+ i0)^{-1}\}^n
\}\varphi,\psi\rangle\right] \biggl]d\xi.\nonumber
\end{eqnarray}
One may now proceed essentially as we did for the unweighted estimate, leading to the estimate 
\begin{eqnarray}
\langle e^{it{{\mathcal H}}}(1-\chi({{\mathcal H}}-\omega))
P_c\varphi,\psi\rangle
\lesssim t^{-\frac32}\|\langle x\rangle \varphi\|_{L^1}\|\langle x\rangle \psi\|_{L^1}.
\label{H2}
\end{eqnarray}
In particular, the additional decay is already present in the formula above, while the weights appear due to the fact that the additional derivative hits the resolvent, leading to terms like $|x_j-x_{j-1}|$ (cf. the definition of $K_{\pm}^1$ and $K^2$ above).  We then apply the simple bound $|x_j-x_{j-1}|\leq |x_j|+|x_{j-1}|$ and use either $V_1$, $\varphi$, or $\psi$ to absorb the weight. 

This completes the estimation of the high-energy regime, and hence completes the proof of Proposition~\ref{P:D}. 
%%%%

%\end{proof}

%%%%%%%%%%%%%%%%%%%%%
%%%%%%%%%%%%%%%%%%%%%
\subsection{Interpolation estimates}

In Section~\ref{S:stability}, we will rely on the dispersive estimates appearing in Proposition~\ref{P:D} in an essential way.  However, we will not apply these estimates directly, but rather in an interpolated form.  In this section, we collect the estimates in the form that we will need them.

First, interpolating \eqref{dispersive1} with \eqref{dispersive3} leads to the estimate
\begin{equation}\label{dispersive4}
\|e^{-it\H}P_c\varphi\|_{L^{r,2}} \lesssim |t|^{-\frac12(1-\frac{2}{r})}\|\varphi\|_{L^{{r',2}}},\quad 2\leq r<\infty, 
\end{equation} 
where $L^{a,b}$ denotes the Lorentz space (see Section~\ref{S:notation}).

We may also interpolate \eqref{dispersive2} with \eqref{dispersive3}, leading to 
\begin{equation}\label{dispersive5}
\|\langle x\rangle^{-(1-\frac{2}{r})}e^{-it\H}P_c\varphi\|_{L^{r,2}} \lesssim |t|^{-\frac32(1-\frac{2}{r})}\|\langle x\rangle^{1-\frac2r}\varphi\|_{L^{r',2}},\quad 2\leq r<\infty. 
\end{equation}
In particular, interpolating between these last two estimates yields
\begin{equation}\label{dispersive6}
\| \langle x\rangle^{-\theta(1-\frac2r)}e^{-it\H}P_c\varphi\|_{L^{r,2}}\lesssim |t|^{-(\frac12+\theta)(1-\frac2r)}\|\langle x\rangle^{\theta(1-\frac2r)}\varphi\|_{L^{r',2}}
\end{equation}
for any $2\leq r<\infty$ and $\theta\in[0,1]$.  Using this estimate together with H\"older's inequality, we can establish the following weighted $L^2$ estimate:
\begin{proposition}\label{P:weighted-dispersive} For any $2\leq r<\infty$ and $\alpha\in[\tfrac12-\tfrac1r,\tfrac32-\tfrac3r]$,  
\begin{equation}\label{dispersive7}
\|\langle x\rangle^{-\alpha}e^{-it\H}P_c \varphi\|_{L^2} \lesssim |t|^{-\alpha}\|\langle x\rangle^{\alpha-[\frac12-\frac1r]}\varphi\|_{L^{r',2}}. 
\end{equation}
\end{proposition}

\begin{proof} Using H\"older's inequality (for Lorentz spaces) and \eqref{dispersive6}, we estimate
\begin{align*}
\|\langle x\rangle^{-\alpha}e^{-it\H}P_c \varphi\|_{L^2} & \lesssim \| \langle x\rangle^{-[\frac12-\frac1r]}\|_{L^{\frac{2r}{r-2},\infty}}\|\langle x\rangle^{-\alpha+[\frac12-\frac1r]}e^{-it\H}P_c\varphi\|_{L^{r,2}} \\
& \lesssim |t|^{-(1-\frac2r)(\frac12+\beta)}\|\langle x\rangle^{\alpha-[\frac12-\frac1r]}\varphi\|_{L^{r',2}}, 
\end{align*}
where $\beta\in[0,1]$ is determined by imposing
\[
\beta(1-\tfrac2r)=\alpha-[\tfrac12-\tfrac1r]. 
\]
In particular, the requirement that $\beta$ belong to $[0,1]$ requires $\alpha\in[\tfrac12-\tfrac1r,\tfrac32-\tfrac3r]$.  Noting that
\[
(1-\tfrac{2}{r})(\tfrac12+\beta)=\alpha
\]
then completes the proof. 
\end{proof}

%%%%%%%%%%%%%%%%%%%%%
%%%%%%%%%%%%%%%%%%%%%
%%%%%%%%%%%%%%%%%%%%%
%%%%%%%%%%%%%%%%%%%%%
%%%%%%%%%%%%%%%%%%%%%
%%%%%%%%%%%%%%%%%%%%%

\section{Proof of asymptotic stability}\label{S:stability}

In this section, we prove the main result (Theorem~\ref{T}). We fix $q<0$ and fix $\omega_0$ such that $Q_{\omega_0}$ is orbitally stable and $\omega_0$ obeys the spectral condition appearing in Definition~\ref{D:spectral}.  Let us also fix a weight $1<\alpha<\min\{\frac{p}{4},\tfrac32\}$, which will appear in the bootstrap norm below (see \eqref{D:X}).  In \eqref{orthogonality} and in the arguments below, we use notation such as $\partial_\omega Q_{\omega(t)}$ or $\partial_\omega Q_\omega$ as short-hand for the expression $\partial_\omega Q_{\omega}\bigr|_{\omega=\omega(t)}$.  

By the orbital stability results of \cite{FOO,KO}, we know that any solution that starts close to $Q_{\omega_0}$ in $H^1$ will be global and remain close to the orbit of $Q_{\omega_0}$ in $H^1$ for all time.  We begin by upgrading the orbital stability to a decomposition that incorporates the orthogonality conditions needed to establish decay and scattering for the perturbative part of the solution.  

\begin{proposition}[Decomposition]\label{P:decomposition} There exists $\eps=\eps(\omega_0)$ sufficiently small and $\eta=\eta(\eps)$ sufficiently small that if
\[
u_0=Q_{\omega_0}+\tilde v_0\qtq{with} \|\tilde v_0\|_{H^1}<\eta
\]
and $u(t)$ is the corresponding solution to \eqref{nls}, then there exists a decomposition
\begin{equation}\label{u=Q+v}
u(t) = e^{i\Phi(t)}\bigl\{Q_{\omega(t)} + v(t)\bigr\}\qtq{for all}t\geq 0
\end{equation}
obeying the following properties:
\begin{itemize}
\item The phase has the following form: 
\begin{equation}\label{D:Phi}
\Phi(t):=\theta(t)+\int_0^t \omega(s)\,ds,
\end{equation}
and $\omega(t)$ obeys 
\begin{equation}\label{control-omegat}
\sup_{t\in[0,\infty)}|\omega(t)-\omega_0|<\eps.
\end{equation}
\item The following orthogonality conditions hold for $v(t)$: 
\begin{equation}\label{orthogonality}
\Re\bigl\langle v(t),Q_{\omega(t)}\bigr\rangle = \Re\bigl\langle v(t),i\partial_\omega Q_{\omega(t)}\bigr\rangle = 0.
\end{equation}
\item The following estimates hold for $v(t)$:
\begin{align}
\|\langle x\rangle^\alpha v(0)\|_{L^2} & \leq \|\langle x\rangle^\alpha\tilde v_0\|_{L^2} + \mathcal{O}_{\omega_0}(\eps), \label{control-v0} \\
%\| v(0)\|_{H^1} & \leq \|\tilde v_0 \|_{H^1} + \mathcal{O}_{\omega_0}(\eps), \label{control-v02} \\
\sup_{t\in[0,\infty)}\|v(t)\|_{H^1} & \lesssim_{\omega_0} \eps. \label{control-v-H1}
\end{align} 
\item The following evolution equation holds for $v(t)$:
\begin{equation}\label{nls-v}
i\partial_t v -\dot\theta v = (H+\omega)v + e^{-i\Phi}f(u)-f(Q_\omega) + \dot\theta Q_\omega - i\dot\omega \partial_\omega Q_\omega. 
\end{equation}
\end{itemize}
\end{proposition}

\begin{proof}  We let $\eps>0$ be a small parameter to be determined below, and take $\nu=\nu(\eps)>0$ and $\delta=\delta(\nu)>0$ to be determined below as well.  By orbital stability of $Q_{\omega_0}$, we may find $\eta=\eta(\delta)>0$ and $\beta(t)\in\R$ (with $\beta(0)=0$) such that
\begin{equation}\label{orbital-stability}
\|\tilde v_0\|_{H^1}<\eta \implies \sup_{t\geq 0} \|u(t)-e^{i\beta(t)}Q_{\omega_0}\|_{H^1}<\delta. 
\end{equation}

We will construct the modulation parameters by the implicit function theorem.  We begin by defining the function $F:L^2\times\R^2\to\R^2$ by 
\begin{align*}
F(h,\theta,\omega) = \left[\begin{array}{cc} F_1(h,\theta,\omega) \\ F_2(h,\theta,\omega)\end{array}\right] & = \left[\begin{array}{ll} \Re\langle e^{-i\theta}h-Q_\omega,Q_\omega\rangle \\ \Re\langle e^{-i\theta}h-Q_\omega,i\partial_\omega Q_\omega\rangle\end{array}\right] \\
& = \left[\begin{array}{ll} \Re\langle e^{-i\theta}h,Q_\omega\rangle - \|Q_\omega\|_{L^2}^2 \\ \Re\langle e^{-i\theta}h,i\partial_\omega Q_\omega\rangle\end{array}\right]. 
\end{align*}
By construction, we have
\[
F(e^{i\beta(t)}Q_{\omega_0},\beta(t),\omega_0) \equiv 0. 
\]
Next, by direct computation, 
\[
\left[\begin{array}{cc} \partial_\theta F_1 & \partial_\omega F_1 \\ \partial_\theta F_2 & \partial_\omega F_2 \end{array}\right] = \left[\begin{array}{ll} \Re\langle -ie^{-i\theta}h,Q_\omega\rangle &\Re\langle e^{-i\theta}h,\partial_\omega Q_\omega\rangle - 2\langle Q_\omega,\partial_\omega Q_\omega\rangle \\ \Re\langle -ie^{-i\theta}h,i\partial_\omega Q_\omega\rangle & \Re\langle e^{-i\theta}h,i\partial_\omega^2 Q_\omega\rangle \end{array}\right]. 
\]
In particular, evaluating at $(e^{i\beta(t)}Q_{\omega_0},\beta(t),\omega_0)$ yields the matrix 
\[
J = \left[\begin{array}{cc} 0 & -\langle Q_{\omega_0},\partial_\omega Q_{\omega_0}\rangle\\ -\langle Q_{\omega_0},\partial_\omega Q_{\omega_0}\rangle & 0 \end{array}\right],
\]
which is boundedly invertible (with norm depending on $\omega_0$).  Note that the matrix $J$ and the bounds on the derivatives of $F$ above are \emph{independent of $t$}.  This will allow us to apply the implicit function theorem to the entire curve of zeros $\{(e^{i\beta(t)}Q_{\omega_0},\beta(t),\omega_0):t\geq 0\}$, yielding neighborhoods around $e^{i\beta(t)}Q_{\omega}$ and $(\beta(t),\omega_0)$ of fixed size (that is, with size independent of $t$). In particular, by the implicit function theorem, if we choose $\nu=\nu(\omega_0)$ and $\delta=\delta(\nu)$ sufficiently small, we may find for each $t\geq 0$ a map
\[
T_t:B_\delta(e^{i\beta(t)}Q_{\omega})\subset L^2\to B_{\nu}((\beta(t),\omega_0)) \subset \R^2
\]
such that
\[
F(h,T_t(h))=0 \qtq{for all}h\in B_\delta(e^{i\beta(t)}Q_{\omega_0}). 
\]
As \eqref{orbital-stability} implies $u(t)\in B_\delta(e^{i\beta(t)}Q_{\omega_0})$ for all $t\geq 0$, we may choose
\[
(\Phi(t),\omega(t)):=T_t(u(t))
\]
and set
\[
v(t):=e^{-i\Phi(t)}u(t)-Q_{\omega(t)},\quad \theta(t):=\Phi(t)-\int_0^t\omega(s)\,ds,
\] 
to arrive at the desired decomposition with \eqref{D:Phi} and \eqref{orthogonality}.  We also have \eqref{control-omegat} provided we take $\nu\leq \eps$. 

We next establish \eqref{control-v0}. We first write
\[
Q_{\omega_0}+\tilde v_0 = u_0 = e^{i\theta(0)}\{Q_{\omega(0)}+v(0)\}, 
\]
so that
\[
\|\langle x\rangle^\alpha v(0)\|_{L^2} \leq \|\langle x\rangle^\alpha \tilde v_0\|_{L^2} + \|\langle x\rangle^\alpha[Q_{\omega_0}-e^{i\theta(0)}Q_{\omega(0)}]\|_{L^2}. 
\]
Now we estimate
\begin{align*}
\|& \langle x\rangle^\alpha[Q_{\omega_0}-e^{i\theta(0)}Q_{\omega(0)}]\|_{L^2} \\
& \leq |\theta(0)|\cdot\|\langle x\rangle^\alpha Q_{\omega_0}\|_{L^2} + \|\langle x\rangle^\alpha[Q_{\omega_0}-Q_{\omega(0)}]\|_{L^2}.
\end{align*}
By construction, we have $|\theta(0)|<\nu\leq \eps$ (cf. $\beta(0)=0$).  Similarly, as $|\omega(0)-\omega_0|<\nu$, we may choose $\nu=\nu(\eps,\omega_0)$ possibly even smaller to guarantee that the second term above is bounded by $\eps$.  Continuing from above, we derive \eqref{control-v0}.

The proof of \eqref{control-v-H1} follows in a similar fashion.  We first write
\[
e^{i\Phi(t)}\{Q_{\omega(t)}+v(t)\} = e^{i\beta(t)}Q_{\omega_0}+ [u(t) - e^{i\beta(t)}Q_{\omega_0}],
\]
so that by \eqref{orbital-stability} we have
\[
\|v(t)\|_{H^1} \leq \|e^{i\Phi(t)}Q_{\omega(t)} - e^{i\beta(t)}Q_{\omega_0}\|_{H^1}+\delta. 
\]
We now observe that
\[
\|e^{i\Phi(t)}Q_{\omega(t)} - e^{i\beta(t)}Q_{\omega_0}\|_{H^1} \leq |\Phi(t)-\beta(t)|\cdot\|Q_{\omega_0}\|_{H^1} + \|Q_{\omega(t)}-Q_{\omega_0}\|_{H^1}. 
\]
Once again we have $|\Phi(t)-\beta(t)|< \nu\leq\eps$, and as $|\omega(t)-\omega_0|<\nu$, we may choose $\nu=\nu(\eps,\omega_0)$ possibly even smaller we may guarantee that the second term above is bounded by $\eps$. This proves \eqref{control-v-H1}.

Finally, \eqref{nls-v} follows from \eqref{nls} and \eqref{nls-Q}. \end{proof}

%%%%%%%%%%%%%%%%%%%%%%%%%%
%%%%%%%%%%%%%%%%%%%%%%%%%%

We will show that with the decomposition \eqref{u=Q+v}, the perturbation $v(t)$ in fact converges to zero in suitable norms, and the modulation parameter $\omega(t)$ converges as $t\to\infty$.
%%%%%%%%%%%%%%%%%%%%%%%%%%
%%%%%%%%%%%%%%%%%%%%%%%%%%

By differentiating the orthogonality relations \eqref{orthogonality}, we first obtain an ODE system describing the evolution of the modulation parameters $(\omega,\theta)$.  The choice of orthogonality conditions guarantees that the nonlinearity in the equation below contains terms that are quadratic and higher in $v$. 

\begin{proposition}\label{P:ODE} The modulation parameters $(\omega,\theta)$ obey the ODE system
\begin{equation}\label{ODE}
A\left[\begin{array}{c} \dot\omega \\ \dot\theta\end{array}\right] = \left[\begin{array}{c} \Im\langle \N,Q_\omega\rangle \\ -\Re\langle \N,\partial_\omega Q_\omega\rangle\end{array}\right],
\end{equation}
where $A=A(t)$ is defined by 
\begin{equation}\label{D:A}
A:=\left[\begin{array}{ll} \langle Q_\omega,\partial_\omega Q\rangle-\Re\langle v,\partial_\omega Q\rangle   & -\Im\langle v,Q_\omega\rangle \\ -\Im\langle v,\partial_\omega^2 Q_\omega\rangle & \langle Q_\omega,\partial_\omega Q_\omega\rangle+ \Re\langle v,\partial_\omega Q_\omega\rangle  \end{array}\right]
\end{equation}
and
\begin{equation}\label{D:N}
\N := e^{-i\Phi} f(u)-f(Q_\omega)-\sigma\tfrac{p+2}{2}Q_\omega^p v - \sigma\tfrac{p}{2}Q_\omega^p \bar v.
\end{equation}
\end{proposition}

\begin{proof} We start with \eqref{nls-v}, multiply by $Q_\omega$, integrate, and take the imaginary part.  Using self-adjointness of $H+\omega$, this yields
\begin{align*}
\Re&\langle \dot v,Q_\omega\rangle \\
& = \dot\theta\Im\langle v,Q_\omega\rangle -\dot\omega\langle Q_\omega,\partial_\omega Q_\omega\rangle +\Im \langle e^{-i\Phi} f(u)-f(Q_\omega),Q_\omega\rangle-\Im\langle v,f(Q_\omega)\rangle.
\end{align*}
On the other hand, differentiating the first equation in \eqref{orthogonality} implies
\[
\Re\bigl[\langle \dot v,Q_\omega\rangle + \dot\omega\langle v,\partial_\omega Q_\omega \rangle\bigr] =0,
\]
so that continuing from above yields
\begin{equation}\label{mod-eqn1}
\begin{aligned}
\dot\omega[\langle Q_\omega,\partial_\omega & Q_\omega\rangle-\Re\langle v,\partial_\omega  Q_\omega\rangle] -\dot\theta\Im\langle v,Q_\omega\rangle \\
&= \Im\langle e^{-i\Phi} f(u)-f(Q_\omega),Q_\omega\rangle-\Im\langle v,f(Q_\omega)\rangle  = \Im\langle \N,Q_\omega\rangle,
\end{aligned}
\end{equation}
with $\N$ defined as in \eqref{D:N}. 

Similarly, we start with \eqref{nls-v}, multiply by $i\partial_\omega Q$, integrate, and take the imaginary part.  This yields 
\begin{align*}
-\Im\langle \dot v,\partial_\omega Q_\omega\rangle & = \dot\theta\Re\langle v,\partial_\omega Q_\omega\rangle + \dot\theta\langle Q_\omega,\partial_\omega Q_\omega\rangle \\
&\quad + \Re\langle e^{-i\Phi} f(u)-f(Q_\omega),\partial_\omega Q_\omega\rangle + \Re\langle v,(H+\omega)\partial_\omega Q_\omega\rangle. 
\end{align*}
Recalling \eqref{nls-Q}, we have
\[
(H+\omega)\partial_\omega Q_\omega = - Q_\omega - \sigma(p+1)Q_\omega^p \partial_\omega Q_\omega,
\]
so that using \eqref{orthogonality}, we arrive at
\[
\Re\langle v,(H+\omega)\partial_\omega Q_\omega\rangle =  - \sigma(p+1)\Re\langle v,Q_\omega^p\partial_\omega Q_\omega\rangle.
\]
On the other hand, differentiating the second equation in \eqref{orthogonality} yields 
\[
\Im\langle \dot v,\partial_\omega Q_\omega\rangle + \dot \omega\Im\langle v,\partial_\omega^2 Q_\omega\rangle = 0. 
\]
Thus we arrive at
\begin{equation}\label{mod-eqn2}
\begin{aligned}
\dot\omega\Im\langle v,\partial_\omega^2 & Q_\omega\rangle - \dot\theta[\Re\langle v,\partial_\omega Q_\omega\rangle + \langle Q_\omega,\partial_\omega Q_\omega\rangle] \\
& = \Re\langle e^{-i\Phi} f(u)-f(Q_\omega),\partial_\omega Q_\omega\rangle -\sigma(p+1)\Re\langle v,Q_\omega^p\partial_\omega Q_\omega\rangle \\
& = \Re\langle \N,\partial_\omega Q_\omega\rangle.
\end{aligned}
\end{equation}

Collecting \eqref{mod-eqn1} and \eqref{mod-eqn2} yields \eqref{ODE}. \end{proof}

%%%%%%%%%%%%%%%%%%%%%%%%%%
%%%%%%%%%%%%%%%%%%%%%%%%%%

We turn to the problem of proving global-in-time estimates for $v(t)$ by a bootstrap argument.  We will proceed by proving estimates on intervals of the form $[0,T]$.  It is difficult to analyze $v(t)$ directly due to the fact that the linear operator in \eqref{nls-v} is time-dependent.  Instead, given $T>0$, we will introduce a new variable that will solve an equation with a time-independent linear part.

\begin{proposition}\label{P:COV} Given $T>0$, we set
\[
\Psi(t) = \theta(T) + t\omega(T)
\]
and define $z(t)$ via
\[
e^{i\Phi(t)}v(t) = e^{i\Psi(t)}z(t). 
\]
Then $z(t)$ obeys the evolution equation
\begin{equation}\label{nls-z}
i\partial_t z = [H+\omega(T)]z + \sigma\tfrac{p+2}{2}Q_{\omega(T)}^p z + \sigma\tfrac{p}{2}Q_{\omega(T)}^p \bar z + \F, 
\end{equation}
where
\begin{equation}\label{D:F}
\begin{aligned}
\mathcal{F} & = e^{i\Phi-i\Psi}\bigl\{\N+ \dot\theta Q_\omega - i\dot\omega\partial_\omega Q_\omega\bigr\} \\
& \quad + \sigma\tfrac{p+2}{2}[Q_\omega^p-Q_{\omega(T)}^p]z \\
& \quad + \sigma\tfrac{p}{2}[Q_{\omega}^p - Q_{\omega(T)}^p]\bar z + \sigma\tfrac{p}{2}[e^{2i\Phi-2i\Psi}-1]Q_{\omega}^p \bar z.
\end{aligned}
\end{equation}
\end{proposition}

\begin{proof} To derive \eqref{nls-z}, we first use \eqref{nls-v} along with
\[
\dot\Phi(t) = \dot\theta(t)+\omega(t)\qtq{and} \dot\Psi(t)=\omega(T)
\]
to derive
\[
i\partial_t z = [H+\omega(T)]z + e^{i\Phi-i\Psi}\bigl\{ e^{-i\Phi}f(u)-f(Q)+\dot\theta Q_{\omega} - i\dot\omega \partial_\omega Q_\omega\bigr\}. 
\]
After isolating the parts of $e^{-i\Phi}f(u)-f(Q)$ that are linear in $v$ and writing 
\begin{align*}
e^{i\Phi-i\Psi}&\{\sigma\tfrac{p+2}{2}Q_\omega^p v + \sigma \tfrac{p}{2}Q_\omega^p \bar v\} \\
& = \sigma\tfrac{p+2}{2}Q_{\omega(T)}^p z + \sigma\tfrac{p+2}{2}[Q_\omega^p-Q_{\omega(T)}^p]z \\
& \quad + \sigma\tfrac{p}{2}Q_{\omega(T)}^p \bar z + \sigma\tfrac{p}{2}[Q_{\omega}^p - Q_{\omega(T)}^p]\bar z + \sigma\tfrac{p}{2}[e^{2i\Phi-2i\Psi}-1]Q_{\omega}^p \bar z,
\end{align*}
we arrive at \eqref{nls-z}. \end{proof}

%%%%%%%%%%%%%%%%%%%%%%%%%%
%%%%%%%%%%%%%%%%%%%%%%%%%%

The introduction of $z(t)$ solves one problem (that is, the issue of the time-dependent linear operator) but leads to a new one.  In particular, $z$ does not obey the orthogonality relation needed to apply the dispersive estimates adapted to the linear operator appearing in \eqref{nls-z}.  To remedy this, we introduce one more variable. For this, we recall Proposition~\ref{P:Projection} and introduce the notation $P_c(T)$ for the projection onto the continuous spectrum of $\L(\omega(T))$. 

\begin{proposition} For each $t$, we have a spectral decomposition
\begin{equation}\label{z=abeta}
z(t)=ia(t)Q_{\omega(T)}+b(t)\partial_\omega Q_{\omega(T)}+ \eta(t),
\end{equation}
so that $\eta(t)=P_c(T)z(t)$. The coefficients $(a,b)$ obey the linear system
\begin{equation}\label{ab} 
B\left[\begin{array}{c} a \\ b \end{array}\right] = -\left[\begin{array}{c} \Re\langle e^{i[\Psi-\Phi]}\eta,Q_{\omega(t)} \\ \Re\langle e^{i[\Psi-\Phi]}\eta,i\partial_\omega Q_{\omega(t)}\rangle \end{array}\right],
\end{equation}
where 
\begin{equation}\label{D:B}
B=\left[\begin{array}{ll} -\sin(\Psi-\Phi) \langle Q_{\omega(T)},Q_{\omega(t)}\rangle &  \cos(\Psi-\Phi)\langle \partial_\omega Q_{\omega(T)},Q_{\omega(t)}\rangle \\ \cos(\Psi-\Phi) \langle Q_{\omega(T)},\partial_\omega Q_{\omega(t)}\rangle & \sin(\Psi-\Phi)\langle \partial_\omega Q_{\omega(T)},\partial_\omega Q_{\omega(t)}\rangle \end{array}\right].
\end{equation}
The function $\eta$ (identified with the vector $[\Re\eta\ \Im\eta]^t$) obeys the integral equation
\begin{equation}\label{nls-eta}
\eta(t) = e^{t\L(\omega(T))}P_c(T)z(0) + P_c(T)\int_0^t e^{(t-s)\L(\omega(T))}\left[\begin{array}{r} \Im \F \\ -\Re\F\end{array}\right]\,ds. 
\end{equation} 
\end{proposition}  

\begin{proof} The system \eqref{ab} with \eqref{D:B} follows from direct computation using \eqref{orthogonality}.  

As for \eqref{nls-eta}, we first re-cast \eqref{nls-z} as the following system for $(z_1,z_2)$: 
\begin{equation}\label{zsyst}
\partial_t\left[\begin{array}{r} z_1 \\ z_2 \end{array}\right] = \L(\omega(T))\left[\begin{array}{r} z_1 \\ z_2 \end{array}\right] + \left[\begin{array}{rr} \Im\F \\ -\Re \F\end{array}\right],
\end{equation}
where $\L$ is as in \eqref{D:L}. We now apply $P_c(T)$ and recall that $P_c(T)iQ_{\omega(T)}=P_c(T)\partial_\omega Q_{\omega(T)}=0$, yielding an evolution equation for $\eta$.  Recasting this equation in integral form yields \eqref{nls-eta}.  \end{proof}

%%%%%%%%%%%%%%%%%%%%%%%%%%
%%%%%%%%%%%%%%%%%%%%%%%%%%

We now turn to the bootstrap argument.   We fix $T>0$ and first introduce the quantities that we will need to control.  We recall that we have chosen a weight $\alpha$ obeying
\[
1<\alpha<\min\{\tfrac{p}{4},\tfrac32\},
\]
which requires $p>4$. We then fix a large but finite exponent $r=r(p,\alpha)$, which will be determined more precisely below (see Lemma~\ref{L:controlF} and its proof).  We define

\begin{align}
\|v\|_{X([0,T])} & = \sup_{t\in[0,T]} \bigl\{|t|^{\frac12-\frac1r}\|v(t)\|_{L_x^{r,2}} + |t|^{-\alpha}\|\langle x\rangle^{-\alpha}v(t)\|_{L_x^2}\bigr\}, \label{D:X} \\
\|(\omega,\theta)\|_{Y([0,T])} & = \sup_{t\in[0,T]} \bigl\{ \langle t\rangle^{2\alpha}\bigl[ |\dot\omega(t)|+|\dot\theta(t)|\bigr]\bigr\}, \label{D:Y} \\
\|Q_\omega\|_{Z([0,T])} & =\sup_{t\in[0,T]} \bigl\{ \|\langle x\rangle^{20} Q_\omega\|_{H^1}+\|\langle x\rangle^{20} \partial_\omega Q_\omega\|_{L^2} + \|\langle x\rangle^{20}\partial_\omega^2 Q_\omega\|_{L^2}\bigr\}, \label{D:Z}
\end{align}
where $L^{r,2}$ denotes the Lorentz space. The weight $\langle x\rangle^{20}$ does not have any special significance; any `sufficiently large' weight would do.

We first collect some bounds that depend only on $\omega_0$. First, recall that $\langle Q_{\omega_0},\partial_\omega Q_{\omega_0}\rangle\neq 0$ (see e.g. Section~\ref{S:solitons}).  Thus, choosing $\eps=\eps(\omega_0)$ possibly even smaller in Proposition~\ref{P:decomposition}, we can guarantee (via \eqref{control-omegat}) that 
\begin{equation}
\inf_{t\in[0,\infty]}|\langle Q_\omega,\partial_\omega Q_{\omega}|\rangle  \gtrsim_{\omega_0} 1. \label{bs0}
\end{equation}
Next, noting that
 \[
\|\langle x\rangle^{20}Q_{\omega_0}\|_{H^1} + \|\langle x\rangle^{20}\partial_\omega Q_{\omega_0}\|_{L^2} + \|\langle x\rangle^{20} \partial_\omega^2 Q_{\omega_0}\|_{L^2} \lesssim_{\omega_0} 1
\]
we can choose $\eps=\eps(\omega_0)$ possibly even smaller and  guarantee via \eqref{control-omegat} that
\begin{equation}\label{bs3}
\|Q_\omega\|_{Z([0,\infty])} \lesssim_{\omega_0}1
\end{equation}
(see Section~\ref{S:solitons}). 

We suppose that the following estimates hold on $[0,T]$:
\begin{align}
\|v\|_{X([0,T])} &\leq 2C_1\eps, \label{bs1} \\
\|(\omega,\theta)\|_{Y([0,T])}&\leq 2C_2\eps^2, \label{bs2}
\end{align}
where $\eps>0$ is as in Proposition~\ref{P:decomposition} and the constants $C_1,C_2$ will be determined below; in particular, we will need to choose $C_1$ depending on $\omega_0$ and $C_2$ depending $C_1$. We will also need to take $\eps$ possibly even smaller depending on the constants $\{C_j\}_{j=1}^2$ and the power $p$.  The parameters we choose will also depend on $q$ (the coupling constant for the delta potential); however, as $q$ is fixed throughout the paper, we will not track this dependence throughout the argument. 

Henceforth we assume that the estimates \eqref{bs1} and \eqref{bs2} hold.  We will show we may in fact improve the estimates \eqref{bs1} and \eqref{bs2} on $[0,T]$, removing the factor $2$ in each estimate.  Using this and the smallness assumption at time $t=0$, a standard continuity argument will imply that in fact, \eqref{bs1} and \eqref{bs2} hold on the entire interval $[0,\infty)$. 

We first record a weighted estimate for $v$, which follows from a straightforward virial estimate at the level of the original solution $u$. Because the estimate involves the weight $|x|^2$, the delta potential ultimately has no real effect (e.g. when integrating by parts).  We will proceed with the formal computation and refer the reader to \cite{LeCoz} for a more detailed and careful presentation.

%%%%%%%%%%%%%%%%%%%%%%%%%%
%%%%%%%%%%%%%%%%%%%%%%%%%%

\begin{lemma}\label{L:v1} We have
\[
\|\langle x\rangle v(t)\|_{L^2}\lesssim_{\omega_0} \langle t\rangle
\]
uniformly in $t\in[0,T]$.
\end{lemma}

\begin{proof} By orbital stability, we have
\[
\|u\|_{L_t^\infty \dot H_x^1([0,T]\times\R^3)} \lesssim_{\omega_0} 1.
\]
We now use the virial identity.  By direct computation using \eqref{nls} and integration by parts, we have
\[
\bigl| \tfrac{d}{dt} \|xu(t)\|_{L^2}^2\bigr| =\biggl| 4\Im \int x\bar u \partial_x u\,dx\biggr| \leq 4\|xu(t)\|_{L^2}\|\partial_x u(t)\|_{L^2}, 
\]
which implies
\[
\|xu(t)\|_{L^2} \leq \|xu_0\|_{L^2}+2t\|u\|_{L_t^\infty \dot H_x^1}.
\]
The result now follows from the triangle inequality.  \end{proof}

The next estimate provides control over the modulation parameters.
%%%%%%%%%%%%%%%%%%%%%%%%%%
%%%%%%%%%%%%%%%%%%%%%%%%%%

\begin{lemma}\label{L:control-omega} If $\eps$ is sufficiently small, the matrix $A$ appearing in \eqref{D:A} is invertible for all $t\in[0,T]$, with
\begin{equation}\label{control-Ainverse}
\|A^{-1}\| \lesssim_{\omega_0,C_1} 1
\end{equation}
uniformly in $t\in[0,1]$.  The modulation parameters $(\omega,\theta)$ obey 
\begin{equation}\label{control-omega1}
\|(\omega,\theta)\|_{Y([0,T])}\lesssim_{p,\omega_0,C_1} \|v\|_{X([0,T])}^2.
\end{equation}
In particular, by \eqref{bs1}, 
\begin{equation}\label{control-omega2}
\|(\omega,\theta)\|_{Y([0,T])} \lesssim_{p,\omega_0,C_1} \eps^2. 
\end{equation}
\end{lemma}

\begin{proof} Using \eqref{bs1} and \eqref{bs3}, we compute
\begin{align*}
\det A&  \geq \langle Q_\omega,\partial_\omega Q_\omega\rangle^2 - |\langle v,\partial_\omega Q_\omega\rangle|^2 - |\langle v,Q_\omega\rangle|\,|\langle v,\partial_\omega^2 Q_\omega\rangle| \\
& \gtrsim_{\omega_0} 1 - \mathcal{O}_{\omega_0, C_1}(\eps^2)\gtrsim_{\omega_0} 1
\end{align*}
for $\eps=\eps(\omega_0, C_1)$ sufficiently small.  Using this, we may deduce \eqref{control-Ainverse}.

We turn to \eqref{control-omega1} (which then implies \eqref{control-omega2} via \eqref{bs1}). In light of \eqref{ODE} and \eqref{bs3}, it is enough to prove
\begin{equation}\label{control-N}
\| \langle x\rangle^{-2\alpha}\N\|_{L^1}\lesssim_{p,\omega}\langle t\rangle^{-2\alpha}\|v\|_X^2,
\end{equation}
uniformly in $t\in[0,1]$, where $\N$ is as in \eqref{D:N}. For this we observe that by Young's inequality, Sobolev embedding, and \eqref{control-v-H1}, we have the pointwise estimate
\[
|\N| \lesssim_p |v|^2|Q_{\omega}|^{p-1}+|v|^{p+1}\lesssim_{\omega_0} |v|^2. 
\]
Now \eqref{control-N} follows from the fact that 
\[
\|\langle x\rangle^{-\alpha} v(t)\|_{L^2} \lesssim_{C_1} \langle t\rangle^{-\alpha} \|v\|_{X([0,T])}
\]
uniformly in $t\in[0,T]$. \end{proof}

%%%%%%%%%%%%%%%%%%%%%%%%%%
%%%%%%%%%%%%%%%%%%%%%%%%%%

We next estimate the spectral parameters $a(t)$ and $b(t)$, which in turn allows us to control $v(t)$ (using $\eta(t)=P_c(T)z(t)$). 

\begin{lemma} If $\eps$ is sufficiently small, matrix $B$ appearing in \eqref{D:B} is invertible for all $t\in[0,T]$, with
\begin{equation}\label{control-Binverse}
\|B^{-1}\| \lesssim_{\omega_0} 1.
\end{equation}
The coefficients $(a,b)$ obey
\begin{equation}\label{control-ab}
\sup_{t\in[0,T]}\langle t\rangle^{\alpha}\bigl\{|a(t)| + |b(t)|\bigr\} \lesssim_{\omega_0} \|\eta\|_{X([0,T])},
\end{equation}
and consequently
\begin{equation}\label{eta-controls-v}
\|v\|_{X([0,T])}\lesssim_{\omega_0} \|\eta\|_{X([0,T])}.
\end{equation}
\end{lemma}

\begin{proof} We first abbreviate $s^2 = \sin^2(\Psi-\Phi)$ and $c^2=\cos^2(\Psi-\Phi)$ and compute
\begin{align}
-\det B & = s^2\langle Q_{\omega(T)},Q_\omega\rangle\langle\partial_\omega Q_{\omega(T)},\partial_\omega Q_\omega\rangle + c^2\langle \partial_\omega Q_{\omega(T)},Q_\omega\rangle\langle Q_{\omega(T)},\partial_\omega Q_\omega\rangle \nonumber \\
& = \langle \partial_\omega Q_{\omega(T)},Q_\omega\rangle \langle Q_{\omega(T)},\partial_\omega Q_\omega\rangle \label{B-LB}\\
&\quad  + s^2\bigl[  \langle Q_{\omega(T)},Q_\omega\rangle\langle\partial_\omega Q_{\omega(T)},\partial_\omega Q_\omega\rangle  - \langle \partial_\omega Q_{\omega(T)},Q_\omega\rangle\langle Q_{\omega(T)},\partial_\omega Q_\omega\rangle\bigr].\label{B-error} 
\end{align}
For \eqref{B-LB}, we begin by writing
\begin{align*}
\langle& \partial_\omega Q_{\omega(T)},Q_\omega\rangle \langle Q_{\omega(T)},\partial_\omega Q_\omega\rangle \\
 & \geq [\langle \partial_\omega Q_\omega,Q_\omega\rangle]^2  - \mathcal{O}_{\omega_0}( \|Q_\omega - Q_{\omega(T)}\|_{L^2} + \|\partial_\omega Q_\omega - \partial_\omega Q_{\omega(T)}\|_{L^2}).
\end{align*}
We now observe that by \eqref{bs2}, 
\[
\|Q_\omega -Q_{\omega(T)}\|_{L^2} \lesssim_{\omega_0} \int_t^T |\dot\omega(s)|\,ds \lesssim_{\omega_0,C_2} \int_t^T \langle s\rangle^{-2\alpha}\,ds\lesssim_{\omega_0,C_2}\eps^2
\]
uniformly in $t\in[0,T]$, with a similar bound for $\|\partial_\omega Q_\omega - \partial_\omega Q_{\omega(T)}\|_{L^2}$.  In particular, using \eqref{bs3}, we obtain
\[
\eqref{B-LB}\gtrsim_{\omega_0} 1 - \mathcal{O}_{\omega_0,C_2}(\eps^2) \gtrsim_{\omega_0}1
\]
for $\eps=\eps(\omega_0,C_2)$ sufficiently small. 

It remains to estimate \eqref{B-error}.  We begin with
\[
|\eqref{B-error}| \lesssim_{\omega_0} |\Psi-\Phi|^2. 
\]
We now claim that
\begin{equation}\label{control-phase}
\sup_{t\in[0,T]}|\Psi(t)-\Phi(t)| \lesssim_{C_2} \eps^2,
\end{equation}
so that for $\eps$ chosen possibly even smaller, we may continue from above to deduce
\begin{equation}\label{detB}
|\det B|\gtrsim_{\omega_0} 1
\end{equation}
uniformly over $t\in[0,T]$.  We begin by writing
\[
\Phi(t)-\Psi(t)=\theta(t)-\theta(T)+\int_0^t [\omega(s)-\omega(T)]\,ds. 
\]
Using \eqref{bs2}, we then have
\[
|\theta(t)-\theta(T)|\leq \int_t^T |\dot\theta(s)|\,ds \lesssim_{C_2} \eps^2\int_t^T \langle s\rangle^{-2\alpha}\,ds \lesssim_{C_2} \eps^2.
\]
Similarly, recalling $\alpha>1$, 
\begin{align*}
\int_0^t |\omega(s)-\omega(T)|\,ds & \leq \int_0^t \int_s^T |\dot\omega(\tau)|\,d\tau\,ds \\
& \lesssim_{C_2} \eps^2\int_0^t\int_s^T \langle \tau\rangle^{-2\alpha}\,d\tau\,ds \lesssim_{C_2} \eps^2,
\end{align*}
thus completing the proof of \eqref{control-phase}.  

In particular, we have now proven \eqref{detB}, which shows $B(t)$ is invertible for all $t\in[0,T]$.  Using \eqref{bs3} as well, we readily deduce \eqref{control-Binverse}. 

We turn to \eqref{control-ab}.  In fact, this follows from \eqref{ab}, \eqref{control-Binverse}, \eqref{bs3}, and the straightforward estimate
\[
|\langle |\eta|,|Q_\omega|\rangle| \leq \| \langle x\rangle^{-\alpha} \eta\|_{L^2} \|\langle x\rangle^\alpha Q_\omega\|_{L^2},
\]
with a similar estimate using $\partial_\omega Q_{\omega}$ in place of $Q_\omega$.  Finally, recalling \eqref{z=abeta} and \eqref{bs3}, we see that \eqref{control-ab} implies \eqref{eta-controls-v}.
\end{proof}

%%%%%%%%%%%%%%%%%%%%%%%%%%
%%%%%%%%%%%%%%%%%%%%%%%%%%

We turn to the estimation of $\eta$.  In light of \eqref{eta-controls-v}, control over $\eta$ will yield control over $v$.  We will use the evolution equation \eqref{nls-eta}, relying on the dispersive estimates for the linearized operator.  We begin by recording a pointwise-in-time bound for the nonlinearity $\F$ (see \eqref{D:F}).

\begin{lemma}\label{L:controlF} For $r=r(\alpha,p)$ sufficiently large, 
\[
\|\langle x\rangle^{\alpha-[\frac12-\frac1r]} \F \|_{L^{r',2}}+\|\F\|_{L^2} \lesssim_{p,\omega_0,C_1,C_2} \langle t\rangle^{-\alpha}\eps^2,
\]
uniformly over $t\in[0,T]$, where $\tfrac{1}{r}+\tfrac{1}{r'}=1$. 
\end{lemma}

\begin{proof} We begin by observing that in light of \eqref{control-v-H1}, we may upgrade the bootstrap assumptions
\[
\|v(t)\|_{L^{r,2}} \lesssim_{C_1} |t|^{-(\frac12-\frac1r)}\eps\qtq{and} \|\langle x\rangle^{-\alpha}v(t)\|_{L^2} \lesssim_{C_1} |t|^{-\alpha}\eps
\]
to
\begin{equation}\label{bs1'}
\|v(t)\|_{L^{r,2}} \lesssim_{\omega_0,C_1} \langle t\rangle^{-(\frac12-\frac1r)}\eps\qtq{and} \|\langle x\rangle^{-\alpha}v(t)\|_{L^2}\lesssim_{C_1} \langle t\rangle^{-\alpha}\eps.
\end{equation}

We will first estimate the weighted norm.  For this, we first estimate the contribution of $\N$ to $\F$ (see \eqref{D:N}). We rely on the pointwise estimate
\begin{equation}\label{PW}
|\N| \lesssim_p |Q_{\omega}|^{p-1}|v|^2 + |v|^{p+1}. 
\end{equation}

We begin with the purely nonlinear term $|v|^{p+1}$. We define the parameter
\[
k=\tfrac{2(r-(p+2))}{r-2},
\]
where we impose $r>2(p+1)$ to guarantee that $1<k<2$. 

We may then estimate
\begin{equation}
\begin{aligned}\label{xvpp1}
\| \langle x\rangle^{\alpha-[\frac12-\frac1r]}|v|^{p+1}\|_{L^{r',2}} & \lesssim \| \langle x\rangle^{\alpha-[\frac12-\frac1r]}v^k \|_{L^{\frac2k,\frac2{k-1}}}\|v^{p+1-k}\|_{L^{\frac{r}{p+1-k},\frac{2}{2-k}}} \\ 
&\lesssim \| \langle x\rangle^{\frac{1}{k}(\alpha-[\frac12-\frac1r])}v\|_{L^2}^k \|v\|_{L^{r,\frac{2}{2-k}(p+1-k)}}^{p+1-k},
\end{aligned}
\end{equation}
where we have used $\frac{2k}{k-1}>2$.  For the weighted term, we will use Lemma~\ref{L:v1}.  This requires the condition
\[
\tfrac{1}{k}(\alpha-[\tfrac12-\tfrac1r]) \leq 1. 
\]
Recalling the definition of $k$ and rearranging, we find that this follows provided
\[
\alpha\leq \tfrac{5}{2}-\tfrac{1}{r}-\tfrac{2p}{r-2}.
\]
In particular, as $\alpha<\tfrac{3}{2}$, we can guarantee this condition by choosing $r$ sufficiently large.  We will control the remaining norm in \eqref{xvpp1} by the $L^{r,2}$-norm.  This requires that
\[
\tfrac{2}{2-k}(p+1-k)>2, 
\]
which after rearranging is equivalent to $p>1$.  In particular, we arrive at the estimate
\[
\| \langle x\rangle^{\alpha-[\frac12-\frac1r]}|v|^{p+1}\|_{L^{r',2}}  \lesssim \eps^{p+1-k}\langle t\rangle^{\alpha-[\frac12-\frac1r]-(p+1-k)(\frac12-\frac1r)}. 
\]
We now observe that the condition
\[
\alpha-[\tfrac12-\tfrac1r]-(p+1-k)(\tfrac12-\tfrac1r)<-\alpha
\]
follows from $\alpha<\tfrac{p}{4}$. We also note that $p+1-k>2$ follows from $k<2<p-1$.  Thus we have
\[
\|\langle x\rangle^{\alpha-[\frac12-\frac1r]}|v|^{p+1}\|_{L^{r',2}} \lesssim \eps^2 \langle t\rangle^{-\alpha},
\]
as desired.

Next, we estimate
\begin{align*}
\| \langle x\rangle^{\alpha-[\frac12-\frac1r]} Q_\omega^{p-1}v^2 \|_{L^{r',2}} & \lesssim \| \langle x\rangle^{-\frac{r-3}{r-2}\alpha}v\|_{L^{2r',4}}^2 \|\langle x\rangle^{\alpha-[\frac12-\frac1r]+\frac{2(r-3)}{r-2}\alpha}Q_{\omega}^{p-1}\|_{L^\infty} \\
& \lesssim \| \langle x\rangle^{-\frac{r-3}{r-2}\alpha}v\|_{L^{2r',2}}^2 \|Q_\omega\|_Z^{p-1}.
\end{align*}
To arrive at a suitable estimate, it therefore remains to show that
\[
\|\langle x\rangle^{-\frac{r-3}{r-2}\alpha}v\|_{L^{2r',2}} \lesssim \eps \langle t\rangle^{-\frac{\alpha}{2}}.
\]
We first have the following estimate by interpolation: 
\begin{align*}
\|\langle x\rangle^{-\frac{r-3}{r-2}\alpha}v\|_{L^{2r',2}} & \lesssim \| \langle x\rangle^{-\alpha}v \|_{L^2}^{\frac{r-3}{r-2}}\|v\|_{L^{r,2}}^{\frac{1}{r-2}} \\
& \lesssim \eps\langle t\rangle^{-\frac{r-3}{r-2}\alpha -\frac{1}{r-2}(\frac12-\frac1r)}.
\end{align*}
In particular, we need to show that 
\[
-\tfrac{r-3}{r-2}\alpha -\tfrac{1}{r-2}(\tfrac12-\tfrac1r)<-\tfrac\alpha2.
\]
In fact, the left-hand side converges to $-\alpha$ as $r\to\infty$, and so this holds for any large $r$.

Collecting the estimates above, this completes the estimation of the $L^{r',2}$-norm of $\N$.

We turn to the contribution of the modulation parameters to $\F$. By \eqref{bs2}, \eqref{bs3}, and H\"older's inequality, we have
\begin{align*}
\|\langle x\rangle^{\alpha-[\frac12-\frac1r]}& \dot\theta Q_{\omega}\|_{L^{r',2}} + \|\langle x\rangle^{\alpha-[\frac12-\frac1r]}\dot\omega \partial_\omega Q_\omega\|_{L^{r',2}} \\
& \lesssim_{C_2}\langle t\rangle^{-2\alpha}\eps^2\|Q_\omega\|_Z \lesssim_{\omega_0,C_2}\langle t\rangle^{-2\alpha}\eps^2,
\end{align*}
which is acceptable. 

We next use \eqref{bs1}, \eqref{bs2}, and \eqref{bs3} to estimate
\begin{align*}
\|\langle & x\rangle^{\alpha-[\frac12-\frac1r]}[Q_\omega^p-Q_{\omega(T)}^p]z\|_{L^{r',2}} \\ 
& \lesssim \|\langle x\rangle^{-\alpha}z\|_{L^2} \|\langle x\rangle^4[Q_\omega^p-Q_{\omega(T)}^p]\|_{L^{\frac{2r}{r-2}}} \\
& \lesssim_{C_1,p}\eps  \langle t\rangle^{-\alpha}\biggl\|\langle x\rangle^4  \int_t^T \dot\omega(s)\cdot Q_{\omega(s)}^{p-1}\partial_\omega Q_{\omega(s)}\,ds\biggr\|_{H_x^1} \\
& \lesssim_{C_1,C_2,p} \eps^3\langle t\rangle^{-\alpha} \int_t^T \langle s\rangle^{-2\alpha} \|\langle x\rangle^4 Q_{\omega(s)}^{p-1} \partial_\omega Q_{\omega(s)}\|_{H_x^1}\,ds \\
& \lesssim_{\omega_0,C_1,C_2,p} \eps^3\langle t\rangle^{-3\alpha+1},
\end{align*}
which is acceptable. 

Finally, we use \eqref{bs1}, \eqref{bs2}, and \eqref{control-phase} to estimate
\begin{align*}
\| \langle x\rangle^{\alpha-[\frac12-\frac1r]}|e^{2i\Phi-2i\Psi}-1]Q_{\omega}^p z \|_{L^{r',2}} & \lesssim |\Phi-\Psi|\cdot \|\langle x\rangle^{-\alpha}z(t)\|_{L^2}\|\langle x\rangle^4 Q_{\omega}^p\|_{L^{\frac{2r}{r-2}}} \\
& \lesssim_{\omega_0,C_1,C_2,p} \eps^3 \langle t\rangle^{-\alpha}\|\langle x\rangle^4 Q_{\omega}^p\|_{H^1} \\
& \lesssim_{\omega_0,C_1,C_2} \eps^3\langle t\rangle^{-\alpha},
\end{align*}
which is acceptable.  This completes the estimate of the weighted norm.

It remains to estimate $\F$ in $L^2$. Again, we begin with the contribution of $\N$, relying on the pointwise estimate \eqref{PW}.  We begin by estimating
\[
\| |v|^{p+1}\|_{L^2} \lesssim \|v\|_{L^{2(p+1)}}^{p+1} \lesssim \eps^{p+1}\langle t\rangle^{-\frac{p}{2}},
\]
where we have used interpolation between the $L^2$ and $L^r$ norms of $v$ (and assumed $r>2(p+1)$).  As $\alpha<\frac{p}{2}$, the contribution of this term is acceptable.  On the other hand, by interpolation,
\begin{align*}
\| v^2 Q_\omega^{p-1}\|_{L^2} &  \lesssim \| \langle x\rangle^{-\alpha[\frac{r-4}{2(r-2)}]}v\|_{L^4}^2 \|\langle x\rangle^4 Q_{\omega}^{p-1}\|_{L^\infty} \\ 
& \lesssim_{\omega_0} \| \langle x\rangle^{-\alpha}v\|_{L^2}^{\frac{r-4}{r-2}}\|v\|_{L^r}^{\frac{r}{r-2}} \\
& \lesssim_{\omega_0} \eps^2 \langle t\rangle^{-\alpha[\frac{r-4}{r-2}]-[\frac12-\frac1r][\frac{r}{r-2}]}.
\end{align*}
As the power tends to $-\alpha-\frac12$ as $r\to\infty$, we find that this contribution is acceptable provided $r$ is sufficiently large. 

It remains to estimate the contribution of the modulation parameters. First,
\[
\|\dot\theta Q_{\omega}\|_{L^2}+\|\dot\omega\partial_\omega Q_\omega\|_{L^2} \lesssim_{C_2}\eps^2\langle t\rangle^{-2\alpha}\|Q_{\omega}\|_{Z} \lesssim_{\omega_0,C_2} \eps^2\langle t\rangle^{-2\alpha},
\]
which is acceptable. Next, estimating as we did above, we have
\begin{align*}
\|[Q_\omega^p - Q_{\omega(T)}^p]z\|_{L^2}& \lesssim \|\langle x\rangle^{-\alpha}z\|_{L^2}\int_t^T \langle s\rangle^{-2\alpha}\|Q_{\omega(s)}^{p-1}\partial_\omega Q_{\omega(s)}\|_{H^1}\,ds \\
& \lesssim_{\omega_0,C_1,C_2,p}\eps^3\langle t\rangle^{-3\alpha+1},
\end{align*}
which is acceptable. Finally, estimating again as we did above,
\[
\|[e^{2i\Phi-2i\Psi}-1]Q_\omega^p z\|_{L^2} \lesssim |\Phi-\Psi|\cdot\|\langle x\rangle^{-\alpha}z\|_{L^2} \|Q_\omega\|_Z^p \lesssim_{\omega_0,C_1,C_2} \eps^3\langle t\rangle^{-\alpha}.
\]
This completes the proof.  \end{proof}

%%%%%%%%%%%%%%%%%%%%%%%%%%
%%%%%%%%%%%%%%%%%%%%%%%%%%

We can now establish our desired estimate for $\eta$. 

\begin{lemma}\label{L:control-eta} We have the following estimate for $\eta$:
\[
\|\eta\|_{X([0,T])} \lesssim_{\omega_0} \|\langle x\rangle^\alpha v(0)\|_{L^2}+ \mathcal{O}_{p,\omega_0,C_1,C_2}(\eps^2)
\]
\end{lemma}

\begin{proof} We will utilize the Duhamel formula in \eqref{nls-eta} and apply dispersive estimates adapted to $e^{t\L(\omega(T))}P_c(T)$ (see \eqref{dispersive4} and \eqref{dispersive7}).  By \eqref{control-omegat}, we may guarantee that the implicit constants appearing in the dispersive estimates for all $\omega(t)$ are all bounded by a constant depending only on $\omega_0$.  Also, as $\alpha<\tfrac32$ by assumption, we may guarantee $\alpha<\tfrac32-\tfrac3r$ (which is necessary to apply \eqref{dispersive7}) provided $r$ is chosen sufficiently large. 

We first consider the $L^{r,2}$-norm.  We apply \eqref{dispersive4}, H\"older's inequality, and Lemma~\ref{L:controlF} to obtain
\begin{align*}
|t|^{\frac12-\frac1r}\|\eta(t) \|_{L^{r,2}} & \lesssim_{\omega_0} \|z(0)\|_{L^{r',2}} + |t|^{\frac12-\frac1r}\int_0^t |t-s|^{-[\frac12-\frac1r]}\|\F(s)\|_{L^{r',2}}\,ds \\
& \lesssim_{\omega_0} \|\langle x\rangle^{\frac12-\frac1r}z(0)\|_{L^2} +\eps^2 |t|^{\frac12-\frac1r}\int_0^t |t-s|^{-[\frac12-\frac1r]}\langle s\rangle^{-\alpha}\,ds \\
& \lesssim_{\omega_0} \|\langle x\rangle^{\alpha}v(0)\|_{L^2} + \mathcal{O}_{p,\omega_0,C_1,C_2}(\eps^2)
\end{align*}
uniformly over $t\in[0,T]$, which is acceptable.  For the weighted norm, we use both \eqref{dispersive7} and \eqref{dispersive3} to arrive at the estimate
\begin{align*}
|t|^{\alpha}&\|\langle x\rangle^{-\alpha}\eta(t)\|_{L^2}  \lesssim_{\omega_0} \|\langle x\rangle^{\alpha-[\frac12-\frac1r]}z(0)\|_{L^{r',2}} \\
& \quad + |t|^{\alpha}\int_0^t \min\{|t-s|^{-\alpha},1\}\bigl\{\|\langle x\rangle^{\alpha-[\frac12-\frac1r]}\F\|_{L^{r',2}}+\|\F\|_{L^2}\bigr\}\,ds \\
& \lesssim_{\omega_0} \|\langle x\rangle^{\alpha}v(0)\|_{L^2} + \eps^2|t|^{\alpha}\int_0^t \min\{|t-s|^{-\alpha},1\}\langle s\rangle^{-\alpha}\,ds \\
& \lesssim_{\omega_0}  \|\langle x\rangle^{\alpha}v(0)\|_{L^2} + \mathcal{O}_{p,\omega_0,C_1,C_2}(\eps^2) 
\end{align*}
uniformly over $t\in[0,T]$, which is acceptable.  This completes the proof.\end{proof}

%%%%%%%%%%%%%%%%%%%%%%%%%%
%%%%%%%%%%%%%%%%%%%%%%%%%%

With the preceding lemmas in place, we can complete the bootstrap argument.

\begin{proposition}[Completing the bootstrap]\label{PT1} There exist $C_1,C_2>0$ depending on $\omega_0$ and $p$ such that for $\eps=\eps(\omega_0,p)>0$ sufficiently small and $\eta=\eta(\eps)>0$ sufficiently small, if
\begin{equation}\label{cbs-data}
\|\tilde v_0\|_{H^1}+\|\langle x\rangle^\alpha \tilde v_0\|_{L^2} \leq \eta,
\end{equation}
then the estimates \eqref{bs1} and \eqref{bs2} hold uniformly for all $T>0$. 
\end{proposition}

\begin{proof} We first claim that the estimates \eqref{bs1} and \eqref{bs2} hold in a neighborhood of $t=0$.  We first consider the estimate \eqref{bs1}.  For this, we note that for small times we may control the $X$-norm by the $L_t^\infty H_x^1$-norm, which is $\mathcal{O}_{\omega_0}(\eps)$ by \eqref{control-v-H1}.  In particular, we need only choose $C_1$ sufficiently large depending on $\omega_0$.  The estimate \eqref{bs1} for small times follows from the estimate \eqref{bs1} via \eqref{control-omega1}, choosing $C_2$ sufficiently large depending on $p,\omega_0,C_1$. 

Next, we suppose that the estimates \eqref{bs1} and \eqref{bs2} hold on an interval $[0,T]$.  We will then show that for $\eps>0$ sufficiently small, we have the improved estimates
\begin{align}
\|v\|_{X([0,T])} &\leq C_1\eps, \label{bs11} \\
\|(\omega,\theta)\|_{Y([0,T])}&\leq C_2\eps^2. \label{bs22}
\end{align}
A standard continuity argument then implies the result. 

Recall that for $\eps>0$ small enough, the condition \eqref{control-omegat} guarantees that the implicit constant in the dispersive estimates for $\L(\omega(t))$ is bounded uniformly in $t$. In particular, the implicit constant in Lemma~\ref{L:control-eta} is independent of $T$. 

We first consider \eqref{bs22}.  As above, by \eqref{control-omega1} we may guarantee that \eqref{bs22} holds by choosing $C_2$ sufficiently large depending on the parameters $p,\omega_0,$ and  $C_1.$ 

Finally, we consider \eqref{bs11}.  We first apply Lemma~\ref{L:control-eta} and \eqref{control-v0}, which yields
\[
\|\eta\|_{X([0,T])} \lesssim_{\omega_0} \eps + \mathcal{O}_{p,\omega_0,C_1,C_2}(\eps^2).
\]
In particular, for $\eps=\eps(p,\omega_0,C_1,C_2)$ sufficiently small, we may guarantee that
\[
\|\eta\|_{X([0,T)} \lesssim_{\omega_0}\eps. 
\]
Now using \eqref{eta-controls-v} and choosing $C_1$ sufficiently large depending on $\omega_0$, we may derive \eqref{bs11}.  This completes the proof. \end{proof}

With the bootstrap closed and the global estimates for $v$ and $(\omega,\theta)$ in place, it is now straightfoward to prove that $v(t)$ scatters in $L^2$ as $t\to\infty$. 

\begin{proposition}[Scattering]\label{PT2} There exists $\omega_+$ (with $|\omega_0-\omega_+|\lesssim \eps^2$) and $v_+\in L^2$ such that 
\begin{equation}\label{scatter}
\lim_{t\to\infty} \|v(t) - e^{t\L(\omega_+)}v_+\|_{L^2} = 0. 
\end{equation}
\end{proposition}

\begin{proof} We have the decomposition
\[
u(t) = e^{i\Phi(t)}\{Q_{\omega(t)}+v(t)\}, \quad \Phi(t)=\theta(t)+\int_0^t \omega(s)\,ds,
\]
for all $t\geq 0$, with $v$ and $(\theta,\omega)$ obeying the estimates \eqref{bs1} and \eqref{bs2} on the entire interval $[0,\infty)$.  We now carry out a constriction on $[0,\infty)$ analogous to the one we used to prove estimates on $[0,T]$.  In particular, we set
\[
\theta_+ = \lim_{t\to\infty}\theta(t) \qtq{and} \omega_+ = \lim_{t\to\infty}\omega(t).
\]
Observe that by using the explicit decay rates obtained, we may also define the limit
\[
\gamma = \int_0^\infty [\omega(s)-\omega_+]\,ds. 
\]

We then define $\Psi$ and $z$ via 
\[
\Psi(t) = \theta_+ + t\omega_+ \qtq{and} e^{i\Psi(t)}z(t) = e^{i\Phi(t)}v(t),
\]
and we formulate the equation solved by $z$ with linear operator given by $\L(\omega_+)$.  In particular, we obtain the PDE \eqref{nls-z} with \eqref{D:F}, with $\omega(T)$ replaced by $\omega_+$.  We then introduce the variable $\eta_+$, with coefficients $a_+$ and $b_+$, as in \eqref{z=abeta}.  In particular $\eta=P_c(\omega_+)z$ and $\eta$ evolves according to \eqref{nls-eta}, with $\omega(T)$ replaced by $\omega_+$. 

Now the estimates appearing in Lemma~\ref{L:control-eta} suffice to prove that $e^{-t\mathcal{L}(\omega_+)}\eta(t)$ forms a Cauchy sequence in $L^2$ and hence has a limit $\eta_+\in L^2$.  In light of the estimates on $a_+$ and $b_+$ (see \eqref{control-ab}), this implies $z(t)=e^{t\L(\omega_+)}\eta_+ + o(1)$ as $t\to\infty$.  This in turn yields
\[
v(t) = e^{i[\Psi_+(t)-\Phi(t)]}e^{t\L(\omega_+)}\eta_+ + o(1) = e^{i\gamma} e^{t\L(\omega_+)}\eta_+ + o(1)
\]
in $L^2$ as $t\to\infty$, where we have used the dominated convergence theorem as well.  The result now follows with $v_+ = e^{i\gamma}\eta_+$. \end{proof}

Collecting the results of Propositions~\ref{P:decomposition}, \ref{PT1}, and \ref{PT2}, we have now completed the proof of Theorem~\ref{T}. 

\appendix

%%%%%%%%%%%%%%%%%%%%%%%%%%%%%%%%%%%%%%
\section{Construction of ODE solutions}\label{S:F1F4}

In this section we discuss the construction of the solutions to the generalized eigenvalue problem \eqref{gep} ($f_1,f_2,f_3$ and $\tilde f_4$ appearing in Lemma~\ref{L:f1-f4}).  The construction is similar to what appears in \cite[Lemmas~5.2, 5.3, 5.5]{KS}, with the main differences 
arising from the presence of the jump condition at $x=0$. Accordingly, we shall be somewhat brief in our presentation.  

We let $H_q=-\frac12 \d_x^2 + q \delta_0$ and $\mathcal{H}_V=\mathcal{H}_q + \mathcal{V}$, where
\[
	\mathcal{H}_q := 
	\begin{pmatrix}
	H_q + \omega & 0 \\ 0 & - H_q -\omega
	\end{pmatrix}, \quad
	\mathcal{V} :=
	\sigma_3\begin{pmatrix}
	V_1(x) & V_2(x) \\  V_2(x) &  V_1(x)
	\end{pmatrix}.
\]
For now, we regard $\omega>0$ just as a parameter.
We let 
\[
	\mu = \mu(\omega):= \sqrt{\xi^2 + 4\omega}.
\]
In this section, the precise form of the potential $\mathcal{V}$ appearing in the operator is not important.  We only assume that $V_j(x) \in C^\I(\R\setminus\{0\})$ are even, real-valued functions such that
\[
\sum_{j=0}^k (|\d_x^j V_1(x)| + |\d_x^j V_2(x)|) \lesssim_{k} e^{-\gamma |x|}
\]
holds for some $\gamma>0$ and any $k\geq 0$.  In our setting, $V_1=\sigma\frac{p+1}{2}Q_{\omega}^p$ and 
$V_2=\sigma\frac{p}{2}Q_{\omega}^p$, so that the assumption above is satisfied; furthermore, the implicit constants (as well as $\gamma$) depend continuously on $\omega$.  (Recall that we needed to assert continuous dependence in Proposition~\ref{P:D}.) 

We consider the ODE system
\begin{equation}\label{E:ODEs}
	\mathcal{H}_V f =  \(\frac{\xi^2}2 + \omega \)f , \quad f(x;\xi) : \R^2 \to \C^2.
\end{equation}
By the standard ODE theory, one sees that the initial value problem of \eqref{E:ODEs} has solutions.  Our goal is to find four solutions $f_1$, $f_2$, $f_3$, and $f_4$ satisfying the asymptotic conditions
\begin{align*}
	&f_1(x;\xi) \to e^{i\xi x} 	\begin{pmatrix}
	1 \\ 0
	\end{pmatrix}, 
	\qquad f_2(x;\xi) \to e^{-i\xi x} 	\begin{pmatrix}
	1 \\ 0
	\end{pmatrix}, \\
&f_3(x;\xi) \to e^{-\mu x} 	\begin{pmatrix}
	0 \\ 1
	\end{pmatrix}, 
\qquad f_4(x;\xi) \to e^{\mu x} 	\begin{pmatrix}
	0 \\ 1
	\end{pmatrix}
\end{align*}
as $x\to\I$.
Before beginning the construction of our specific solutions, we collect some basic facts about the solution space.

We first note that a solution to the homogeneous equation
\[
	\mathcal{H}_q f=  \(\tfrac12{\xi^2} + \omega \)f, 
	\quad f(0) =\begin{pmatrix}a_1 \\ a_2\end{pmatrix}, 
	\quad f'(0)=\begin{pmatrix}b_1 \\ b_2\end{pmatrix}
\]
is given by
\[
	f(x) = D_\xi' (x) 
	\begin{pmatrix}a_1 \\ a_2\end{pmatrix}
  + D_\xi (x) \begin{pmatrix}b_1 \\ b_2\end{pmatrix}
,
\]
where
\[
	D_\xi (y) = \int_0^y \begin{pmatrix}
	\cos \xi s & 0 \\ 0 & \cosh \mu s
	\end{pmatrix} ds = \begin{pmatrix}
	\frac{\sin \xi y}{\xi} & 0 \\ 0 & \frac{\sinh \mu y}{\mu}
	\end{pmatrix},
\]
with $\mu=\sqrt{\xi^2 + 4\omega}$ as above.

Next, we have the following result.
\begin{proposition}
For each $\xi \in \R$,
The solution space $X(\xi)$ to \eqref{E:ODEs} is isomorphic to $\C^4$.
Moreover, for any fixed $x_0 \in \R \setminus \{0\}$, the map
\[
X(\xi) \ni f= \begin{pmatrix}f^{(1)}(x;\xi)  \\  f^{(2)}(x;\xi) \end{pmatrix} \mapsto \begin{pmatrix}
f^{(1)}(x_0 ;\xi)  \\  f^{(2)}(x_0;\xi)\\ \d_x f^{(1)}(x_0 ;\xi)  \\  \d_x f^{(2)}(x_0;\xi) \end{pmatrix} \in \C^4
\]
is an isomorphism. Similarly,
\[
X(\xi) \ni f= \begin{pmatrix}f^{(1)}(x;\xi)  \\  f^{(2)}(x;\xi) \end{pmatrix} \mapsto \begin{pmatrix}
f^{(1)}(0 ;\xi)  \\  f^{(2)}(0;\xi)\\ \d_x f^{(1)}(0+ ;\xi)  \\  \d_x f^{(2)}(0+;\xi) \end{pmatrix} \in \C^4
\]
and
\[
X(\xi) \ni f= \begin{pmatrix}f^{(1)}(x;\xi)  \\  f^{(2)}(x;\xi) \end{pmatrix} \mapsto \begin{pmatrix}
f^{(1)}(0 ;\xi)  \\  f^{(2)}(0;\xi)\\ \d_x f^{(1)}(0- ;\xi)  \\  \d_x f^{(2)}(0-;\xi) \end{pmatrix} \in \C^4
\]
are also isomorphisms.  Here $\pm$ denote limits from the right/left.
\end{proposition}

\begin{proof} This is just a rephrasing of the existence of a unique solution for data given at a fixed time.  Thus it is very standard if we do not have the delta potential (that is, if $q=0$). Incorporating the delta potential requires that we add the jump condition
\[
	\begin{pmatrix}
	f^{(1)}(0 ;\xi)  \\  f^{(2)}(0;\xi)\\ \d_x f^{(1)}(0+ ;\xi)  \\  \d_x f^{(2)}(0+;\xi) \end{pmatrix}
	=
	\begin{pmatrix}
	1 & 0 & 0 & 0\\
	0 & 1 & 0 & 0\\
	2q & 0 & 1 & 0 \\
	0 & 2q & 0 & 1 
	\end{pmatrix}
	\begin{pmatrix}
	f^{(1)}(0 ;\xi)  \\  f^{(2)}(0;\xi)\\ \d_x f^{(1)}(0- ;\xi)  \\  \d_x f^{(2)}(0-;\xi) \end{pmatrix}.
\]
As the matrix in the right hand side is invertible, the unique solvability follows by gluing together solutions in the positive and negative regions.
\end{proof}

\begin{corollary}
Fix $\xi\in\R$. For solutions $\{f_j\}_{j\in J} \subset X(\xi)$ to \eqref{E:ODEs}, the following statements are equivalent
\begin{enumerate}
\item The solutions $\{f_j\}_{j\in J}$ are linearly independent;
\item For some $x_0 \in \R \setminus\{0\}$, vectors
	\[
	\left\{  \begin{pmatrix} f_j^{(1)} (x_0)\\ f_j^{(2)}(x_0) \\ \d_x f_j^{(1)}(x_0) \\ \d_x f_j^{(2)}(x_0) \end{pmatrix}  \right\}_{j\in J} \subset \C^4
	\]
	are linearly independent;
\item The statement (2) with $x_0$ replaced by $0-$;
\item The statement (2) with $x_0$ replaced by $0+$.
\end{enumerate}
\end{corollary}

We next introduce a `scalar' Wronskian.

\begin{proposition}[Scalar Wronskian] Fix $\xi \in \R$.  For two solutions $f_1, f_2 \in X(\xi)$, the Wronskian
\[
	W[f_1,f_2] (x) := \begin{pmatrix} f_1^{(1)} (x) \\ f_1^{(2)} (x) \\ \d_x f_1^{(1)} (x) \\ \d_x f_1^{(2)} (x) \end{pmatrix}^t
	\begin{pmatrix}
	0 & 0 & -1 & 0 \\
	0 & 0 & 0 & 1 \\
	1 & 0 & 0 & 0 \\
	0 & -1 & 0 & 0
	\end{pmatrix}
	\begin{pmatrix} f_2^{(1)} (x) \\ f_2^{(2)} (x) \\ \d_x f_2^{(1)} (x) \\ \d_x f_2^{(2)} (x) \end{pmatrix}
\]
is a constant function for $x\neq0$.  We denote the constant by $W[f_1,f_2]$.
\end{proposition}

\begin{proof}
By a direct calculation, one can see that $\frac{d}{dx}W[f_1,f_2]=0$ for $x\neq0$.
Further, since
\[
	\begin{pmatrix}
	1 & 0 & 0 & 0\\
	0 & 1 & 0 & 0\\
	2q & 0 & 1 & 0 \\
	0 & 2q & 0 & 1 
	\end{pmatrix}^t
	\begin{pmatrix}
	0 & 0 & -1 & 0 \\
	0 & 0 & 0 & 1 \\
	1 & 0 & 0 & 0 \\
	0 & -1 & 0 & 0
	\end{pmatrix}	
	\begin{pmatrix}
	1 & 0 & 0 & 0\\
	0 & 1 & 0 & 0\\
	2q & 0 & 1 & 0 \\
	0 & 2q & 0 & 1 
	\end{pmatrix}
	=	\begin{pmatrix}
	0 & 0 & -1 & 0 \\
	0 & 0 & 0 & 1 \\
	1 & 0 & 0 & 0 \\
	0 & -1 & 0 & 0
	\end{pmatrix},
\]
we have $W[f_1,f_2](0+) = W[f_1,f_2](0-)$. %This shows the desired result.
\end{proof}

In the scalar equation case, one can investigate linear independence of two solutions by Wronskian. This is not the case for ODE systems (consider the vectors $(1\ 0\ 0 \ 0)^t$ and $(0\ 1\ 0 \ 0)^t$).

\begin{proposition} Fix $\xi \in \R$. For a pair of solutions $f_1,f_2 \in X(\xi)$, consider the following two statements:
\begin{enumerate}
\item $f_1$ and $f_2$ are linearly dependent;
\item $W[f_1,f_2]=0$.
\end{enumerate}
Then, (i) implies (ii), but (ii) does not imply (i).
\end{proposition}

We turn to the construction of the specific solutions to \eqref{E:ODEs}.

\subsection{Construction of $f_3$ for $x>0$}

Let us first consider $f_3$, which should solve the integral equation
\begin{equation}\label{E:f3int}
	f_3(x;\xi) = e^{-\mu x} \begin{pmatrix}
	0 \\ 1
	\end{pmatrix}
	- 2\int_x^\I D_\xi (y-x) \sigma_3  \mathcal{V}(y) f_3(y;\xi) \,dy
\end{equation}
for $x>0$.  Note that $f_1$, $f_2$, and $f_4$ will not solve the above integral equation; indeed, the the integral is not finite due to the asymptotic behavior of these solutions and the kernel $D_\xi$.

We will solve the above equation for $x\in \R$. In addition to constructing $f_3$ on $x>0$, we will use this solution in the construction of $f_1$ and $f_2$ below. 

\begin{proposition}[Construction of $f_3(x;\xi)$ for $x>0$]\label{P:f3 plus}
For all $\xi \in \R$, there exists a unique solution $f_3(x;\xi)$ to \eqref{E:f3int} (on $\R$) with the asymptotic condition
\[
	\abs{ e^{\mu x} f_3(x;\xi) -
	\begin{pmatrix}
	0 \\ 1
	\end{pmatrix} }
	\lesssim_{\omega} \mu^{-1} \min(1,e^{-\gamma x}) .
\]
The solution is smooth in both variables and
satisfies 
\[
	|\d_\xi^l \d_x^k (e^{\mu x} f_3(x;\xi))| \lesssim_{\omega} 
	\mu^{-l-1} \min(1, e^{- 2\gamma x}), 
\]
for $l,k \geq 0$ with $l+k\geq1$.  Furthermore, the implicit constants depend continuously on $\omega$. 

The restriction of $f_3(x;\xi)$ to the region $x>0$ coincides with any solution to \eqref{E:ODEs}
which satisfies the above asymptotic condition for $x>0$.
\end{proposition}

The proof of Proposition~\ref{P:f3 plus} follows as in \cite{KS}.  We briefly recall the ideas:  To solve \eqref{E:f3int}, we introduce a modified kernel
\[
	K_\xi (z) = e^{-\mu z} D_\xi (z),
\]
which is bounded for $z\geq 0$.
Then, \eqref{E:f3int} is written as
\begin{equation}\label{E:f3int1}
	e^{\mu x} f_3(x;\xi) =
	\begin{pmatrix}
	0 \\ 1
	\end{pmatrix}
	- 2\int_x^\I K_\xi (y-x) \sigma_3 \mathcal{V}(y) e^{\mu y}f_3(y;\xi) \,dy.
\end{equation}
This is a Volterra equation, and the result follows from standard arguments. 

\subsection{Construction of $f_1$ and $f_2$ for $x>0$}

Let us next consider the solutions $f_1(x;\xi)$ and $f_2(x;\xi)$, with our main focus on $f_1(x;\xi)$. It is difficult to formulate \eqref{E:ODEs} as a Volterra type integral equation (like we did for \eqref{E:f3int}), because the integral kernel $D_\xi$ grows exponentially if $\mu > \gamma$.  Nonetheless, we can show the following.

\begin{proposition}[Construction of $f_1(x;\xi)$ for $x>0$]\label{P:f12 plus}
For all $\xi \in \R$, there exists a unique solution $f_1(x;\xi)$ to \eqref{E:ODEs} with the asymptotic condition
\[
	\abs{ e^{-i \xi x } f_1(x;\xi) -
	\begin{pmatrix}
	1 \\ 0
	\end{pmatrix} } 
	\lesssim_{\omega} \mu^{-1} \Jbr{x} e^{-x \min (\gamma ,\mu)} 
\]
for $x> 0$.
Moreover, we have the estimate
\[
	\abs{ \d_\xi^l \d_x^k( e^{-i \xi x } f_1(x;\xi)) } 
	\lesssim_{k,l,\omega} \mu^{-1} \Jbr{x}^{l+1} e^{- x \min (\gamma ,\mu)}
\]
for $x >0$ and for $k,l\geq 0$ with $k+l\geq 1$.  Furthermore, the implicit constants in the above inequalities depend continuously on $\omega$.
\end{proposition}

We again follow the strategy set out in \cite{KS} (see Lemma~5.3 therein).  We introduce the following ansatz:
\[
f_1(x;\xi) = \begin{pmatrix} f_1^{(1)}(x;\xi) \\ f_1^{(2)}(x;\xi) \end{pmatrix} := \begin{pmatrix} 1 \\ 0 \end{pmatrix} e^{i\xi x} v(x;\xi) + \begin{pmatrix} f_3^{(1)}(x;\xi) \\ f_3^{(2)}(x;\xi) \end{pmatrix} u (x;\xi),
\]
where $u$ and $v$ are new scalar unknowns.  Here $f_3(x;\xi)$ is the solution to \eqref{E:f3int} on $\R$ constructed in the previous section.
To have the desired asymptotic behavior for $f_1$, we must impose that
\begin{equation}\label{E:cond for uv}
v(x;\xi) \to 1, \quad f_3^{(2)}(x;\xi) u (x;\xi) \to 0
\end{equation}
as $x\to\I$.  We will eventually obtain the following estimate on $v$:
\begin{equation}\label{E:sys for uv6}
|\d_\xi^l \d_x^k v(x;\xi) | \lesssim_{k,l} \mu^{-1} \Jbr{x}^{l} e^{-\gamma x},
\end{equation}
which provides a slight improvement over the analogous estimate appearing in \cite{KS} (as we do not lose powers of $\mu$ under differentiation with respect to $x$).

We derive a system of equations for $u$ and $v$.
For $x>0$, \eqref{E:ODEs} is the same as
\begin{equation}\label{E:sys for uv0}
	\begin{pmatrix} -\frac12 \d_x^2 -\frac{\xi^2}2 + V_1 & V_2 
	\\ -V_2 & \frac12 \d_x^2 - \frac{\mu^2}2 -V_1 \end{pmatrix}
	f_1(x;\xi) = 0.
\end{equation}
As in Proposition \ref{P:f3 plus}, we consider the equation for $x\in \R$.

We first show that $u$ may be obtained in terms of $v$.  In view of \eqref{E:cond for uv}, the second coordinate of
 \eqref{E:sys for uv0} implies that

\begin{equation}\label{E:sys for uv2}
	u'(x;\xi) = -2 (f_3^{(2)}(x;\xi))^{-2}  \int_x^\I e^{iy\xi} V_2(y)f_3^{(2)}(y;\xi)  v(y;\xi) \,dy
\end{equation}
provided $f_3^{(2)}(x;\xi)\neq0$.  Thus we have the following formula:
\begin{equation}\label{E:sys for uv35}
		u(x;\xi) = \int_{x_0}^x 2 (f_3^{(2)}(y;\xi))^{-2}  \int_y^\I e^{iz\xi} V_2(z)f_3^{(2)}(z;\xi)  v(z;\xi) \,dz \,dy ,
\end{equation}
where $x_0 = x_0(\xi)$ will be chosen below. Note that this formula requires $f_3^{(2)}(x;\xi) \neq 0$ for all $x\geq x_0$.

We deviate from the approach of \cite{KS} in our choice of $x_0$.  In \cite{KS}, the authors rely on the follow estimate (due to Proposition \ref{P:f3 plus}):
\[
\abs{ e^{\mu x} f_3(x;\xi) -
	\begin{pmatrix}
	0 \\ 1
	\end{pmatrix} }
	\lesssim \mu^{-1} \max(1,e^{-\gamma x}).
\]
In particular, this estimate shows that there exists a constant $\tilde{x}>0$ such that
\[
	\inf_{\xi \in \R} \inf_{x\geq \tilde{x}} |e^{\mu x} f_3^{(2)}(x;\xi)| \geq \tfrac12.
\]
The authors of \cite{KS} then choose $x_0(\xi)\equiv \tilde{x}$. In our case, we rely on the fact that there exists $R>0$ such that 
\[
	\inf_{|\xi| \geq R} \inf_{x\in \R} |e^{\mu x} f_3^{(2)}(x;\xi)| \geq \tfrac12,
\] 
which implies that the choice of $x_0(\xi)$ for $|\xi|\geq R$ may be taken arbitrarily. In fact, the authors of \cite{KS} already observed this and choose $x_0=0$ for large $\xi$.  We take this one step further and choose \emph{negative} $x_0$ for large $\xi$.  This will allow us to derive the slightly improved estimate. 

In summary, we choose $x_0=x_0(\xi) \in C^\I(\R)$ to be an even function such that
\begin{equation}\label{E:def of x0}
	x_0 = x_0(\xi)= \begin{cases} \tilde{x} & |\xi| \leq R , \\
	-1 & |\xi | \geq 2R
	\end{cases}
\end{equation}
and $-1 \leq x_0(\xi) \leq \tilde{x}$ for all other values of $\xi$, where $\tilde x>0$ and $R>0$ are as above.

We now need to derive an integral equation for the unknown scalar $v$.  From the first component of \eqref{E:sys for uv0}, we obtain the relation
\begin{equation}\label{E:sys for uv4}
	- \frac12 v'' - i\xi  v' + V_1  v + e^{-ix\xi}\( \frac12 u''  f_3^{(1)}  + u' (f_3^{(1)})'\)  = 0.
\end{equation}
We may rewrite this equation as 
\[
	v(x) = 1 + \int_x^\I E_{\xi}(y-x) \left[ 2V_1(y)v(y;\xi) + e^{-iy\xi}\( u''  f_3^{(1)}  + 2 u' (f_3^{(1)})'\)(y)  \right]\,dy,
\]
where
\[
	E_{\xi}(z) := \int_0^z e^{2is\xi} ds = \frac{e^{2iz\xi}-1}{2i\xi}.
\]
Using \eqref{E:sys for uv2} we derive
\begin{equation}\label{E:sys for uv5}
	\begin{aligned}
	v(x) 
	={}& 1 + \int_x^\I  (K_1(x, y)+ K_2(x,y))  v(y;\xi) \,dy,
	\end{aligned}
\end{equation}
where %$A$ is given in \eqref{E:def A},
\begin{equation}\label{E:def A}
	A(x;\xi) := \frac{4}{(f_3^{(2)}(x;\xi))^3}(f_3^{(2)}(f_3^{(1)})'- (f_3^{(2)})'f_3^{(1)})(x;\xi),
\end{equation}
\[
	K_1(x,y) = 2 E_{\xi}(y-x) \( V_1(y) -\frac{f_3^{(1)}(y;\xi)}{f_3^{(2)}(y;\xi)} V_2(y) \),
\]
and
\[
	K_2(x,y) = V_2(y)f_3^{(2)}(y;\xi) \int_x^y E_{\xi}(z-x) A(z)  e^{i(y-z)\xi} \,dz .
\]
This is a Volterra-type equation. To complete the proof of Proposition \ref{P:f12 plus}, we need to construct $v$ as a solution to \eqref{E:sys for uv5}.

To this end, we first collect some estimates on $A$.  The key fact is that $A$ has the decaying factor $e^{-\mu x}$.

\begin{lemma}\label{L:f12 lem15}
For the solution $f_3(x;\xi)$ to \eqref{E:f3int} on $\R$ given in Proposition \ref{P:f3 plus}, we define
$A(x;\xi)$ as in \eqref{E:def A}.
Let $x_0$ be a number such that $|f_3(x;\xi)|>0$ holds for $x>x_0$.
Let $c:= \inf_{y\geq x_0}|e^{\mu y} f_3(y;\xi)|$.
Then, 
\[
	|\d_\xi^l \d_x^k(e^{-\mu x} A(x;\xi))| \lesssim_{k,l,c} \mu^{-l-1} \min (1, e^{-2\gamma x})
\]
for $k,l\geq 0$ and $x\geq x_0$.
If $f_3(x;\xi)>0$ for all $x\in \R$, the estimate is true for all $x\in\R$.
\end{lemma}

\begin{proof}
We have the formula
\[
	e^{-\mu x} A(x) = \frac{4}{(e^{\mu x} f_3^{(2)}(x;\xi) )^3} ((e^{\mu x} f_3^{(2)} )(e^{\mu x} f_3^{(1)})'-(e^{\mu x} f_3^{(2)} )'(e^{\mu x} f_3^{(1)} ))(x;\xi).
\]
Thus the lemma is a consequence of Proposition \ref{P:f3 plus}.
\end{proof}

We turn to problem of solving \eqref{E:sys for uv5}.  We begin by estimating the integral kernels.  For now, let us restrict to the interval $[x_0,\infty)$. 

We first note that
\[
	|E_\xi (z)| \leq 
	\left\{ 
	\begin{aligned}
	&\frac1{|\xi|}, && |\xi| \geq 1, \\
	& z, && |\xi| \leq 1
	\end{aligned}
	\right.
\]
 for $z\geq 0$.  Thus by Proposition \ref{P:f3 plus}, we have the following estimate for $K_1$:
\[
	\int_x^\I |K_1(x,y)| \,dy \lesssim \int_x^\I |E_\xi (x-y)|e^{-\gamma y} dx \lesssim \mu^{-1} e^{-\gamma x}. %\lesssim_\gamma \mu^{-1} e^{-\gamma y/2}.
\]

Let us next consider $K_2(x,y)$.  By Lemma \ref{L:f12 lem15},
\[
	|A(x;\xi)| \lesssim \mu^{-1} e^{(\mu-2\gamma) x},
\]
and so
\[
	\int_x^\I |K_2(x,y)|\,dy \lesssim \int_x^\I e^{-(\mu+\gamma)y} \int_x^y \mu^{-2} z e^{(\mu-2\gamma) z} \,dz \,dy
	%\lesssim_\gamma \mu^{-2} e^{-(\mu+\gamma)y} \int_x^y  e^{\mu z} \,dz 
	\lesssim \mu^{-2} e^{-\gamma x}
\]
for $x\geq x_0$.

As \eqref{E:sys for uv5} is a Volterra-type equation, we may now obtain a solution $v(x;\xi)$ on $[x_0,\I)$ by standard arguments (see e.g. \cite{KS}). Furthermore, we can obtain
\[
	|v(x;\xi) -1 | \lesssim \int_x^\I (|K_1(x,y)|+ |K_2(x,y)|) \,dy  \lesssim \mu^{-1} e^{-\gamma x}
\]
on this interval.

We next estimate $x$-derivatives, still restricting to the interval $[x_0,\infty)$.  As we claim a small improvement over \cite{KS}, we will estimate in some detail. Let $k\geq1$ and $l=0$ and observe that 
\begin{align*}
	&\d_x^{k} \int_x^\I  K_1(x, y) v(y;\xi) \,dy \\
% 	&{}= \int_x^\I 2 E_{\xi}(y-x) \d_y^k \( \( V_1(y) -\frac{f_3^{(1)}(y;\xi)}{f_3^{(2)}(y;\xi)} V_2(y) \) v(y,\xi)\) \,dy\\
	&{}= \sum_{j=0}^{k-1} \int_x^\I K_1^{(k,j)}(x,y) \d_y^j v(y,\xi) \,dy + \int_x^\I K_1(x,y) \d_y^k v(y,\xi) \,dy,
\end{align*}
where
\[
	K_1^{(k,j)}(x,y):= 2
%	{}_kC_j
	\begin{pmatrix} k \\ j \end{pmatrix}
	 E_{\xi}(y-x) \d_y^{k-j} \( V_1(y) -\frac{f_3^{(1)}(y;\xi)}{f_3^{(2)}(y;\xi)} V_2(y) \)
\]
for $0\leq j \leq k-1$.

By Proposition \ref{P:f3 plus}, we have
\[
	\int_x^\I |K_1^{(k,j)}(x,y)| \lesssim_k \int_x^\I E_\xi(y-x) e^{-\gamma y} \,dy \lesssim \mu^{-1} e^{-\gamma x}
\]
for $0\leq j \leq k-1$.
We also remark that
\begin{align*}
	&\d_x^k  \int_x^\I  K_2(x, y) v(y;\xi) \,dy \\
	&= \sum_{j=0}^k\sum_{m=0}^{\min (j,k-1)}  \int_x^\I K_2^{(k,j,m)}(x,y) \d_y^m v(y;\xi) \,dy + \int_x^\I K_2(x,y)\d_y^k v(y;\xi)\,dy,
\end{align*}
where
\begin{multline*}
	K_2^{(k,j,m)}(x,y) := 
% 	{}_kC_j {}_jC_l
	\begin{pmatrix} k \\ j \end{pmatrix} \begin{pmatrix} j \\ m \end{pmatrix}
	\d_y^{j-m} (V_2(y)e^{\mu y}f_3^{(2)}(y;\xi))\\ \times \int_x^y E_{\xi}(z-x) (\d_z^{k-j}(e^{-\mu z} A(z)))  e^{(y-z)(i\xi-\mu)} \,dz .
\end{multline*}
By Proposition \ref{P:f3 plus} and by assumption, we have the estimate
\[
	|\d_x^k (V_2(x)e^{\mu x}f_3^{(2)}(x;\xi))| \lesssim_k e^{-\gamma x}.
\]
Recalling Lemma~\ref{L:f12 lem15}, we arrive at the estimate
\[
	|K_2^{(k,j,m)}(x,y)| \lesssim_{k} e^{-\gamma y} \int_x^y \mu^{-1} (z-x) \mu^{-1} e^{-2\gamma z} e^{-\mu(y-z)} \,dz
	\lesssim \mu^{-2} e^{-\gamma y}.
\]
Combining the above identities, one sees that $\d_x^k v(x;\xi)$ solves a Volterra-type equation, and hence we obtain the following estimate by standard arguments: 
\[
|\d_x^k v(x;\xi) | \lesssim_k \mu^{-1} e^{-\gamma x}
\]
for $x\geq x_0$.  The case $l \geq 1$ is handled similarly.

We next estimate $u$ (given by $v$ as above) and complete the proof of Proposition~\ref{P:f12 plus}.  We rely on the following lemma, which follows from direct computation.

\begin{lemma}\label{L:f12 lem2}
Let $\mu,\gamma>0$ and $x_0\in\R$. For $n,m\in \mathbb{N}\cup\{0\}$, we have 
\[
	\int_{x}^\I (z-x)^n e^{-\mu(z-x)} e^{-\gamma z} \,dz
	\lesssim_n  \mu^{-n-1}  e^{-\gamma x}
\]
for $x\in\R$, and
\begin{align*}
	\int_{x_0}^x& (x-y)^n e^{\mu(y-x)} \int_{y}^\I (z-y)^{m} e^{-\mu (z-y)} e^{-\gamma z} \,dz \,dy\\
	&\lesssim_{m,n,\gamma,x_0} \mu^{-(m+n)-2} \Jbr{x}^{n+1} e^{-x \min(\mu,\gamma)}
\end{align*}
for $x\geq \max(0,x_0)$.
\end{lemma}

With all of the estimates above in place, we can turn to the proof of Proposition~\ref{P:f12 plus}. 

\begin{proof}[Proof of Proposition \ref{P:f12 plus}]
We define $x_0$ as in \eqref{E:def of x0}.
Then, one obtains $u$ by the formula \eqref{E:sys for uv35}.  Furthermore, by the estimates on $v$ and by Lemma~\ref{L:f12 lem2},
\begin{align*}
	|f_3(x;\xi) u(x;\xi)| 
	&{}\lesssim \int_{x_0}^x e^{\mu (y-x)} \int_y^\I e^{-\mu ( z-y)} e^{- \gamma z} \,dz \,dy\\
	&{}\lesssim \mu^{-2} \Jbr{x} e^{-x \min (\gamma,\mu) }. %\to 0
\end{align*}
%as $x\to\I$.

Thus, the pair $(u,v)$ gives us the solution $f_1(x;\xi)$ which satisfies
\[
	\abs{ f_1(x;\xi) - \begin{pmatrix} e^{ix\xi} \\ 0 \end{pmatrix} } \lesssim \mu^{-1} \Jbr{x} e^{-\min(\gamma,\mu) x}
\]
for $x>x_0(\xi)$. 
For $\xi$ such that $x_0(\xi)>0$, we extend the solution to the region $0< x< x_0(\xi)$ by the general existence theorem.

We turn to the estimate of derivatives of the solution.  In particular, we need to estimate derivatives of $u$.

We introduce
\begin{align*}
	\tilde{u}(x)
	&{}:=e^{-(i\xi + \mu)x} u(x) \\
	&{}= 
	\int_{x_0}^x 2 (e^{\mu y} f_3^{(2)}(y;\xi))^{-2} e^{(y-x)(i\xi +\mu)} \\
	&\qquad \qquad \times \int_y^\I e^{(z-y)(i\xi-\mu)} V_2(z)(e^{\mu z}f_3^{(2)}(z;\xi))  v(z;\xi) \,dz \,dy 
\end{align*}
We may also write
\begin{align*}
	\tilde{u}(x)
	&{}= 
	\int^{x-x_0}_0  e^{-w(i\xi +\mu)} H(x-w) \,dw ,
\end{align*}
where
\[
	H(z) := 2 (e^{\mu z} f_3^{(2)}(z;\xi))^{-2} \int_{0}^\I e^{a(i\xi-\mu)} V_2(a+z)(e^{\mu (a+z)}f_3^{(2)}(a+z;\xi))  v(a+z;\xi) \,da .
\]
This expression yields
\begin{align*}
	\d_x^{k} \tilde{u}(x)
	&{}= \sum_{j=0}^{k-1} H^{(j)}(x_0) e^{x_0(i\xi + \mu)} (-(i\xi+\mu))^{k-1-j} e^{-x(i\xi + \mu)}\\
	&\quad + \int_0^{x-x_0} H^{(k)}(x-w) e^{-w(i\xi +\mu)} \,dw\\
	&{}=: I_1 + I_2
\end{align*}
for $k\geq1$. 

We now observe that the derivative of the first term is bounded by
\[
	|\d_\xi^l I_1(x,\xi)| \leq C_{k,l}\Jbr{x}^l \mu^{k-1} e^{x_0 \mu} e^{-x\mu} \lesssim_{k,R} e^{-\frac12 \mu} \Jbr{x}^l e^{-x\mu}
\]
for $|\xi|\geq 2R$ and $x>0$, where we have used $x_0=-1$. For $|\xi| \leq 2R$ and $x \geq x_0(\xi)$, we use boundedness $|\xi| \leq 2R$ and $x_0 \leq \tilde{x}$ to obtain the same conclusion.

Let us next consider derivatives of $I_2$.  If any derivative lands on $e^{\mu x} f_3^{(2)}(x;\xi)$, then we gain $\mu^{-1}$, while if it lands on $v(x;\xi)$ or $e^{x(i\xi\pm \mu)}$ then we lose $\langle x\rangle$.  Hence, 
\begin{align*}
	|\d_\xi^l I_2(x)| \lesssim{}&
	\sum_{j=0}^l \int_{0}^{x-x_0} w^j e^{-\mu w} \int_{0}^\I (a+w)^{l-j} e^{-\mu a} e^{-\gamma (a+w)} da\, dw \\
	={}& \sum_{j=0}^l \int_{x_0}^x (x-y)^j e^{\mu(y-x)} \int_{y}^\I (z-y)^{l-j} e^{-\mu (z-y)} e^{-\gamma z} \,dz \,dy.
\end{align*}
Using Lemma \ref{L:f12 lem2}, we find
\[
	|\d_\xi^l  I_2(x)| \lesssim \mu^{-l-1} \Jbr{x}^{l+1} e^{-x \min (\gamma, \mu)}
\]
for $x\geq 0$. This yields the desired estimate for $|\xi| \geq 2R$ or $x \geq 0$.  Note that the above argument also works  for $k=0$ and $l\geq1$.

Finally, we consider the case $|\xi| \leq 2R$, $x_0(\xi) > 0$, and $x \in [0,x_0(\xi)]$.
Notice that this case both $x$ and $\xi$ is bounded.
Hence, it is sufficient to use the following estimate for a general solution:
\[
	|\d_\xi^l \d_x^k f_1(x;\xi)| \lesssim_{k,l,R} 1.
\]
This completes the proof.
\end{proof}

\subsection{Construction of $f_4$ for $x>0$}

Let us next construct the solution $f_4(x;\xi)$.  We first need to deal with the fact that the desired asymptotic behavior, namely
\[
	f_4(x;\xi) \to e^{\mu x} \begin{pmatrix} 0 \\ 1 \end{pmatrix} \qtq{as}x\to\infty,
\]
does not uniquely specify the solution.  Indeed, if one finds one such solution then any linear combination of the solution and the solutions $f_1,f_2$, and $f_3$ satisfies the same asymptotics.

Following \cite{KS}, we will look for $f_4$ obeying the Wronskian condition
\[
	W[f_1,f_4] = W[f_2,f_4] = 0.
\]
As we already have
\[
	W[f_1,f_3] = W[f_2,f_3] = 0,
\]
the above condition determines the `coefficients' of $f_1$ and $f_2$.  We remark, however, that we still have the freedom to add a scalar multiple of $f_3$.

To obtain a solution, we solve the following integral equation:
\begin{equation}\label{E:f4eq}
\begin{aligned}
	f_4(x;\xi) ={}& e^{\mu x} \begin{pmatrix} 0 \\ 1 \end{pmatrix} + \int_x^\I 
	\begin{pmatrix} 0 & 0 \\ 0 & -\frac{1}\mu e^{\mu (x-y)} \end{pmatrix} \mathcal{V}(y) f_4(y;\xi) \,dy \\
	&{}+	\int_{x_1}^x 
	\begin{pmatrix} \frac{2\sin \xi(x-y)}{\xi} & 0 \\ 0 & -\frac{1}\mu e^{-\mu (x-y)} \end{pmatrix} \mathcal{V}(y) f_4(y;\xi) \,dy
\end{aligned}
\end{equation}
with a suitable choice of $x_1=x_1(\xi) \in \R$.
We denote the particular solution we construct by $\tilde{f}_4(x;\xi)$.  Similar to the construction of $f_1$ and $f_2$, we will choose \emph{negative} $x_1$ for large $\xi$.  More precisely, we define $x_1=x_1(\xi) \in C^\I(\R)$, with constants $\tilde{x} \geq 0 $ and $R=R(\norm{\mathcal{V}}_{L^1(\R)},\omega,\gamma)> 0$ to be chosen later, to be an even function satisfying
\begin{equation}\label{E:defx1}
	x_1(\xi) = 
	\begin{cases}
	\tilde{x} & |\xi| \leq R\\
	-1 & |\xi| \geq 2R
	\end{cases}
\end{equation}
and $-1 \leq x_1(\xi) \leq \tilde{x}$ for any $\xi \in \R$.

\begin{proposition}[Construction of $\tilde{f}_4(x;\xi)$ for $x>0$]\label{P:f4 plus}
There exists a bounded function $x_1 = x_1(\xi)$ such that \eqref{E:f4eq} admits a unique solution $\tilde{f}_4(x;\xi)$
obeying 
\[
\tilde{f}_4(x;\xi) \leq 2 e^{\mu x}
\]
for $x\geq x_1$ and
\[
	\abs{ e^{-\mu x } \tilde{f}_4^{(j)}(x;\xi) - \delta_{2j} } 
	\lesssim_{\omega} \mu^{-(3-j)} \Jbr{x}^2 e^{-x \min (\gamma ,\mu)} 
\]
for $j=1,2$, $x \geq x_1$, and $\xi \in \R$.  Here $\delta_{ij}$ denotes the Kronecker delta.  Moreover, we have the estimate
\[
	\abs{ \d_x^k (e^{-\mu x } \tilde{f}_4^{(j)}(x;\xi)) } \lesssim_{k,\omega} \mu^{-(3-j)} \Jbr{x}^{2} e^{- x \min (\gamma ,\mu)}
\]
for $x \geq x_1$ and for $k\geq 1$.  The implicit constants in the above inequalities depend continuously on $\omega$.
\end{proposition}

As before, our choice of $x_1$ affords us some slight improvement compared to the estimates appearing in \cite{KS}.  In particular, the power of $\mu$ does not grow with $k$, and the estimate for the first component is improved.

\begin{proof}[Proof of Proposition~\ref{P:f4 plus}] Writing $g(x;\xi) := e^{-\mu x} f_4(x;\xi)$, we arrive at the equation
\begin{align*}
	g(x;\xi) ={}& \begin{pmatrix} 0 \\ 1 \end{pmatrix} + \int_x^\I 
	\begin{pmatrix} 0 & 0 \\ 0 &- \frac{1}\mu \end{pmatrix} \mathcal{V}(y) g(y;\xi) \,dy \\
	&{}+	\int_{x_1}^x 
	\begin{pmatrix} \frac{2\sin \xi(x-y)}{\xi}e^{-\mu (x-y)} & 0 \\ 0 &- \frac{1}\mu e^{-2\mu (x-y)} \end{pmatrix} \mathcal{V}(y) g(y;\xi) \,dy .
\end{align*}
Denoting the right hand side by $\Psi[g](x)$, we will show that $\Psi$ is a contraction map on the complete metric space $(X,d)$ defined by
\begin{align*}
	X &{}:= \{ g \in C([x_1,\I); \C^2 ) : \norm{g}_{L^\I([x_1,\I))} \leq 2 \},\\
	d(g_1,g_2) &{}:= \norm{g_1-g_2}_{L^\I([x_1,\I))}.
\end{align*}
Note that the integrals have bounded kernels since
\[
	\abs{\frac{\sin \xi(x-y)}{\xi}e^{-\mu (x-y)} } \leq (x-y)e^{-\mu (x-y)} \leq \mu^{-1} \( \sup_{w \geq 0} w e^{-w}\)
\]
for $0 \leq y \leq x$.

Now let $g,g_1,g_2 \in X$. For $x\in (x_1,\I)$, we have
\begin{align*}
	|\Psi[g](x)| \le{}& 1 + \int_x^\I \mu^{-1} C |\mathcal{V}(y)| \,dy + \int_{x_1}^x \mu^{-1} C |\mathcal{V}(y)| \,dy\\
	\leq {}& 1 + C\mu^{-1} \norm{\mathcal{V}}_{L^1([x_1,\I))}
\end{align*}
and similarly
\[
	|\Psi[g_1](x)-\Psi[g_2](x)| \leq C\mu^{-1} \norm{\mathcal{V}}_{L^1([x_1,\I))} \norm{g_1 - g_2}_{L^\I ([x_1,\I))}.
\]

We now observe that there exists $\tilde{x} \geq 0 $ such that $\Psi$ is a contraction for all $\xi \in \R$ with the choice  $x_1=\tilde{x}$.  Moreover, there exists $R=R(\norm{\mathcal{V}}_{L^1(\R)})\geq 0$ such that if $|\xi| \geq R$ then $\Psi$ is a contraction for any $x_1 \in \R$. Without loss of generality, we may suppose that $R \geq \max(\omega,\gamma)$ so that $|\xi| \sim_\omega \mu \geq 2\gamma$ for $|\xi|\geq 2R$.  Now, we define $x_1(\xi)$ by \eqref{E:defx1} with this choice of $\tilde{x}$ and $R$.  We then obtain a unique solution $g \in X$ to the equation $g=\Psi[g]$.

We next estimate derivatives of $g$. We will show
\[
	\abs{ \d_x^k (e^{-\mu x } \tilde{f}_4(x;\xi)) } \lesssim_{k,\gamma} \mu^{-1} \Jbr{x}^{2} e^{- x \min (\gamma ,\mu)}
\]
for $x \geq x_1$ and for $k\geq 1$.  We first have
\begin{equation}\label{E:gkform1}
\begin{aligned}
	& g^{(k)}(x;\xi)\\
	&{} = \int_{0}^{x-x_1} 
	\begin{pmatrix} 2\cos (\xi w) e^{-\mu w}-\frac{2\mu \sin \xi w}{\xi}e^{-\mu w} & 0 \\ 0 & 2 e^{-2\mu w} \end{pmatrix} (\d_w^{k-1}(\mathcal{V} g))(x-w) dw \\
	&\quad {}+ \sum_{j=0}^{k-2}  \d_x^{k-1-j}
	\begin{pmatrix} \frac{2\sin \xi (x-x_1)}{\xi}e^{-\mu (x-x_1)} & 0 \\ 0 & -\frac{1}\mu e^{-2\mu (x-x_1)} \end{pmatrix} (\d_w^{j}(\mathcal{V} g))(x_1)
\end{aligned}
\end{equation}
for $k\geq1$, with the convention that the second term is zero when $k=1$.

Consider the case $|\xi|\geq 2R$.   We will show by induction that
\[
	|g^{(k)}(x)| \lesssim_k \mu^{-1} e^{-\gamma x}
\]
holds for $k\geq1$ and $x\geq x_1=-1$. The base case is immediate from \eqref{E:gkform1}. Indeed, since $\mu \geq 2\gamma$ holds in this case, we have
\[
	|g'(x)| \lesssim ( ({\bf 1}_{\{w\geq 0\}} e^{-\mu w})* ({\bf 1}_{\{w\geq -1\}} e^{-\gamma w}))(x)
	\leq 2\mu^{-1} e^{-\gamma x}
\]
for $x\geq -1$.  Now suppose that the above bound is true for $1 \leq k \leq k_0$.  Then we derive
\[
	\abs{\d_w^{k_0} (\mathcal{V} g(\cdot;\xi))(z) dw} \lesssim_{k_0,\gamma}  e^{-\gamma z}
\]
for $z \geq 0$.

Hence, as in the base case, we have
\[
	|g^{(k_0+1)}(x)| \lesssim ( ({\bf 1}_{\{w\geq 0\}} e^{-\mu w})* ({\bf 1}_{\{w\geq -1\}} e^{-\gamma w}))(x) + e^{-\frac12 \mu}e^{-\mu x} \lesssim \mu^{-1} e^{-\gamma x}
\]
from \eqref{E:gkform1}. Here the choice $x_1=-1$ is used for the estimate of the second term.

For $|\xi| \leq 2R$ and $x \geq x_1$, we have
\[
	|g^{(k)}(x)| \lesssim_k \Jbr{x}^{2} \min(1, e^{-x \min(\mu,\gamma) }) \lesssim_\gamma 1
\]
by a similar induction argument.  The difference is that we need to use the worse estimate
\[
	\abs{\begin{pmatrix} \cos (\xi w)e^{-\mu w} - \frac{\mu \sin \xi w}{\xi}e^{-\mu w} & 0 \\ 0 & -2 e^{-2\mu w} \end{pmatrix}}
	\lesssim_R \Jbr{w} e^{-\mu w}
\]
for $w\geq 0$. The base case $k=1$ follows from
\[
	|g'(x)| \lesssim ( ({\bf 1}_{\{w\geq 0\}} \Jbr{w} e^{-\mu w})* ({\bf 1}_{\{w\geq -1\}} e^{-\gamma w}))(x)
	\lesssim \Jbr{x}^2  e^{-x \min (\gamma,\mu) } \lesssim_\gamma 1.
\]
If the estimate is true up to $k_0$, we have
\[
	\abs{\d_w^{k_0} (\mathcal{V} g(\cdot;\xi))(z) dw} \lesssim_{k_0,\gamma}  e^{-\gamma z}
\]
for $z\geq 0$.
Hence,
\begin{align*}
	|g^{(k_0+1)}(x)| &{} \lesssim ( ({\bf 1}_{\{w\geq 0\}} \Jbr{w} e^{-\mu w})* ({\bf 1}_{\{w\geq -1\}} e^{-\gamma w}))(x) + e^{-\frac12 \mu}\Jbr{x} e^{-\mu x} \\
	&{}\lesssim \Jbr{x}^2 e^{- x \min(\gamma,\mu)}
\end{align*}
as desired.

Finally, we consider an improved estimate for the first component. We first observe that $g$ satisfies the ODE
\[
	g'' = -2\mu g' -(\xi^2 +2E)(\mathrm{Id}+\sigma_3) g + 2 \sigma_ 3\mathcal{V} g.
\]
Thus
\[
	|g_1^{(k)}| \leq \mu^{-2}|g^{(k+2)}| + 2\mu^{-1} |g^{(k+1)}| + 2 \mu^{-2} |\d_x^k(\mathcal{V} g)|
\]
for $k\geq 0$, where $g_1$ is the first component of $g$. 
Combining the above estimates, we complete the proof of Proposition~\ref{P:f4 plus}.
\end{proof}

\begin{remark}
We have similar estimates for $\xi$-derivatives. 
We omit the details and only remark that the equation for $\d_\xi^l g$ is of the same form as that for $g$.
As a result, we obtain 
\[
	|\d_\xi^l (e^{-\mu x} \tilde{f}_4(x;\xi))| \lesssim \mu^{-2} \Jbr{x}^l e^{-x \min (\gamma,\mu)}
\]
for $l\geq 1$, for instance. One also has estimates on $\d_x^k \d_\xi^l (e^{-\mu x} \tilde{f}_4(x;\xi))$ ($k,l\geq1$)
by proceeding as in the case $l=0$.\end{remark}

We conclude this section by demonstrating that $g-1$ decays as $x\to\I$.  As $g$ is a bounded solution, the equation for $g$ yields
\begin{align*}
	|g(x)-1| \lesssim{}& \int_x^\I \mu^{-1} |\mathcal{V}(y)| \,dy + \int_{x_1}^x \abs{\begin{pmatrix} \frac{2\sin \xi(x-y)}{\xi}e^{-\mu (x-y)} & 0 \\ 0 & -\frac{1}\mu e^{-2\mu (x-y)} \end{pmatrix}}|\mathcal{V}(y)| \,dy.
\end{align*}
It is straightforward to see that the first term is $O(\mu^{-1} e^{-\gamma x})$ for large $x$.
We estimate the second term as follows: 
\[
	\int_{x_1}^x (x-y) e^{-\mu (x-y)} e^{-\gamma y} \,dy
	= e^{- \gamma x} \int_{x_1}^x (x-y) e^{-(\mu-\gamma) (x-y)} \,dy,
\]
which is bounded by
\[
	e^{- \gamma x} \int_{0}^\I w e^{-(\mu/2) w} dw = C_\gamma \mu^{-2} e^{-\gamma x}
\]
if $\mu \geq 2\gamma$, and by
\[
	e^{-\gamma x} \max (1, e^{-(\mu-\gamma) (x-x_1)}  ) \int_{x_1}^x (x-y)  \,dy \lesssim_\gamma
	\Jbr{x}^2 e^{-x \min (\gamma, \mu)}
\]
if $\mu \leq 2\gamma$.
%
%{\color{red} 
%\begin{remark}
%In Krieger-Schlag, it seems that they only prove
%\[
%	\abs{Tg(x) - \begin{pmatrix}0\\1 \end{pmatrix}} = \abs{Tg(x)-T{\bf 0}} \lesssim \mu^{-1} e^{-\gamma x_1} 
%\]
%for $x\geq x_1$. The right hand side goes to zero only when we take the limit $x_1 \to\I$, not $x\to\I$.
%Again, the weight $\Jbr{x}^2$ in our upper bound seems almost necessary, in view of the case $\gamma=\mu$.
%We give another proof below, with which we may see this respect.
%To solve the equation, we merely need the uniform boundedness of the integral kernel in the second term of equation for $g$.
%
%The weight appears when we estimate the kernel in a different way.
%To see the solution decays in $x$,
%we need a decay estimate for the kernel.
%It seems that $\xi^{-1} |\sin \xi (x-y)| e^{-\mu(x-y)}\leq |x-y|e^{-\mu(x-y)} $ is the best way.
%The weight $\Jbr{x}^2$ comes from the integration of $|x-y|$ part.
%So, it seems necessary since it actually appears when $\mu=\gamma$.
%\end{remark}}

\subsection{Construction of solutions for $x<0$}

So far, we have constructed solutions $f_j(x;\xi)$ ($j=1,2,3,4$) to \eqref{E:ODEs} for $x>0$.
We now consider the continuation to $x<0$.  Note that we already have four linearly independent solutions $f_j(-x;\xi)$ ($j=1,2,3,4$) when $\xi\neq0$, and so the continued solution is given by a linear combination of these for $x<0$.  However, it it is not clear whether we have a uniform-in-$\xi$ bound on the coefficients, as $f_1$ and $f_2$ coincide when $\xi=0$.

In order to get a uniform-in-$\xi$ bound, we solve the integral equation
% In all cases,
% we solve the following equation
\begin{align*}
	f_j(x;\xi) ={}& D_\xi' (x) f_j(0;\xi) + D_\xi (x)  f_j'(0-;\xi) \\
	&{}-2 \int_x^0 D_\xi(y-x) \sigma_3 \mathcal{V}(y)  f_j(y;\xi) \, dy.
\end{align*}
for $x<0$.  We summarize the result as follows:

\begin{proposition}[Continuation of $f_j(x;\xi)$ ($j=1,2,3,4$) for $x<0$]\label{P:f3 minus}
For all $\xi \in \R$, the solution $f_j(x;\xi)$ ($j=1,2,3,4$) to \eqref{E:ODEs} given above
can be extended as solutions on the whole line $x\in \R$. 
The solution is smooth in $(x,\xi) \in (\R \setminus\{0\}) \times \R$ and
satisfies 
\[
	| \d_\xi^l \d_x^k f_j(x;\xi)| \lesssim \mu^{-1+k}\Jbr{x}^l e^{\mu |x|}
\]
for $x<0$ and $k,l \geq 0$. 
\end{proposition}

\begin{proof}[Sketch of proof]  One sees that $e^{-\mu|x|} f_j(x)$ solves a Volterra type integral equation on $x\in(-\I,0)$.
This can be solved by standard arguments.  For the data, we have
\[
	\d_x^k f_j (0-;\xi) =  O(\mu^{-1+k})
\]
for $k\geq 0$ and $j=1,2,3,4$.
Indeed,
thanks to estimates for solutions on $x>0$, we know
\[
	\d_x^k f_j (0+;\xi) =  O(\mu^{-1+k})
\]
for $k\geq 0$ and $j=1,2,3,4$. Then, by the jump condition we obtain the desired bound for $k=0,1$.
The differential form of the equation gives us the bound for $k\ge2$.
\end{proof}

%%%%%%%%%%%%%%%%%%%%%%%%%%%%%%%%%%%%%%
%%%%%%%%%%%%%%%%%%%%%%%%%%%%%%%%%%%%%%
%%%%%%%%%%%%%%%%%%%%%%%%%%%%%%%%%%%%%%
%%%%%%%%%%%%%%%%%%%%%%%%%%%%%%%%%%%%%%

\end{document}